\documentclass[11pt,leqno]{amsart}

\usepackage[english]{babel}
\usepackage[utf8]{inputenc}
\usepackage{amsmath}

\newcommand\numberthis{\addtocounter{equation}{1}\tag{\theequation}}
	\usepackage{subcaption} 
\usepackage{amsfonts}
\usepackage{amssymb}
\usepackage{tikz-cd}
\usepackage{mathrsfs}
\usepackage[toc,page]{appendix}
\usepackage{varwidth}
\usepackage{array}
\usepackage{tikz}
\usetikzlibrary{matrix}
\usepackage{bm}
\usepackage{graphicx}
\usepackage[colorinlistoftodos]{todonotes}
\usepackage{tensor}
\usepackage{hyperref}
\usepackage{enumerate}
\usepackage{mathtools}
\usepackage{braket}
\usepackage{bbm} 
\usepackage[numbers,square]{natbib}
 
\usepackage{amsthm}
\newtheorem{theorem}{Theorem} 
\newtheorem{proposition}[theorem]{Proposition}

\newtheorem{lemma}[theorem]{Lemma} 
\newtheorem{corollary}[theorem]{Corollary} 
\newtheorem{definition}[theorem]{Definition}
\newtheorem{remark}[theorem]{Remark}
\theoremstyle{definition}
\usepackage[small,nohug,heads=LaTeX]{diagrams}
\newtheorem{question}{Question}

\newtheorem{conjecture}[theorem]{Conjecture}

\makeatletter
\newcommand*{\transpose}{%
  {\mathpalette\@transpose{}}%
}
\newcommand*{\@transpose}[2]{%
  \raisebox{\depth}{$\m@th#1\intercal$}%
}

\makeatother
\newenvironment{proof-sketch}{\noindent{\it Sketch of Proof.}\hspace*{1em}}{\qed\bigskip}

\newcommand*\stab{\textrm{Stab}(\mathcal{T})}

\title{}
\author{Arkadij Bojko and George Dimitrov}
\date{}

\def\mk {\mathfrak}
\def\CC {{\mathbb C}}     
\def\KK {{\mathbb K}}     
\def\PP {{\mathbb P}}     
\def\RR {{\mathbb R}}     
\def\ZZ {{\mathbb Z}}     
\def \bd {\begin{diagram}}
	\def \ed {\end{diagram}}

\def \Aut {\mathrm{Aut}}

\def\Hom {\mathrm{Hom}}

\def\be  {\begin{eqnarray}}
\def\ee  {\end{eqnarray}}

\def\st {\mathrm{Stab}}

\def\mc {\mathcal}

\newcommand{\abs}[1]{\left\vert#1\right\vert}

\newcommand{\ri}{{\mathrm i}}

\begin{document}
	
	\renewcommand{\contentsname}{Contents}
	\renewcommand{\refname}{References}
	
	\title[Non-commutative counting and stability]{Non-commutative counting and stability}
	

	Address: Professur für Mathematik ETH Zürich, HG J 16.2 Rämistrasse 101, Zürich 8092, Switzerland , E-mail: arkadij.bojko@math.ethz.ch
	\address[Bojko]{ETH Zürich\\
	 Professur für Mathematik\\
	HG J 16.2 R\"amistrasse 101, Zürich\\Switzerland
	}
	\email{arkadij.bojko@math.ethz.ch}
	
		\address[Dimitrov]{Universit\"at Wien\\ Fakultät für Mathematik\\
			Oskar-Morgenstern-Platz 1, 1090 Wien\\
			\"Osterreich \\
		}
		\email{george.dimitrov@univie.ac.at}
	
	\renewcommand{\abstractname}{Abstract}
	
	\renewcommand{\figurename}{Figure}

\begin{abstract}
The second author and Katzarkov introduced categorical invariants based on counting of full  triangulated subcategories in a given triangulated category $\mathcal T$, and they demonstrated   different choices of additional  properties  of the subcategories being counted, in particular  -  an approach to
make non-commutative counting in  $\mathcal T$ dependable on a stability condition $\sigma \in {\rm Stab}(\mathcal T)$.   In this paper, we focus on this approach.   After recalling the definitions  of a \textit{stable non-commutative curve in $\mathcal T$} and related notions, we prove a few general facts and study an example: $\mathcal T = D^b(Q)$, where $Q$ is the acyclic triangular quiver.  In previous papers, it was shown that there are two non-commutative curves of non-commutative genus $1$ and infinitely many non-commutative curves of non-commutative genus $0$ in  $D^b(Q)$. Our studies here imply that  for an open and dense subset in ${\rm Stab}(D^b(Q))$ the stable non-commutative curves in $D^b(Q)$ are finite.   This paper also introduces counting of semistable derived points and shows that the corresponding invariants are finite on an open dense subset of $\st\big(D^b(Q)\big)$. 
\end{abstract}

\maketitle
\setcounter{tocdepth}{1}
\tableofcontents

\section{Introduction}
\subsection{Overview} \label{setup and results}
Orlov's reconstruction theorem states that for a smooth variety with an ample (anti-)canonical divisor, one can uniquely recover the variety from its bounded derived category. This motivated the point of view of non-commutative derived geometry through considering general triangulated categories as non-commutative varieties. An Example of this appeared in Kontsevich--Rosenberg \cite[§3.3]{KR}, where it is conjectured that the derived category of the non-commutative projective space $N\mathbb{P}^{l-1}$ is the derived category of the wild Kronecker quiver with $l$ arrows:
$$
K(l)=
\begin{tikzcd}
    1
    \arrow[r, draw=none, "\raisebox{+1.5ex}{\vdots}" description]
    \arrow[r, bend left]
    \arrow[r, bend right, swap, "l\,\textnormal{- times}"]
    &
    2
\end{tikzcd}
$$
Following this, the second author and Katzarkov defined non-commutative counting in  \cite{DK1} as taking cardinalities of certain sets of fully faithful embeddings of a fixed triangulated subcategory $\mathcal{A}$ into a larger one $\mathcal{T}$. We will focus here on the cases $\mathcal{A}=N\mathbb{P}^l$ or  $\mathcal{A}=D^b(A_1)$. The former will be called a \textit{non-commutative curve} and the latter a \textit{derived point}. One can additionally introduce a dependence on a parameter $\sigma \in \st(\mathcal{T})$, where $\st(\mathcal{T})$ is the manifold of stability conditions of Bridgeland \cite{Bridgeland}. This leads to counting of stable non-commutative curves and derived points as suggested in \cite[Section 12.4]{DK1}.

Our main goal is to study this for non-trivial example of $\mathcal{T}=D^b(Q)$. where $Q$ is the quiver 
\begin{equation} \label{triangular quiver}
Q=\begin{diagram}[1em]
&       &  z  &       &    \\
& \ruTo &    & \rdTo &       \\
x  & \rTo  &    &       &  y
\end{diagram} .
\end{equation}
Knowing which derived points are semistable for each stability condition, we can determine the stability on non-commutative curves. It then turns out that to obtain a complete answer, we only need to study the behavior of stable exceptional objects in detail. 

We thus describe explicitly the wall and chamber structure of the space $\st(\mathcal{T})$, where non-commutative curves or derived points become stable. The most important observation is that in our example, the set of $\sigma\in \st(\mathcal{T})$, such that the number of semistable curves is finite is dense. This motivates Conjecture \ref{conjecture1}. As this does not apply verbatim to derived points, we propose an alternative: we count stable derived points lying outside of stable non-commutative curves. Then an analogous property is satisfied (see Conjecture \ref{conjecture derived points}, Theorem \ref{theorem derived points}). We now give a more detailed summary of the definitions, previous work in this area and our new results.

\subsection{Summary of the results}
 The  idea of  non-commutative counting in  \cite{DK1} is to count elements in a set $C_{\mc A, P}^{\Gamma}(\mc T)$ defined as follows:   given two  triangulated  categories $ \mc A $, $\mc T$, a subgroup $\Gamma \subset {\rm Aut}(\mc T)$, and  choosing  some additional restrictions $P$ of fully faithful exact functors,  the set  $C_{\mc A, P}^{\Gamma}(\mc T)$ consists of equivalence classes of fully faithful exact functors from $\mc A$   to $\mc T$  satisfying $P$ 
 with two functors  $F, F'$  being equivalent iff  $F\circ \alpha \cong \beta \circ F'$ for some $\alpha \in {\rm Aut}(\mc A)$, $\beta \in \Gamma$ (see  Definition \ref{nccurves}).
  Various  choices of  $\mc A, \Gamma, P, \mc T$ have been shown to  result in finite sets   $C_{\mc A, P}^{\Gamma}(\mc T)$  in   \cite{DK1, dimkatz2, DKlatest}. In this paper, we let  $P$ depend on $\sigma \in \st(\mathcal{T})$ and  demonstrate an example with fixed $\mc T$ and $\mc A$, in which  for some points $\sigma$ in the parameter space the  set  $C_{\mc A, P(\sigma)}^{\{\rm Id\}}(\mc T)$  is an infinite set which was already studied in \cite{dimkatz2}. However, for a  generic parameter $\sigma$ the set $C_{\mc A, P(\sigma)}^{\{\rm Id\}}(\mc T)$   is finite.   
 
 When the categories $\mc T$, $\mc A$ are $\KK$-linear, then choosing the following  property $P$ imposed on the functors:  \textit{to be $\KK$-linear}, and restricting  $\Gamma$ to be a subgroup of the group of $\KK$-linear auto-equivalences  $ {\rm Aut}_{\KK}(\mc T)$ results in  the set   $C_{\mc A, P}^{\Gamma}(\mc T)$, which is denoted by  $C_{\mc A, \KK}^{\Gamma}(\mc T)$ .

 As already mentioned, we denote  $D^b(\textnormal{Rep}_\KK(K(l+1)))$   by  $N\PP^l$  for $l\geq -1$ (see Section \ref{notations}  for basic notations). In \cite{DK1},  $N\PP^l$ is selected  as  a first   choice  for  $\mc A$,  being   non-trivial but a well studied category.  For a $\KK$-linear category $\mc T$ and any $l\geq -1$   the set $C_{N\PP^l, \KK}^\Gamma(\mc T)$ is denoted by  $C_{l}^{\Gamma}(\mc T)$ and is referred to   as \textit{the set of non-commutative curves of non-commutative genus $l$ in $\mc T$ modulo  $\Gamma$}.  
 On the other hand, in \cite{dimkatz2} the  theory of  intersection of the entities being counted  is initiated and,   in that respect, the elements in  $C^{\{\rm Id\}}_{D^b(pt), \KK}(\mc T)$  are referred   to as \textit{derived points in $\mc T$}. 

 A derived point in  a proper $\KK$-linear $\mc T$ with a dg-enhancement is just a full triangulated  subcategory generated by an exceptional object, whereas a non-commutative curve of non-commutative  genus $l$ is the same as  a full triangulated  subcategory generated by a strong exceptional pair $(A,B)$ with $\hom(A,B)=l+1$ (see \cite[Proposition 5.5]{dimkatz2}). 
 
For the rest of the introduction, we let $\mc T$ be as in the previous paragraph .  Let $\Gamma\subset \Aut(\mc T)$ be the  trivial subgroup $\Gamma=\{\rm Id\}$ and we will omit writing $\{\rm Id\}$. We will now explain $P(\sigma)$ (see also Definition \ref{main def} and Remark \ref{remark for general A}).

The notion of stability condition on a triangular category $\mc T$, introduced by T. Bridgeland in \cite{Bridgeland}, has its roots in string theory and mirror symmetry. In \cite{Bridgeland}, the set of so called locally finite stability conditions is equipped with a  structure of a complex manifold, denoted by $\st(\mc T)$.   Each $\sigma \in \st(\mc T)$  distinguishes a set of non-zero objects in $\mc T$ and labels them with real numbers,  these objects are referred to as \textit{$\sigma$-semistable objects} and the real number attached to such an object $X$ is denoted by $\phi_{\sigma}(X)\in \RR$ and called the  phase of $X$.  We denote the  set of semistable objects  by $\sigma^{ss}$.  Since the set $\sigma^{ss}$ is invariant under the translation functor, it follows that for a derived point $\mc A \in C_{D^b(pt), \KK}(\mc T)$ either no or all exceptional  objects in $\mc A$ are also in $\sigma^{ss}$, thus $\sigma^{ss}$ naturally distinguishes those derived points satisfying the latter. The set of derived  points, distinguished by $\sigma$, will be  denoted by $C_{D^b(pt),\sigma \sigma}(\mc T)\subset C_{D^b(pt), \KK}(\mc T)$  (see Definition \ref{def of stable derived points}). 

We will be mainly interested in two subsets of $C_{l}(\mc T)$, denoted by $C_{l, \sigma \sigma}(\mc T)$ and $C_{l, \sigma}(\mc T)$, whose elements are refered to as   $\sigma$-stable and, respectively $\sigma$-semistable non-commutative curves, as defined in  \cite[Section 12.4]{DK1}. \textbf{A non-commutative curve $\mc A \in C_l(\mc T)$ is in $C_{l,\sigma}(\mc T)$ if and only if  infinitely many derived points in it are also in $C_{D^b(pt),\sigma \sigma}(\mc T)$, and it is in  $C_{l,\sigma \sigma}(\mc T)$ if and only if all derived points in $\mc A$ are also in $C_{D^b(pt),\sigma \sigma}(\mc T)$ (see Lemma \ref{sigmasemistablesubcategory}).} Details on the definitions and basic properties are  given in Section \ref{noncommutativecurves}.  In particular, for $l\leq 0$ we obtain $C_{l,\sigma}(\mc T) = \emptyset$,  whereas for $l\geq 1$  $C_{l,\sigma\sigma}(\mc T) \subset C_{l,\sigma}(\mc T) \subset C_{l}(\mc T)$ (see Remark \ref{remark for helix}). For any $\mc A \in C_{l}(\mc T)$	the set $\{\sigma \in \st(\mc T):  \mc A\in C_{l,\sigma \sigma}(\mc T)\}$ is a closed subset in $\st(\mc T)$ (see Corollary \ref{closed subset}), which makes it the correct object to study. Because the set $C_l(\mathcal{T})$ can often become infinite, the dependence on stability conditions is introduced to cut down to a finite number of non-commutative (semi-)stable curves and define integer invariants counting these. This is a usual approach in enumerative theories (e.g. Donaldson--Thomas invariants  and Gromov--Witten invariants) where one restricts to moduli spaces of stable objects carrying well-defined virtual fundamental classes. In this paper, we also give an example, where $\textnormal{stable}\neq \textnormal{semistable}$ (see Proposition \ref{no semistable1}) showing a further similarity to other enumerative theories. This naturally motivates the following question.
\begin{question}
\label{question}
Let $\mathcal{T}$ be a triangulated $\KK$ category with a phase gap (\cite[Definition 4.6]{DK1}) and $\textnormal{rk}\big(K^0(\mathcal{T})\big)<\infty$. Does there exist a dense open subset of $\Sigma\subset \stab$, such that $|C^\Gamma_{l,\sigma\sigma}(\mathcal{T})|<\infty$ for all $\sigma\in\Sigma$ and $\Gamma$ a subgroup of $\textnormal{Aut}(\mathcal{T})$?
\end{question}

We now formulate a conjecture for a more approachable situation. 
 In \cite[12.1]{DK1} the first author and Katzarkov define a new invariant
$$
\left\{\begin{array}{c}\textnormal{triangulated}\\
\textnormal{categories}\\
\textnormal{with a phase gap}\end{array}\right\}\xrightarrow{\textnormal{dim}_{nc}}[0,+\infty]\,,
$$
which gives the following vanishing criterion:
$$
\textnormal{dim}_{nc}(\mathcal{T})\leq n\implies C_l(\mathcal{T}) = \emptyset \textnormal{ for } l>n\,.
$$
They show $\textnormal{dim}_{nc}\big(D^b(Q)\big)=1$ for an acyclic quiver $Q$  if and only if it is affine. In this case, we expect the following to hold.

\begin{conjecture}
\label{conjecture1}
Let $Q$ be an extended acyclic Dynking quiver and $\mathcal{T} = D^b(Q)$, then the set of stability conditions $\sigma\in \stab$ such that $|C_{l,\sigma\sigma}|<\infty$ is open and dense. 
\end{conjecture}
This paper contains  an  example of such $Q$, where $C_{0}(\mc T)$ is infinite however the set $\{ \sigma \in \st(\mc T) : \abs{\cup_{l\geq -1}C_{l,\sigma \sigma}(\mc T)} < \infty \}$ is open and dense subset in $\st(\mc T)$.   Furthermore, counting the  elements of $C_{0,\sigma \sigma}(\mc T)$ as $\sigma$ varies in this open and dense subset results in all the numbers in $\{1\} \cup 2 \ZZ_{\geq 0}$. The same approach fails when applied to derived points as shown in Proposition \ref{some equal subsets}. 
This motivates us to introduce a new approach to counting semistable exceptional objects in terms of counting derived points. In Definition \ref{definition derived points outside of curves}, we define the set $\mathfrak{D}_{\sigma}\subset C_{D^b(pt),\sigma\sigma}$ of \textit{stable derived points outside of stable curves of positive genus}. This set consists of all elements $p\in C_{D^b(pt),\sigma\sigma}$ which are not a derived point of a curve $c\in C_{l,\sigma\sigma}$. 

We can pose a question analogous to the Question \ref{question}. Moreover, Theorem \ref{theorem derived points} suggests the following.

\begin{conjecture}
    \label{conjecture derived points}
   Let $Q$ be an extended acyclic Dynking quiver and $\mathcal{T} = D^b(Q)$, then the set of stability conditions $\sigma\in \stab$ such that $|\mathfrak{D}_\sigma|<\infty$ is open and dense.  
\end{conjecture}

\subsection{Some general results and a complete example} 

Before focusing on our main example,  we show in Section \ref{stabilities without stable curves}  general situations ensuring non-trivial behavior of $C_{l,\sigma \sigma}(\mc T)$,  $C_{D^b(pt),\sigma \sigma}(\mc T)$ as $\sigma$ varies in $\st(\mc T)$.  To that end we   recall the notion of $\sigma$-exceptional collection, introduced in \cite{dimkatz3}. The following then summarizes multiple results.
\begin{proposition}[Cor. \ref{nonentangled non semistable} and Prop. \ref{stable curve part 1}]
Let $\mathcal{T}$ be a triangulated category with an enhancement and $\sigma\in  \st(\mathcal{T})$. Let $\mathcal {E}  = (E_0,\ldots, E_n)$ be a full  $\sigma$-exceptional collection such that
\begin{enumerate}[(i)]
\item
  $\phi(E_{i-1})<\phi(E_{i})$ for all $i=1,\ldots ,n$, then 
$C_ {D^b(pt),\sigma\sigma}= \{\langle E_0\rangle,  \langle E_1\rangle ,\ldots ,\langle E_n\rangle\}$. In particular, $C_{l.\sigma\sigma}(\mathcal{T}) = C_{l,\sigma}(\mathcal{T})= \emptyset$ for  $l\geq 0$.
\item such that for some  $0\leq i<n$ and $l\geq 0$, we have $\textnormal{hom}^1(E_i,E_{i+1})=l+1$, $\textnormal{hom}^j(E_i,E_{i+1})=0$ for $j \in \mathbb{Z}\backslash \{1\}$ and  $\phi_{\sigma}(E_i)<\phi_{\sigma}(E_{i+1})$ then   $\langle E_i,E_{i+1}\rangle \in C_{l,\sigma \sigma}(\mathcal{T})$.
\end{enumerate}
\end{proposition}
In Corollary  \ref{stable curve} this result is presented in a slightly different form.  In addition to this. To study our main example, we will use in this paper more subtle tools from \cite{dimkatz3, dimkatz4} to ensure that certain elements in $C_l(\mc T)$ are in $C_{l,\sigma \sigma}(\mc T)$ or in $C_{l,\sigma }(\mc T)$ for some $\sigma \in \st(\mc T)$. 

In \cite[Proposition 12.13]{DK1}  the behavior of $C_{l,\sigma}(N\PP^k)$ and $C_{l,\sigma \sigma}(N\PP^k)$ was described for any  $l\geq -1$, $\sigma \in \st(N\PP^k)$, $k\geq 1$. In this case   $C_{l}(N\PP^k)$ is non-trivial only for $l=k$, $C_{k}(N\PP^k)$ has only one element and it turns out that  $C_{k,\sigma}(N\PP^k)= C_{k,\sigma \sigma}(N\PP^k)$ for each $\sigma$. We see therefore that Conjecture  \ref{conjecture1} and \ref{conjecture derived points} are trivially satisfied. 

Here   we focus on another  example,    namely From now on, $\mathcal{T}=D^b(Q)$, where $Q$ is as in \eqref{triangular quiver}. Let us recall some results from \cite{dimkatz2, dimkatz3, dimkatz4}.  The derived points in $\mc T$ are (\cite[Subsection 3.2, Remark 3.6]{dimkatz4}):
 \begin{gather} \label{derived points}
 C_{D^b(pt), \KK}(\mc T)=\{\langle a^m \rangle, \langle b^m \rangle : m \in \ZZ \} \cup \{\langle M\rangle, \langle M'\rangle \},
 \end{gather}
 where $a^m$, $b^m$, $M$, $M'$ are exceptional objects, $M$, $M'$ is the only couple of exceptional objects in $\textnormal{Rep}_{\KK}(Q)$ such that ${\rm Ext}^1(M,M')\neq 0$,  ${\rm Ext}^1(M',M)\neq 0$ (see the  beginning of \cite[Subsection 2.3]{dimkatz3}). 
 \begin{proposition}\cite[Proposition 8.1, (34), (35), (36), (37))]{dimkatz2}.
 	\label{noncomcurves} For $l\not \in \{ 0,1\}$ we have $C_l(\mc T) = \emptyset$ and:
 		\begin{gather}  C_1(\mc T)=\{\langle a^m,a^{m+1}\rangle = \langle M\rangle^{\perp}, \langle b^m,b^{m+1}\rangle= \langle M'\rangle^{\perp}\} \\ \label{C_0^{{rm Id}}(mc T_1)}  C_0(\mc T) =\{ \langle M',a^{m}, b^{m+1}\rangle=  \langle a^{m+1}\rangle^{\perp}:m\in \ZZ \}\cup\{\langle M,b^{m}, a^{m}\rangle = \langle b^{m+1}\rangle^{\perp}:m\in \ZZ\} . \end{gather}

 	We will denote  $A=\langle a^m,a^{m+1}\rangle $, $B=\langle b^m,b^{m+1}\rangle$, and thus  $C_1(\mc T)=\{ A,B\}$.\footnote{here $(a^m, a^{m+1})$, $(b^m, b^{m+1})$ are exceptional pairs}

 		Denoting  $\alpha^m=\langle M',a^{m}, b^{m+1}\rangle $, $\beta^m=\langle M,b^{m}, a^{m}\rangle$ we can write  $C_0(\mc T)=\{ \alpha^m, \beta^m: m \in \ZZ\}$. 
 	\end{proposition}

 In \cite[Proposition 12.14]{DK1} it was suggested and in  \cite{Bojko}  proved  that  when $\sigma$ varies in $\st(\mc T)$ the  set $C_{1,\sigma}(\mc T)$ takes all  subsets in $C_1(\mc T)$. However,  only $C_{0,\sigma \sigma}(\mc T)$ is meaningful, since we already mentioned that $C_{0,\sigma}(\mc T)= \emptyset$ for any $\sigma$ and any $\mc T$. Additionally, for any $l \geq -1$ we expect that  studying $C_{l,\sigma \sigma}(\mc T)$ is more meaningful  than $C_{l,\sigma }(\mc T)$, for example  Corollary \ref{closed subset} does not seem to hold for $C_{l,\sigma }(\mc T)$. These are  reasons to suggest in \cite[Remark 12.15]{DK1}  that not only   $C_{1,\sigma}(\mc T)$ but also  $C_{1,\sigma \sigma}(\mc T) \subset C_{1,\sigma}(\mc T)$ takes all possible subsets in $C_1(\mc T)$.

 Before giving  detailed description of the functions   $$\st(\mc T) \ni \sigma \mapsto  \{C_{l,\sigma}(\mc T), C_{l,\sigma \sigma}(\mc T), C_{D^b(pt), \sigma \sigma}(\mc T)\}$$ for $l=0,1$ we present some features  of the full picture.   The following proposition captures the non-trivial behaviour of the invariants leading to wall-crossing.
 \begin{proposition} \label{positive answer}	As  $\sigma$ varies in $\st(\mc T)$:
 	
 {\rm	(a)} the subset  $C_{1,\sigma \sigma}(\mc T)\subset C_{1}(\mc T)$ takes all possible subsets of $C_{1}(\mc T)$; 
 	
 {\rm	(b) } the subset $C_{0,\sigma \sigma}(\mc T)\subset C_{0}(\mc T)$ is one of the following subsets and realizes each of them:   	
  $\emptyset$,  	$\{\alpha^i\}$, $\{\beta^i\}$, 	$\{\alpha^j, \beta^j: i\leq j \leq k \}$, 	$\{\alpha^j,  \beta^{j+1}: i\leq j \leq k \}$, $\{\alpha^j, \beta^j: i\leq j  \}$, 	$\{\alpha^j,  \beta^{j+1}: i\leq j  \}$, $\{\alpha^j, \beta^j:  j \leq k \}$, 	$\{\alpha^j,  \beta^{j+1}:  j \leq k \}$, where $i,k$ are arbitrary integers such that $i\leq k$.   In particular $\left \{\abs{C_{0, \sigma \sigma}(\mc T)}: \sigma \in \st(\mc T)\right  \}=\{1,+\infty \}\cup 2 \ZZ_{\geq 0}$.
 \end{proposition}

Furthermore, we obtain a positive confirmation of Conjecture \ref{conjecture1}.

 \begin{theorem} \label{main theorem}  The set $\{\sigma \in \st(\mc T): \abs{C_{0,\sigma \sigma}(\mc T)}<\infty \}$ is an open and dense subset in $\st(\mc T)$.
 \end{theorem}
 
 One could investigate in more details the walls separating the connected components of the open subset from Theorem \ref{main theorem} and  to study  wall-crossing.  In section \ref{section walls and chambers}, we describe in Proposition \ref{proposition walls and chambers} the global picture of walls and chambers for $C_{1,\sigma\sigma}(\mathcal{T})$ by gluing the charts from tables  \eqref{table for b b M'}, \eqref{table for a a M},  \eqref{table for M b b},
\eqref{table for M' a a}, \eqref{table for b a b} and using the results from Proposition \ref{porposition for middle M}.

As already mentioned in §\ref{setup and results}, we can recover the above from studying $C_{D^b(pt), \sigma \sigma}(\mc T)$ for each $\sigma\in \st(\mc T)$. However, counting elements in $C_{D^b(pt), \sigma \sigma}(\mc T)$ is not finite  for generic $\sigma$. More precisely we have: 
 \begin{proposition} \label{some equal subsets}
 	The following subsets are equal and they are open but not dense subsets in $\st(\mc T)$:	\begin{gather}\left \{  \sigma :  \abs{C_{D^b(pt), \sigma \sigma}(\mc T)}<\infty \right \} =\left \{  \sigma :  \abs{C_{1, \sigma }(\mc T)}=0 \right \} = \left \{  \sigma :  \abs{C_{1, \sigma \sigma}(\mc T)}=0 \right \}. \nonumber  \end{gather} 
Furthermore, $\left \{\abs{C_{D^b(pt), \sigma \sigma}(\mc T)}: \sigma \in \st(\mc T) \right \}=\{3,4,5, +\infty\}$. 
\end{proposition}
  For $N\PP^l$ we have $C_{l,\sigma}(N\PP^l) =  C_{l,\sigma \sigma}(N\PP^l)$ for all $\sigma \in \st(\mc T)$ and all $l\geq 1$.  However,  for the case $\mc T=D^b(Q) $  we observe a new phenomenon:
 \begin{proposition} \label{no semistable1}
 	The set  $ C_{1,\sigma}(\mc T) \setminus  C_{1,\sigma \sigma}(\mc T)$ can be non-trivial.
 \end{proposition}
 Next, we obtain a positive result towards Conjecture \ref{conjecture derived points}.
 
 \begin{theorem}
 \label{theorem derived points}
 The set $\{\sigma\in \textnormal{Stab}(\mathcal{T}): |\mathfrak{D}_\sigma|<\infty\}$ is an open and dense subset in $\textnormal{Stab}(\mathcal{T}).$
 \end{theorem}
 
 Another result, following from  Section \ref{section for semistable curves}, is: 
\begin{proposition} \label{MeqMprime}  Let  $\sigma\in \st(\mc T)$, then $C_{1,\sigma \sigma}(\mc T)=C_{1}(\mc T)\iff M,M' \in \sigma^{ss}$, $\phi(M)=\phi(M')$.    Furthermore,   $C_{1,\sigma \sigma}(\mc T)=C_{1}(\mc T) \iff C_{0,\sigma \sigma}(\mc T)=C_{0}(\mc T) \iff C_{D^b(pt),\sigma \sigma}(\mc T)=C_{D^b(pt), \KK}(\mc T)$.
\end{proposition}
A complete summary of the results from which the above Theorems and Propositions were derived are collected in tables \eqref{table for b b M'}, \eqref{table for a a M}, \eqref{table for M b b}, \eqref{table for M' a a},  \eqref{table for b a b} and Proposition  \ref{sigma outside Tab} in the Appendix \ref{sectab}. We use there the notation set in §\ref{homeomoeprhism}. How these results glue to a global picture is described in §\ref{walls and chambers}.
\subsection*{Acknowledgements}
The authors are grateful to Ludmil Katzarkov for his interest and encouraging remarks. 
The first author was supported by Clarendon Fund Scholarship at University of Oxford.
The second author was
supported by FWF Project P 29178-N35.

\section{Notations} \label{notations} 

We fix a universe and assume that the  set of objects and  the set of morphisms of any category we consider are elements of this universe. 

The shift functor  in a triangulated category ${\mathcal T}$ is designated sometimes by $[1]$.

A \textit{triangulated subcategory}  in a triangulated category $\mc T$ is a non-empty full subcategory $\mc B $ in  $\mc T$, s. t. two conditions hold:

{\rm (a) } $Ob(\mc B)[1]=Ob(\mc B)$ ; 

{\rm (b)}  for any $X,Y \in Ob(\mc B) $ and any distinguished triangle $X\rightarrow Z \rightarrow Y \rightarrow X[1]$ in $\mc T$ the object $Z$ is also in $ Ob(\mc B)$.

We write $\langle  S \rangle  \subset \mc T$ for  the triangulated subcategory of $\mc T$ 
generated by $S$, when $S \subset Ob(\mc T)$. 
We write $\Hom^i(X,Y)$ for  $\Hom(X,Y[i])$.
In this paper $\KK$ denotes a field.  
If $\mc T$ is $\KK$-linear triangulated category  we write  $\hom^i(X,Y)$ for  $\dim_\KK(\Hom(X,Y[i]))$, where $X,Y\in \mc T$. 

A $\KK$-linear  triangulated category $\mc T$ is called  \textit{ proper} if $\sum_{i\in \ZZ} \hom^i(X,Y)<+\infty$ for any two objects $X,Y$ in $\mc T$.

 When we say that $\mc T$ \textit{has an enhancement} we mean  enhancement as explained in  \cite{Orlov}.

An \textit{exceptional object} in a $\KK$-linear triangulated category  is an object $E\in \mc T$ satisfying $\Hom^i(E,E)=0$ for $i\neq 0$ and  $\Hom(E,E)=\KK $. We denote  by ${\mc T}_{exc}$ the set of all exceptional objects in $\mc T$, 

We will often write \textit{ffe functor} instead of fully faithful exact functor in the sequel.

An \textit{exceptional collection} is a sequence $\mc E = (E_0,E_1,\dots,E_n)\subset \mc T_{exc}$ satisfying $\hom^*(E_i,E_j)=0$ for $i>j$.    If  in addition we have $\langle \mc E \rangle = \mc T$, then $\mc E$ will be called a full exceptional collection.  For a vector $\textbf{p}=(p_0,p_1,\dots,p_n)\in \ZZ^{n+1}$ we denote $\mc E[\textbf{p}]=(E_0[p_0], E_1[p_1],\dots, E_n[p_n])$. Obviously  $\mc E[\textbf{p}]$ is also an exceptional collection. The exceptional collections of the form   $\{\mc E[\textbf{p}]: \textbf{p} \in \ZZ^{n+1} \}$ will be said to be shifts of $\mc E$. 

If an exceptional collection  $\mc E = (E_0,E_1,\dots,E_n)\subset \mc T_{exc}$ satisfies  $\hom^k(E_i,E_j)=0$ for any $i,j$ and for $k\neq 0$, then it is said to be \textit{strong exceptional collection}.   

For two exceptional collections $\mc E_1$, $\mc E_2$ of equal length we  write $\mc E_1 \sim \mc E_2$ if $\mc E_2 \cong \mc E_1[\textbf{p}]$ for some $\textbf{p} \in \ZZ^{n+1}$.

\textit{An abelian category $\mc A$ is said to be hereditary, if ${\rm Ext}^i(X,Y)=0$ for any  $X,Y \in \mc A$ and $i\geq 2$,  it is said to be of finite length, if it is Artinian and Noterian.}

By $Q$ we denote an acyclic quiver and   by  $D^b(\textnormal{Rep}_\KK(Q))$, or just $D^b(Q)$, -  the derived category of the category of $\KK$-representations of $Q$. Sometimes we write $D^b(pt)$ instead of $D^b(A_1)$.

For an integer $l\geq 0$ the $l$-Kronecker quiver  (the quiver with two vertices and  $l$ parallel  arrows)  will be denoted by  $K(l)$: \begin{tikzcd}
	1
	\arrow[r, draw=none, "\raisebox{+1.5ex}{\vdots}" description]
	\arrow[r, bend left,]
	\arrow[r, bend right, swap]
	&
	2
\end{tikzcd}.

For a subset $S\subset G$  of a group $G$ we denote by $\langle S \rangle \subset G$ the subgroup   generated by $S$.

The number of elements of a finite set $X$ we denote by $\abs{X}$ or by $\#(X)$.

For integers  $a,b \in \ZZ$ we denote by $g.c.d(a,b)$ the greatest common divisor of $a,b$.

	For any $a\in \RR$ and any  complex number $z \in {\rm e}^{\ri \pi a} \cdot (\RR + \ri \RR_{>0})$, respectively  $z \in {\rm e}^{\ri \pi a} \cdot \left (\RR_{<0} \cup (\RR + \ri \RR_{>0}) \right )$,  we denote by $\arg_{(a,a+1)}(z)$, resp. $\arg_{(a,a+1]}(z)$,  the unique $\phi \in (a,a+1)$, resp. $\phi \in (a,a+1]$, satisfying $z=\abs{z} \exp(\ri \pi \phi)$.   
	
	For a non-zero complex number $v \in \CC$ we denote   the two connected components of  $\CC \setminus \RR v $ by:
	\begin{gather} \label{complement of a line} v^c_+ = v \cdot  (\RR + \ri \RR_{>0}) \qquad  v^c_- = v \cdot (\RR - \ri \RR_{>0}) \qquad \qquad v \in \CC \setminus \{0\}.\end{gather}

\begin{definition}
	Let $\mathcal{T}$ be a triangulated category and $\mathcal{T}_i$ for $i=1,\ldots,n$ be its triangulated subcategories, such that $\mathcal{T}=\langle\mathcal{T}_1,\ldots,\mathcal{T}_2\rangle$. One says that $\mathcal{T}=\langle \mathcal{T}_1,\ldots,\mathcal{T}_n\rangle $ is a semi-orthogonal decomposition of $\mathcal{T}$ when for any $X_i\in Ob(\mathcal{T}_i)$ and $X_j\in Ob(\mathcal{T}_j)$ the space of morphisms $\textrm{Hom}^l(X_i,X_j)$ is trivial for all $l$ whenever $i>j$. 
\end{definition}
\section{Triangulated categories and Bridgeland stability conditions}
\label{triangulatedcategories}
In this work, we use the definition of triangulated categories found for example in \cite[p. 239]{gelfmani}. We give here the form of the octahedral axiom that we will use later: If one has the following commutative diagram in a triangulated category $\mathcal{T}$
\begin{equation}
\label{lowerroof}
\begin{tikzcd}
E \arrow{rr}\arrow[bend left=15]{rrrr} && \arrow{dl} F \arrow{rr} && G\arrow{dl}\\
 & A \arrow{ul}{[1]} && B\arrow{ul}{[1]}\arrow{ll}{[1]}&
\end{tikzcd}
\end{equation}

with the triangles  $E,F,A$ and $F,B,G$ being distinguished and the others being commutative, then there exists an object $F'$ with a diagram 
\begin{equation}
\label{upperroof}
\begin{tikzcd}[column sep=1.5em]
E \arrow{rr}&&G\arrow{dl}\arrow{dd}\\
&F'\arrow{ul}{[1]}\arrow{dr}\\
A \arrow{uu}{[1]}\arrow{ur}&&B\arrow{ll}{[1]}
\end{tikzcd}\,,
\end{equation}

such that the upper and lower triangle are again distinguished and the left and right triangle are commutative. The arrows $A\to E$, $E\to G$, $G\to B$ and $B\to A$ are the same in both diagrams. 

\begin{remark}
\label{reverse8}
Using the axioms of a triangulated category one can also show that if one has a diagram \eqref{upperroof}, then one can construct a diagram of the form \eqref{lowerroof}, where the arrows $A\to E$ and $G\to B$ will be the same and the arrows $E\to G$ and $B\to A$ will get an additional minus sign in the newly constructed diagram. 
\end{remark}

The following  lemma  will be used in Section \ref{stabilities without stable curves}. 
\begin{lemma}
\label{equiavalentdiag}
If the diagram of distinguished triangles of the form
\begin{equation}
\label{equiavalentdiagequation}
\begin{tikzcd}
E \arrow{rr} && \arrow{dl}{g} F \arrow{rr} && G\arrow{dl}\\
 & A \arrow{ul}{[1]} && B\arrow{ul}{[1]}[swap]{f}&
\end{tikzcd}
\end{equation}
satisfies:

{\rm (a)}  $g[1]\circ f =0$, then there exists another diagram of distinguished triangles of the following form:
$$
\begin{tikzcd}
E \arrow{rr} && \arrow{dl} F'' \arrow{rr} && G\arrow{dl}\\
 & B \arrow{ul}{[1]} && A\arrow{ul}{[1]}&
\end{tikzcd} ;
$$

{\rm (b)}   $g[1]\circ f $ is isomorphism,  then the composition $
E\rightarrow F \rightarrow G$ is isomorphism. 
\end{lemma}
\begin{proof}
The commutative diagram 
$$
\begin{tikzcd}
E \arrow{rr}\arrow[bend left=15]{rrrr} && \arrow{dl}{g} F \arrow{rr} && G\arrow{dl}\\
 & A \arrow{ul}{[1]} && B\arrow{ul}{[1]}[swap]{f}\arrow{ll}{[1]}[swap]{0}&
\end{tikzcd}
$$
gives by the octahedral axiom explained above the commutative diagram of distinguished triangles of the form
$$
\begin{tikzcd}[column sep=1.5em]
E \arrow{rr}&&G\arrow{dl}\arrow{dd}\\
&F'\arrow{ul}{[1]}\arrow{dr}\\
A \arrow{uu}{[1]}\arrow{ur}&&B\arrow{ll}[swap]{[1]}
\end{tikzcd}\,.
$$ 

(b) In this case the arrow from $B$ to $A[1]$ in the lower distinguished triangle is isomorphism and then by the axioms of triangulated category  it follows that $F'$ is a zero object, hence by the upper distinguished triangle we see that the arrow  $E \rightarrow G$  is isomorphism as well.

(a)  By the vanishing of $B\to A$ the lower triangle is a biproduct diagram. Thus we can interchange $A$ and $B$ in the last diagram and apply Remark \ref{reverse8} to get the following diagram:
$$
\begin{tikzcd}
E \arrow{rr}\arrow[bend left=15]{rrrr} && \arrow{dl}{g} F'' \arrow{rr} && G\arrow{dl}\\
 & B \arrow{ul}{[1]} && A\arrow{ul}{[1]}\arrow{ll}{[1]}&
\end{tikzcd}
$$
\end{proof}

\subsection{Slicings and stability conditions} \label{stab conditions}
 T. Bridgeland  defined in \cite{Bridgeland}:

\begin{definition} \cite{Bridgeland}
\label{slicing}
Let $\mathcal{T}$ be a triangulated category. A slicing $\mathcal{P}$ is a collection of strictly full additive subcategories $\mathcal{P}(\phi)$ given for any $\phi\in\mathbb{R}$, such that the following conditions hold:
\begin{enumerate}
\item $\mathcal{P}(\phi)[1] = \mathcal{P}(\phi + 1)$\,.
\item If $X \in Ob\big(\mathcal{P}(\phi)\big)$ and $Y \in Ob\big(\mathcal{P}(\psi)\big)$ where $\phi>\psi$, then $\textrm{Hom}(X,Y) = 0$\,.
\item For any non-zero $E\in Ob(\mathcal{T})$ there exists a sequence $\phi_1>\phi_2>\ldots>\phi_n$ and a diagram of distinguished triangles
\begin{equation}
\label{17.1.18.1}
\begin{tikzcd}[column sep=1.5em]
 0 \arrow{rr} && \arrow{dl} E_1 \arrow{rr} && E_2\arrow{dl}\arrow{r}&\ldots \arrow[r]&E_{n-1}\arrow{rr}&&E\arrow{dl}\\
 & A_1 \arrow{ul}{[1]} && A_2\arrow{ul}{[1]}&& \,&&A_n\arrow{ul}{[1]}
\end{tikzcd}
\end{equation}
where $A_i\in Ob\big(\mathcal{P}(\phi_i)\big)$ are non-zero.
\end{enumerate}
\end{definition}
\begin{remark}
\label{uniqueHN}
One can prove that  the diagram \eqref{17.1.18.1} for any object $E$ is unique up to isomorphism.  We call this diagram the Harder-Narasimhan filtration of $E$. 
\end{remark}

\begin{definition} \cite{Bridgeland}
\label{stabcondef}
Let $\mathcal{T}$ be a triangulated category. A pair $\sigma = (Z, \mathcal{P})$, where $\mathcal{P}$ is a slicing on $\mathcal{T}$ and $Z: K_0(\mathcal{T})\to \mathbb{C}$ is a group homomorphism, is said to be a stability condition on $\mathcal{T}$ when for any non-zero $A\in \mathcal{P}(t)$ there exists $m_\sigma(A)\in\mathbb{R}^{>0}$, such that
\begin{equation} \label{property of stability}
Z(A) = m_\sigma(A) \exp\big(\ri\pi t\big)\,.
\end{equation}
The homomorphism $Z$ is then called the central charge of $\sigma$.
\end{definition}

For a stability condition $\sigma =(Z,\mathcal{P})$ a  non-zero object $A$ in $\mathcal{P}(t)$ for some $t$ is called $\sigma$-semistable,  we will write $\sigma^{ss}$ for the set of $\sigma$-semistable objects. Additionally,  we denote $\phi_\sigma(A) = t$  for  $A\in \mc P(t)\setminus \{0\}$,  and $\phi_\sigma^+(E) = \phi_\sigma(A_1)$, $\phi_\sigma^-(E) = \phi_\sigma(A_n)$ for any non-zero object $E$ with HN filtration as in \eqref{17.1.18.1}.
One is usually interested in a special case of stability conditions called \textit{locally finite} (see \cite[Definition 5.7]{Bridgeland} for details).

\begin{remark} \label{action}  In \cite{Bridgeland}, the set of locally finite stability conditions is equipped with a  structure of a complex manifold. This complex manifold is denoted by $\st(\mc T)$. 
Furthermore, 	in \cite{Bridgeland} a left action by biholomorphisms  of the group of exact autoequivalences $\Aut (\mc T)$ on $\st(\mc T)$ is constructed. This action is determined as follows:
	\begin{gather} \label{left action} \Aut(\mc T) \times \st(\mc T) \ni \left  (\Phi, (Z,\{ \mc P(t)\}_{t\in \RR}) \right ) \mapsto \left (Z\circ [\Phi]^{-1}, \left
	\{  \overline{\Phi(\mc P(t))} \right \}_{t\in \RR}\right )\in \st(\mc T), \end{gather}
	where $[\Phi]:K_0(\mc T) \rightarrow K_0(\mc T)$ is the induced isomorphism (we will often omit specifying the square brackets) and $\overline{\Phi(\mc P(t))} $ is the full isomorphism closed subcategory containing $\Phi(\mc P(t))$. 
	
	 For any $\Phi \in \Aut(\mc T)$, $\sigma\in \st(\mc T)$ let us denote $\sigma=(\mc P_{\sigma}, Z_\sigma)$, $\Phi\cdot \sigma=(\mc P_{\Phi \cdot \sigma}, Z_{\Phi\cdot \sigma})$, then  we have: 
	 \begin{gather} \label{aut properties 1}
	 (\Phi\cdot \sigma)^{ss}=\overline{\Phi(\sigma^{ss})} \qquad  \phi_{\Phi \cdot \sigma}(X)= \phi_{ \sigma}(\Phi^{-1} X) \qquad X\in (\Phi\cdot \sigma)^{ss} \\ 
	 \label{aut properties 2} Z_{\Phi \cdot \sigma}(X)=Z_\sigma(\Phi^{-1}(X)) \qquad X\in \mc T.\end{gather} 
\end{remark}

\section{Non-commutative curve counting. Dependence on a stability condition}
\label{noncommutativecurves}

In this section we present in details an approach  to make non-commutative counting dependable on stability condition. The idea is  sketched in \cite[Section 12.4]{DK1}.   We  make  also first  basic  observations. 
The main definition is 

\begin{definition} \cite[Definition 12.5]{DK1} \label{nccurves}   Let  $\mc A$, $\mc T$ be any  triangulated categories. And let $\Gamma \subset {\rm Aut}(\mc T)$ be a subgroup of the group of auto-equivalences. We denote 
	\begin{gather} \label{C'_l,P} 
	C'_{\mc A,P}(\mc T) = \{\bd \mc A & \rTo^{F} & \mc T \ed : F \ \mbox{is a fully faithful exact functor satisfying  properties} \ P \}.
	\end{gather} 
		Here we need $P$ to be a property of fully faithful functors, such that \eqref{C'_l,P} is a well defined set,   examples  are  in  Definitions \ref{remark for C_l},  \ref{main def}. 
	
	Next we fix an equivalence relation in $C'_{\mc A ,P}(\mc T)$: 
	\begin{gather}\label{equivalence}
	C_{\mc A,P}^{\Gamma}(\mc T) = C'_{\mc A,P}  (\mc T)/{\sim} \qquad F \sim F' \iff F \circ \alpha \cong \beta\circ F' \ \mbox{for some} \ \ \alpha \in {\rm Aut}(\mc A), \beta \in \Gamma  
	\end{gather}  where $ F \circ \alpha \cong \beta \circ F' $ means equivalence of exact functors between triangulated categories (this is so called graded equivalence).

\end{definition}

In  special cases, important role play the categories  $D^b(\textnormal{Rep}_{\KK}(K(l+1))) = N\PP^l$, $l\geq -1$. 

These special cases of Definition \ref{nccurves} are  relevant to non-commutative \textit{curve} counting and to intersection theory,  are:

\begin{definition} \cite[Definition 12.8]{DK1}, \cite[Definitions 4.3,  6.4]{dimkatz2} \label{remark for C_l}  
	
 Let $\mc T$ be $\KK$-linear and let  $\Gamma\subset\textrm{\rm Aut}_{\KK}(\mathcal{T})$ be a subgroup of the group of exact $\KK$-linear auto-equivalences on $\mathcal{T}$. Let  the property $P$ in \eqref{C'_l,P} be ``$F$ is $\KK$-linear''. 
	
		Let $l\geq -1$. We denote  $C_{N\PP^l, P  }'(\mc T)$ and  $C_{N\PP^l, P  }^\Gamma(\mc T)$  by  $C_{l }'(\mc T)$ and   $C_{l}^{\Gamma}(\mc T)$, respectively,  and refer to the elements of $C_{l}^{\Gamma}(\mc T)$ as to \underline{non-commutative curves  of non-commutative genus $l$ in $\mc T$} modulo $\Gamma$. 
	
	 We denote  $C_{D^b(pt), P  }'(\mc T)$ and  $C_{D^b(pt), P  }^\Gamma(\mc T)$  by  $C_{D^b(pt), \KK}'(\mc T)$ and   $C_{D^b(pt), \KK}^{\Gamma}(\mc T)$, respectively,  and refer to the elements of $C_{D^b(pt), \KK}^{\Gamma}(\mc T)$ as to \underline{derived points in $\mc T$} modulo $\Gamma$.

   Furthermore, in the special case when $\Gamma = \big\{[\textrm{\rm Id}_\mathcal{T}]\big\}$ we will omit writing $\Gamma$ in the superscript.
\end{definition} 
 The following proposition holds

\begin{proposition}   \cite[Proposition 5.5]{dimkatz2}.
\label{genbijec}
Let $\mathcal{T}$ be a proper $\KK$-linear triangulated category which has an enhancement, e.g. $\mathcal T = D^b(\textnormal{Rep}_\KK(Q))$ for an acyclic quiver $Q$. 

Then there are  bijections:

\begin{equation}
\label{25.1.18.1}
C_l(\mathcal{T})\to 
\left\{\begin{array}{ll}
&\mathcal{A} \textrm{ full triangulated subcategory of }\mathcal{T} \textrm{s.t.} \  \mathcal{A} = \langle E_0,E_1 \rangle \\
 \mathcal{A}\subset \mathcal{T} :& \ \textrm{for some strong exceptional pair} \  (E_0,E_1) \ \textrm{with} \\ & \hom(E_0,E_1)=l+1
\end{array}\right\}
\end{equation}
  \begin{gather}
 \label{bijection for eo} 	C_{ D^b(A_1), \KK }(\mc T) \rightarrow\left  \{\mc A\subset \mc T:\begin{array}{l} \mc A \ \mbox{is a full tr. subcategory s.t.} \ \mc A =\langle E\rangle \\  \mbox{for some exceptional object} \ E \end{array} \right \} 
 \end{gather}

defined  by $[F]\mapsto {\rm Im}(F)$. 
\end{proposition}

From now on,we will assume that  $\mathcal{T}$ is a proper $\KK$-linear triangulated category which has an enhancement and we will use Proposition \ref{genbijec}  to identify $C_l(\mathcal{T})$ and $	C_{ D^b(A_1), \KK }(\mc T)$ with the corresponding sets of subcategories in $\mc T$.

Using Definition \ref{nccurves}, the concept of $\sigma$-(semi)stability of  non-commutative curves in $\mathcal T$ was introduced  in  \cite{DK1}. The idea is to specify a restriction $P$ in Definition \ref{nccurves} depending on a stability condition $\sigma \in \st(\mathcal T)$.    The following definition is the same as  \cite[Definition 12.12]{DK1} when $l\geq 1$, the cases $l=-1, 0$ are not carefully explained in \cite[Definition 12.12]{DK1}, which we  repair here:  

\begin{definition} \label{main def}
Let $l\geq -1$.  Let  $\Gamma\subset\textrm{\rm Aut}_{\KK}(\mathcal{T})$ and let 
 $\sigma \in \mathcal{T}$.  

Let $\mc E \subset Ob(N\PP^l)$ be a set of exceptional objects in  $N\PP^l$, s.t. for each exceptional object $X \in  N\PP^l$ there exists unique $Y \in  \mc E$, such that $X \cong Y[k]$ for some $k\in \ZZ$. 

We specify the property $P$ from Definition \ref{nccurves} to two different cases and denote the set $C'_{N\PP^l,P}(\mathcal{T})$ from \eqref{C'_l,P}  by  $C'_{l,\sigma}(\mathcal{T})$,  $C'_{l,\sigma \sigma}(\mathcal{T})$, respectively. More precisely, (recall  $C'_l(\mc T)$ from Definition \ref{remark for C_l}):
\begin{gather}  \label{C'lsigma}	C'_{l,\sigma}(\mc T) = \left \{F \in C'_{l}(\mc T) : \abs{\{E\in \mc E: F(E)\in \sigma^{ss}\}}=\infty  \   \right \} \\ 
\label{C'lsigmasigma}	C'_{l,\sigma \sigma}(\mc T) = \{F \in C'_{l}(\mc T) :      
	\forall E\in \mc E \ \  F(E)\in \sigma^{ss} \  \}.  \end{gather}
 The formula \eqref{equivalence}  gives two sets of equivalence classes corresponding to the two different choices of  $P$ specified above  and we refer to these sets as sets of 
 $\sigma$-semistable (resp. $\sigma$-stable) non-commutative curves of non-commutative  genus $l$ in $\mc T$, and modulo $\Gamma$, we denote them as follows:
 \begin{gather} \nonumber 
 C_{l, \sigma
 }^{\Gamma}(\mc T) = C'_{l,\sigma}  (\mc T)/{\sim} \qquad 	C_{l, \sigma \sigma
}^{\Gamma}(\mc T) = C'_{l,\sigma \sigma}  (\mc T)/{\sim} 
\end{gather}

Again, we will skip  the superscript $\Gamma$ when $\Gamma$ contains only the identity.
\end{definition}
 
 \begin{remark} \label{remark for subsets} From \eqref{C'lsigma}, \eqref{C'lsigmasigma} and Definition  \ref{remark for C_l}  we see that   $ C'_{l,\sigma}(\mc T)$ and $ C'_{l,\sigma \sigma}(\mc T)$ are subsets of  $ C'_{l}(\mc T)$. We show here that these subsets  are unions of equivalence classes in  $ C_{l}(\mc T)$ (according to  Definition \ref{remark for C_l}   $ C_{l}(\mc T)$ is a set of equivalence classes, which are subsets of $ C'_{l}(\mc T)$). In particular,  inclusions  $C_{l,\sigma}(\mc T)\subset C_{l}(\mc T)$ and $ C_{l,\sigma \sigma}(\mc T)\subset  C_{l}(\mc T)$ hold.  
 	
 	Indeed, let $F\in C'_l(\mc T)$ and $G\in  C'_{l,\sigma}(\mc T)$, resp. $G\in  C'_{l,\sigma \sigma}(\mc T)$, be such that $F\circ \alpha \cong G$ for some $\alpha \in {\rm Aut}(N\PP^l)$. Since by definition both $F$ and $G$ are $\KK$-linear it follows  from \cite[Corollary 3.10]{dimkatz2}  that $\alpha$ is $\KK$-linear as well, and hence $\alpha(X)$ is exceptional iff $X$ is exceptional for any $X\in Ob(N\PP^l)$. It follows that $F\in  C'_{l,\sigma}(\mc T)$, resp. $F\in  C'_{l,\sigma\sigma}(\mc T)$,  as well. 
 	
 \end{remark}
 
 Next  we will discuss here the restriction of the bijection \eqref{25.1.18.1} to the subsets $C_{l,\sigma}(\mc T)\subset C_{l}(\mc T)$, $C_{l,\sigma\sigma}(\mc T)\subset C_{l}(\mc T)$.   We note first that the set $\mc E$ from  Definition \ref{main def} is bijective to the set of derived points in $N \PP^l$ (recall  Definition \ref{remark for C_l} and \eqref{bijection for eo}):
 
 \begin{remark} \label{remark for derived points 1}
 		The objects in a derived point  $\langle E \rangle$ in any proper $\KK$-lineat $\mc T$  are exactly all the possible  finite  direct sums of shifts of $E$.  Therefore the function $\mc E \rightarrow C_{D^b(pt),\KK}(N \PP^l)$ defined by $\mc E \ni E \mapsto \langle  E \rangle \in  C_{D^b(pt),\KK}(N \PP^l)$ is a bijection, where $\mc E\subset Ob(N \PP^l)$  is as in Definition \ref{main def}. Hence for any  $F \in C'_{l}(\mc T)$ the function $\mc E \ni E \mapsto \langle  F(E) \rangle \in  C_{D^b(pt),\KK}({\rm Im}(F))$ is also a bijection.  
 
 \end{remark}
 
 Taking into account this we define:
 
 \begin{definition} \label{def of stable derived points} Let $\sigma \in \st(\mc T)$.   Due to the first sentence in  Remark \ref{remark for derived points 1} we see that if we have a derived point $\mc A \in  C_{D^b(pt),\KK}(\mc T)$ and $\mc A =\langle A \rangle=\langle B  \rangle$ with $A$, $B$ exceptional objects, then $A \in \sigma^{ss}$ if and only if  $B \in \sigma^{ss}$. Hence the  set  $C_{D^b(pt), \sigma \sigma}(\mc T)=\{\langle E \rangle \in  C_{D^b(pt),\KK}(\mc T): E \in \sigma^{ss} \}$ is well defined. 
 \end{definition}
 The following is a new definition motivated by counting of derived points.
 
\begin{definition}
\label{definition derived points outside of curves}
Let $\sigma\in\stab$, we define
$$
\mathfrak{D}_{\sigma}=\{\langle E\rangle \in C_{D^b(pt),\sigma\sigma} : \langle E\rangle \not\subset \mathcal{A} \quad \forall \, \mathcal{A}\in C_{l,\sigma\sigma}, l\geq 1\}\,.
$$
The elements of $\mathfrak{D}_\sigma$ are called \underline{stable derived points outside of stable curves of positive genus}.
\end{definition}
 
 Using this terminology and Remark \ref{remark for derived points 1} we see that:
 
 \begin{lemma} \label{sigmasemistablesubcategory}
Let $l\geq -1$ and $\sigma \in \st(\mc T)$. In  Remark \ref{remark for subsets} we explained that $C_{l,\sigma}(\mc T)$, $C_{l,\sigma\sigma}(\mc T)$ are subsets of $C_{l}(\mc T)$.   Let $\mc A\subset \mc T$ be an element in the codomain of \eqref{25.1.18.1}. Then 

{\rm (a) } $\mc A \in C_{l,\sigma}(\mc T)$  iff $\abs{C_{D^b(pt),\sigma \sigma}(\mc T)\cap C_{D^b(pt),\KK}(\mc A)} =\infty$. 

{\rm (b) }  $\mc A \in C_{l,\sigma \sigma}(\mc T)$  iff $ C_{D^b(pt),\KK}(\mc A)\subset C_{D^b(pt),\sigma \sigma}(\mc T)$.
 \end{lemma}

 \begin{corollary} \label{closed subset} For any $\mc A \in C_l(\mc T)$,
 $l\geq -1$ 	the set $\{\sigma \in \st(\mc T):  \mc A\in C_{l,\sigma \sigma}(\mc T)\}$ is a closed subset in $\st(\mc T)$. 
 \end{corollary}
 \begin{proof}  Let $\mc E$ be the set of exceptional objects  in $\mc A$. From \cite[p. 342]{Bridgeland} is known that for any object $X$ in $\mc T$ the set $\{\sigma \in \st(\mc T): X\in \sigma^{ss}\}$ is closed. From Lemma  \ref{sigmasemistablesubcategory} we see that  $\{\sigma \in \st(\mc T): \mc  A\in C_{l,\sigma \sigma}(\mc T)\} = \bigcap_{E \in\mc E} \{\sigma \in \st(\mc T): X\in \sigma^{ss}\}.$
 	\end{proof} 
  \begin{remark} \label{remark for helix}  Let $\mc E$ be a set of exceptional objects as in Definition  \ref{main def}.
  	
  (a)	Let $l\geq 1$ and  let $\{s_i\}_{i\in \ZZ}$ be a Helix in $N\PP^l$ (see \cite[Subsection 7.1]{DK1}). Then from \cite[(72), (73)]{DK1} we see that we can choose  $\mc E=\{s_i:i\in \ZZ\}$ and that the function $\ZZ \ni i \mapsto s_i \in \mc E$ is bijection. 
  	
  (b)	If $l=0$ or $l=-1$, then  $N\PP^l$, and hence $\mc E$, has three or respectively two  elements (see \cite[Figure 1]{dimkatz2}) and it follows that   for $l=0, l=-1$ we have always $	C_{l, \sigma
  	}^{\Gamma}(\mc T)=\emptyset$. The derived points in a genus zero nc curve cannot be ordered in a semi-orthogonal triple (follows also from \cite[Figure 1]{dimkatz2}).  
  	
  	(c) From (a) it follows that for $l\geq 1$  we have $C_{l,\sigma\sigma}(\mc T) \subset C_{l,\sigma}(\mc T) \subset C_{l}(\mc T)$  (recall  remark \ref{remark for subsets}).
  \end{remark}

\begin{remark} \label{remark for general A}
Taking into account Remark \ref{sigmasemistablesubcategory}, it seems reasonable  to call the elements in  $C_{l,\sigma \sigma}(\mc T)$ point-wise semistable and those in $C_{l,\sigma}(\mc T)$ - almost everywhere point-wise semistable. Also one can define similarly point-wise semistability for any subcategory in $\mc T$. Actually, provided that one has defined semi-stability for some triangulated category $\mc A$ ( could be different from derived point), then by analogy one gets in similar fashion notions of  $\mc A$-wise semistable and of almost everywhere $\mc A$-wise semistable subcategories in $\mc T$. However, here we will use the shorter terminology fixed in Definition \ref{main def} and its equivalent description in Lemma \ref{sigmasemistablesubcategory}. 
\end{remark}

\section{Some remarks for $\sigma$-exceptional collections and   $C_{l,\sigma \sigma}(\mc T)$} \label{stabilities without stable curves}

\textit{From now on, we will assume that the triangulated categories we are working with are proper. }

Firstly, we recall a Definition from \cite{dimkatz3} : 
\begin{definition} \cite[Definition 3.17, Remark 3.19]{dimkatz3} 
\label{sigma exceptional collection} Let $\sigma =({\mc P}, Z) \in \st(\mc T)$. An exceptional collection ${\mc E }=(E_0,E_1,\dots,E_n)$  is called
	\textit{$\sigma$-exceptional collection} if the following properties hold:
	\begin{itemize}
		\item $\mc E$ is semistable w. r. to $\sigma$ (i. e. all $E_i$ are semistable).
		\item $\forall i\neq j$ $\hom^{\leq 0}(E_i, E_j)=0$ (i. e. this is an Ext-exceptional collection).
		\item There exists $t \in \RR$, s. t. $\{\phi(E_i)\}_{i=0}^n \subset (t,t+1]$ (this is equivalent to $\abs{\phi(E_i)-\phi(E_j)}<1$ for $0\leq i<k\leq n$).
	\end{itemize}
\end{definition}
\begin{remark} \label{remark for fe} Let $\mc E$ be a full exceptional collection. The following is  a well defined assignment: \begin{gather} \label{assignment} \bd  \{ \sigma \in \st(\mc T): \mc E \subset \sigma^{ss}\} \ni (\mc P,Z) & \rMapsto^{f_{\mc E}}  & \left (\{\abs{Z(E_i)} \}_{i=0}^n, \{ \phi_\sigma(E_i) \}_{i=0}^n\right )  \in \RR^{2(n+1)} \ed \end{gather}
	
{\rm (a) } Let $\mc E$ be Ext\footnote{satisfies the second condition in Definition \ref{sigma exceptional collection}}	and let $\Theta'_{\mc E}$ denote the set of stability conditions for which $\mc E$ is $\sigma$-exceptional. From \cite[Corollary 3.18, Remark 3.19 ]{dimkatz3}  we know that  $\Theta'_{\mc E}$ is an open subset in $\st(\mc T)$ and \eqref{assignment} 
 restricted to  $\Theta_{\mc E}'$  gives a homeomorphism between $\Theta_{\mc E}'$ and  the following  simply connected set:\\ 
$\left \{(x_0,\dots,x_n,y_0,\dots,y_n) \in \RR^{2(n+1)}\ : \ x_i>0,\ \abs{y_i-y_j}<1 \right \}.$ 

{\rm (b) } Let $\mc E$ be any full exceptional collection in a proper $\mc T$, then after shifting it is an Ext-exceptional collection.  Let   $\Theta_{\mc E}$ denote the set of stability conditions $\sigma\in \st(\mc T)$, such that a shift of $\mc E$ is $\sigma$-exceptional.  In \cite[ (11) ]{dimkatz4} it is shown that \eqref{assignment} restricted to the open subset  $\Theta_{\mc E}$ is also homeomorphism, however the image is more complicated than for  $\Theta_{\mc E}'$. 
\end{remark}

For a class of objects $\mathcal{E}$ in  $\mathcal{T}$ we denote by $\hat{\mathcal{E}} $  the extension closure in $\mathcal{T}$ of $\mathcal{E}$.

\begin{remark}
	\label{exceptionalgenerator}
Let $E\in \mc T$ be an exceptional object. It is well known that  every object $X$ in $\langle E\rangle $ is a direct product of finitely many objects from the class $\{E[i]\}_{i\in\mathbb{Z}}$, that  $\hat E[i]$ is a heart of a bounded  t-structure in $\langle E\rangle $, and that  $\hat E[i]$ is of finite length for each $i\in \ZZ$. Due to  $\textnormal{\rm hom}^i(E,E)=0$ whenever $i\neq 0$, it follows that any object $X$ in $\hat E$ is a direct product of copies of $E$.
\end{remark}

\begin{proposition}
	\label{nonentangled}
	Let $\mathcal{T}$ be a triangulated category. Let $\sigma \in \st(\mc T)$.  If  $\mathcal{E} = (E_0,\ldots,E_n)$ its full $\sigma$-exceptional collection with strictly increasing phases of its objects, i. e.  $\phi_{\sigma}(E_{i+1})>\phi_{\sigma}(E_i)$ for all $0\leq i<n$, then the slicing $\mathcal{P}$ of $\sigma$ is given in the following way. 
	\begin{align}
		\label{27.2.18.1}
		\mathcal{P}(t) = \begin{cases}
			\hat{E_i}  [j]&\textrm{when }t = \phi_{\sigma}(E_i) + j\\
			\{0\}&\textrm{otherwise}
		\end{cases}
	\end{align}
\end{proposition}
\begin{proof} We note first that from the given properties of $\mc E$ and $\sigma$ it follows:
	
	\begin{equation} \label{remark for non-entangled} 
	\phi_{\sigma}(E_{i_2}[j_2]) > \phi_{\sigma}(E_{i_1}[j_1])\iff j_2>j_1\textrm{ or }j_2=j_1,i_2>i_1\,.
	\end{equation}

	Let us show that $\sigma' =(\mathcal{P},Z)$ indeed gives a locally finite  stability condition when $\mathcal{P}$ is given by the equation \eqref{27.2.18.1} and $Z$ is the central charge of the given stability condition. For that, we only need to show that $\mathcal{P}$ is a locally finite  slicing on $\mathcal{T}$.  The first two axioms of the Definition \ref{slicing}  follow immediately. By a property of semi-orthogonal decomposition,  for any object $F$ there is a filtration  with factors $A_i$ that are objects of $\langle E_i\rangle $ and as such are the direct sums of shifts of $E_i$ by Remark \ref{exceptionalgenerator}. So $\mathcal{T}$ is the extension closure of $\{E_i[j]\}_{0\leq i\leq n,j\in\mathbb{Z}}$. However, when there is a diagram of the form 
	$$
	\begin{tikzcd}
	A \arrow[rr]&&B\arrow[rr]\arrow{dl}&&C\arrow{dl}\\
	&E_{i_1}[j_1]\arrow{ul}{[1]}&&E_{i_2}[j_2]\arrow{ul}{[1]}
	\end{tikzcd}
	$$
	where $\phi_Z(E_{i_2}[j_2]) > \phi_Z(E_{i_1}[j_1]) $, then by $\mathcal{E}$ being Ext-exceptional and using  \ref{remark for non-entangled} it follows that  the composition $E_{i_2}[j_2]\to B[1]\to E_{i_1}[j_1+1]$ is either  zero morphism or isomorphism \footnote{in the case $i_1=i_2$ and $j_2=j_1+1$ this arrow is of the form $E_{i_1}[j_1+1]\to E_{i_1}[j_1+1]$}, so we see from Lemma \ref{equiavalentdiag} that we can either  interchange the order of the factors  $E_{i_1}[j_1]$, $ E_{i_2}[j_2]$  or the  composition $A\rightarrow B \rightarrow C $ is isomorphism and  we can stick $A$ and $C$. Doing so until any two neighboring factors $E_{i_1}[j_1]$ and $E_{i_2}[j_2]$ have $\phi_Z(E_{i_1}[j_1])\geq \phi_Z(E_{i_2}[j_2])$ and then using the octahedral axiom where  two neighboring factors $E_{i_1}[j_1]$ and $E_{i_2}[j_2]$ have $\phi_Z(E_{i_1}[j_1]) = \phi_Z(E_{i_2}[j_2])$, gives us a HN-filtration of $F$ with respect to $\mathcal{P}$. In this slicing  there exists  $\epsilon>0$, such that $\mc P(t-\epsilon,t+\epsilon)$  is either trivial or $\hat{E_i}  [j]$ for some $i,j$, which by Remark \ref{exceptionalgenerator} is a locally finite abelian category,  and therefore $\mc P$ is locally finite slicing (see \cite[Deifinition 5.7]{Bridgeland}). Therefore $\sigma'\in \st(\mc T)$. 
	
	Since $\phi_{\sigma'}(E_i)=\phi_{\sigma}(E_i)$ for $0\leq i \leq n$ it follows that $\mc E $ is also  $\sigma'$-exceptional. Furthermore, the central charges of $\sigma$ and $\sigma'$ are the same. Therefore the bijection \eqref{assignment} has equal values on $\sigma$ and $\sigma'$ and therefore $\sigma'=\sigma$.  
\end{proof}

\begin{corollary}

	\label{nonentangled non semistable}
	Let  $\mathcal{T}$ be a triangulated category and $\sigma\in \st(\mc T)$. Let $\mathcal{E} = (E_0,\ldots,E_n)$ be a full $\sigma$-exceptional collection with strictly increasing phases.  Then the set of   stable  derived points, as defined in Definition \ref{def of stable derived points}, is $C_{D^b(pt), \sigma \sigma}=\{\langle E_0 \rangle, \langle E_1 \rangle, \dots, \langle E_n \rangle \}$. In particular this set is finite and $C_{l,\sigma \sigma }(\mc T) =C_{l,\sigma}(\mc T) = \emptyset$ for $l\geq 0$. 

\end{corollary}
\begin{proof}  From the previous  Proposition we have $\sigma^{ss}=\{E: E \in \hat{E_i}[j]: 0\leq i \leq n, j \in \ZZ\}$. If $\langle E \rangle$ is a stable derived point, then from Definition \ref{def of stable derived points} $E \in \sigma^{ss}$, hence $E \in  \hat{E_i}[j]$ for some $ 0\leq i \leq n $  and some $ j \in \ZZ$  and from Remark \ref{exceptionalgenerator} it follows that $\langle E \rangle = \langle E_i \rangle $. Obviously $ \langle E_i \rangle $ is stable for each $i$. Therefore indeed  $C_{D^b(pt), \sigma \sigma}=\{\langle E_0 \rangle, \langle E_1 \rangle, \dots, \langle E_n \rangle \}$, and it has $n+1$ elements.  Now from Lemma \ref{sigmasemistablesubcategory} it follows that  $C_{l,\sigma}(\mc T) =C_{l,\sigma \sigma}(\mc T) = \emptyset$ for $l\geq 1$. On the other hand, if $ C_{0,\sigma \sigma}(\mc T) \neq \emptyset$ and  $\mc A \in C_{0,\sigma \sigma}(\mc T)$, then the three derived points in $\mc A$ (see  Remark \ref{remark for helix} (b)) must be  of the form $ \langle E_i \rangle, \langle E_j \rangle,  \langle E_k \rangle$ with $0\leq i<j<k\leq n$, which contradicts the last sentence in  Remark \ref{remark for helix} (b). \end{proof}

\begin{remark} \label{remark for non-entangled 1} Let $\mathcal{E}=(E_0,\ldots,E_n)$ be a full Ext-exceptional collection. 
	The homeomorphism \eqref{assignment} restricted to the set of stability conditions $\sigma \in \Theta_{\mc E}'$ with strictly increasing phases on $\mc E$   gives a homeomorphism between this set and the following  subset of $\RR^{2(n+1)}$:
	$$\left \{(x_0,\dots,x_n,y_0,\dots,y_n) \in \RR^{2(n+1)}\ \ : \ \ x_i>0,\ y_0<y_1<\dots < y_n < y_0+1 \right \}.$$
\end{remark}

\begin{proposition} \label{stable curve part 1}
		Let $\mathcal{T}$  have an enhancement, let $\sigma \in \st(\mc T)$ and let $\mathcal{E} = (E_0,\ldots,E_n)$ be a  full $\sigma$-exceptional collection, such that   $\hom^{1}(E_i,E_{i+1})=l+1$ and $\hom^{j}(E_i,E_{i+1})=0$ for $j\in \ZZ \setminus \{1\}$  for some $0\leq i <n$ and some $l\geq 0$. If  $\phi(E_i)\geq \phi(E_{i+1})$, then $\langle E_i,E_{i+1} \rangle \in C_{l,\sigma \sigma}(\mc T)$. 
\end{proposition}
\begin{proof} From  Proposition \ref{genbijec} we see that $\langle E_i,E_{i+1} \rangle \in C_{l}(\mc T)$, in particular there is an exact equivalence \begin{gather} \label{equivalence of Kl and} F: \langle  E_i,E_{i+1} \rangle\rightarrow D^b(K(l)).\end{gather}  We will show that all exceptional objects of $\mc T$, which lie in  $\langle E_i,E_{i+1} \rangle$,  are in $\sigma^{ss}$ and then from Lemma \ref{sigmasemistablesubcategory} will follow that $\langle E_i,E_{i+1} \rangle \in C_{l,\sigma \sigma}(\mc T)$. 
	Let $\mc A$ be the extension closure of $(E_i,E_{i+1})$ in $\mc T$.    From \cite[Proposition 2.2]{dimkatz4} we see that each exceptional object in $\mc T$, which is also in $\mc A$, is in $\sigma^{ss}$. To finish the proof it is enough to  show that  each exceptional object in $\langle E_i, E_{i+1} \rangle$ is isomorphic to a shift of  an object in $\mc A$. Since $F$ in \eqref{equivalence of Kl and} is exact equivalence, it is enough to prove that 
	\begin{gather}\label{help thing for a lemma} \mbox{\it  Each exceptional object in $D^b(K(l+1))$ is isomophic to a shift of an object in $F(\mc A)$,} \end{gather}
	
Now $F(E_i)$ is an exceptional object in $D^b(K(l+1))$. Using an auto-equivalence of $D^b(K(l+1))$ we can assume that $F(E_i)$ is the simple representation in $\textnormal{Rep}_{\KK}(K(l+1))$ with dimension vector $(1,0)$ (for the case $l=0$ we use \cite[Proposition 10.9]{dimkatz2} and for the case $l\geq 1$ we use \cite[Corollary 5.3]{DK31}). On the other hand  $F(E_i), F(E_{i+1})$ is an  exceptional pair and it follows that $F(E_{i+1})$ is isomorphic to a shift of the simple representation in  $\textnormal{Rep}_{\KK}(K(l+1))$ with dimension vector $(0,1)$ (this follows for example from results in \cite{crawley}). However, we have also  $\hom^1(F(E_i), F(E_{i+1})=l+1$ and it follows that  this shift is trivial. Hence we see that we can assume that $(F(E_i), F(E_{i+1}))$ is the exceptional  pair of the simple representations of $K(l)$. It follows that each  $X\in \textnormal{Rep}_{\KK}(K(l+1))$ is isomorphic to $F(A)$ for some $A\in \mc A$, and since $\textnormal{Rep}_{\KK}(K(l+1))$ is hereditary we obtain \eqref{help thing for a lemma}.
\end{proof} 
\begin{corollary} \label{stable curve}	
		Let $\mathcal{T}$  have  an enhancement, let  $\mathcal{E} = (E_0,\ldots,E_n)$ be a  full exceptional collection such that for some  $0\leq i < n$, some $\alpha \in \ZZ$, and some $l\geq 0$ holds 		
		$\hom^{1+\alpha}(E_i,E_{i+1})=l+1$ and  $\hom^{j}(E_i,E_{i+1})=0$ for $j\in \ZZ \setminus \{1+\alpha\}$.
		 If	 $\sigma \in \Theta_{\mc E}$ and  $\phi(E_i)\geq \phi(E_{i+1}[\alpha])$, then $\langle E_i,E_{i+1} \rangle \in C_{l,\sigma \sigma}(\mc T)$.
\end{corollary}	
\begin{proof}
	By the definition of $\Theta_{\mc E}$ in   Remark \ref{remark for fe},  the sequence $(E_0[p_0], \dots E_{n}[p_n])$ is $\sigma$-exceptional for some $(p_0,\dots, p_n)\in \ZZ^{n+1}$. We can ssumme that $p_i=0$ and therefore $\hom^j(E_i,E_{i+1}[p_{i+1}])=0$ for $j\leq 0$, $\abs{\phi(E_i)-\phi(E_{i+1}[p_{i+1}])}<1$.  Since we are given $\hom^{1+\alpha}(E_i,E_{i+1})=l+1$ and  $\hom^{j}(E_i,E_{i+1})=0$ for $j\in \ZZ \setminus \{1+\alpha\}$, therefore $p_{i+1}\leq \alpha$. If $p_{i+1}<\alpha$, then $p_{i+1}=\alpha-\delta$ for some $\delta \geq 1$ and $\phi(E_i)-\phi(E_{i+1}[p_{i+1}])=\phi(E_i)-\phi(E_{i+1}[\alpha])+\delta \geq \delta\geq 1$ by the given $\phi(E_i)-\phi(E_{i+1}[\alpha])\geq 0$, and this contradicts $\abs{\phi(E_i)-\phi(E_{i+1}[p_{i+1}])}<1$, Therefore $p_{i+1}=\alpha$ and we apply the  Proposition \ref{stable curve part 1} to the sequence $(E_0[p_0], \dots E_{n}[p_n])$.
\end{proof}

\section{The example: Stable  non-commutative curves  in $D^b(Q)$} \label{section for semistable curves}

 We fix  $\mc T=D^b(Q)$, where  $Q$ is as in \eqref{triangular quiver}. 
Here we will study how the sets $C_{i,\sigma \sigma}(\mc T) \subset C_{i,\sigma}(\mc T)$ (defined in Definition \ref{main def})  change   as $\sigma$ varies in $\st(\mc T)$ and $i\in \ZZ$.   Recalling that $C_l(\mc T) = \emptyset$  for $l\geq 2$ (see Proposition \eqref{noncomcurves}) and using  Remark \ref{remark for helix} (c) we see that  $C_{l,\sigma \sigma}(\mc T) = C_{l,\sigma}(\mc T)=\emptyset$ for $l\geq 2$.

Our study uses to a great extent   results from \cite{dimkatz2, dimkatz3, dimkatz4} regarding exceptional objects in $\mc T$ and $\st(\mc T)$. We  present these results in the first two subsections.

\subsection{On the exceptional objects in $D^b(Q)$}
\label{exceptionalobjectofthequiver}

Let $\underline{\dim}(E) =(x,y,z)$ be the dimension vector of an exceptional representation in $\textrm{Rep}_{\KK}(Q)$.
Let $\pi^m_{\pm}:{\KK}^{m+1}\to {\KK}^m$ and $j^m_{\pm}:{\KK}^m\to {\KK}^{m+1}$ be linear maps such that
\begin{align*}
\pi^m_{+}(x_1,\ldots,x_{m+1}) = (x_1,\ldots,x_{m})\qquad\pi^m_{-}(x_1,\ldots,x_{m+1}) = (x_2,\ldots,x_{m+1})\\
j^m_+(x_1,\ldots,x_m) = (x_1,\ldots,x_m,0)\qquad
j^m_-(x_1,\ldots,x_m) = (0,x_1,\ldots,x_m)\,.
\end{align*}
\begin{proposition}\cite[Proposition 2.2]{dimkatz3} \label{the exceptionals}
	The exceptional objects of $\mathcal{T}$ up to equivalence are:
	$$
	E^m_1=\begin{tikzcd}[column sep = 1em,row sep =2em]
	&{\KK}^m\\
	{\KK}^{m+1}\arrow{rr}{\pi^m_-}\arrow{ur}{\pi^m_+}&&{\KK}^m\arrow[ul,swap, "\textrm{id}"]
	\end{tikzcd}\,E^m_2 =\begin{tikzcd}[column sep =1em,row sep =2em]
	&{\KK}^{m+1}\\
	{\KK}^m\arrow{rr}{j^m_-}\arrow{ur}{j^m_+}&&{\KK}^{m+1}\arrow[ul,swap, "\textrm{id}"]
	\end{tikzcd}\,E^m_3=\begin{tikzcd}[column sep = 1em,row sep =2em]
	&{\KK}^{m+1}\\
	{\KK}^m\arrow{rr}{\textrm{id}}\arrow{ur}{j^m_+}&&{\KK}^m\arrow[ul,swap, "j^m_-"]
	\end{tikzcd}
	$$
	$$
	E^m_4= \begin{tikzcd}[column sep = 1em,row sep =2em]
	&{\KK}^m\\
	{\KK}^{m+1}\arrow{rr}{\textrm{id}}\arrow{ur}{\pi^m_+}&&{\KK}^{m+1}\arrow[ul,swap,"\pi^m_-"]
	\end{tikzcd}\,M = \begin{tikzcd}[column sep = 2em,row sep =2em]
	&0\\
	0\arrow{rr}\arrow{ur}&&{\KK}\arrow[ul]
	\end{tikzcd}\,M' = \begin{tikzcd}[column sep = 2em,row sep =2em]
	&{\KK}\\
	{\KK}\arrow{rr}\arrow{ur}&&0\arrow[ul,""]
	\end{tikzcd}
	$$
	where $m$ goes over all non-negative integers. 
\end{proposition}

Following \cite{dimkatz4}, we denote
\begin{equation} \label{ab m}
a^m=\begin{cases}
E^{-m}_1& m\leq 0\\
E^{m-1}_2[1]& m\geq 1
\end{cases}\,,
\qquad b^m=\begin{cases}
E^{-m}_4& m\leq 0\\
E^{m-1}_3[1]& m\geq 1
\end{cases}\,.
\end{equation} 
Due to Proposition \ref{the exceptionals} and Corollary \ref{bijection for eo} holds \eqref{derived points}.

The following results from \cite{dimkatz3,dimkatz4} will be used often later.

\begin{corollary}(\cite[Corollary 3.7]{dimkatz4})\label{nonvanishings} For each $m\in \ZZ$ we have: 
	\begin{gather} \label{nonvanishing1} \hom(M',a^m)\neq 0; \quad \hom(M,b^m)\neq 0; \quad  \hom^*(a^m,M')= 0;  \\ 
	\label{nonvanishing2} \hom^1(a^m,M)\neq 0; \quad \hom^1(b^m,M')\neq 0; \quad  \hom^*(b^m, M)=0 \\ 
	\label{nonvanishing3}  \hom^*(b^{m},a^{m-1}) = 0; \quad  \hom(b^{m},a^{n}) \neq 0   \ \mbox{for} \ m\leq n;  \quad  \hom^1(b^{m},a^n) \neq 0 \ \mbox{for} \ m> n+1;   \\
	\label{nonvanishing4}  \hom^*(a^m,b^{m}) = 0;  \quad   \hom(a^m,b^{n+1}) \neq 0    \ \mbox{for} \ m\leq n; \quad  \hom^1(a^{m},b^{n}) \neq 0 \ \mbox{for} \ m>n;    \\
	\label{nonvanishing5}  \hom^*(a^{m},a^{m-1})=0; \quad \hom(a^{m},a^n) \neq 0 \ \mbox{for} \ m\leq  n;   \quad  \hom^1(a^{m},a^{n}) \neq 0   \ \mbox{for} \ m> n+1; \\ 
	\label{nonvanishing6} \hom^*(b^{m},b^{m-1})=0; \quad  \hom(b^{m},b^n) \neq 0 \ \mbox{for} \ m\leq  n;   \quad  \hom^1(b^{m},b^{n}) \neq 0   \ \mbox{for} \ m> n+1;\\ 
	\label{nonvanishing7}	\hom^1(M,M') \neq 0 \qquad \hom^1(M',M)\neq 0 .\end{gather}
\end{corollary}

\begin{corollary} \label{one nonvanishing degree} \cite[Corollary 2.6 (b)]{dimkatz3}  For any two exceptional objects $X, Y \in D^b(Q )$ at most one  element of the family $\{ \hom^p(X,Y) \}_{p\in \ZZ}$ is nonzero.
\end{corollary}

In \cite[Remark 3.15]{dimkatz4} is explained that 
for any  $p$, $q \in \ZZ$ there are   distinguished triangles
\begin{equation} \label{short filtration 1} 
\begin{diagram}[size=1em] 
b^{p+1}[-1] & \rTo      &        &       &   M'   \\
& \luDashto &        & \ldTo &        \\
&           &  a^p   &       &        
\end{diagram} \qquad \qquad 
\begin{diagram}[size=1em] 
a^{q}[-1] & \rTo      &        &       &   M   \\
& \luDashto &        & \ldTo &        \\
&           &  b^q   &       &        
\end{diagram}.
\end{equation} 
\begin{remark} \label{remark for delta}
	We will denote by $\delta \in K_0(D^b(Q))$ the element $\delta = [E_1^0]+[M]+[E_3^0]$. For $\delta$ holds $\delta =[M]+[M']$. In particular  for any stability condition $(Z, \mc P) \in \st(\mc T)$ hold the equalities $Z(\delta)=Z(M)+ Z(M')$. Furthermore, from \eqref{short filtration 1}  we see that $Z(a^p)=Z(M')+ Z(b^{p+1})$ and $Z(b^q)=Z(M)+ Z(a^{q})$ for any $p,q\in \ZZ$. On the other hand, the latter two equalities imply that $Z(b^p)=Z(\delta)+Z(b^{p+1})$,  $Z(a^p)=Z(\delta)+Z(a^{p+1})$.
\end{remark}
\subsection{On stability conditions on $D^b(Q)$}\label{homeomoeprhism}
 The entire  manifold $\st(\mc T)$ was  described in \cite{dimkatz4}, and there it was proved that it is contractible.
We recollect some material from loc cit.

For a given triple $(A,B,C)\in \mk{T}$  we will denote the open subset $\Theta_{(A,B,C)} \subset \st(D^b(Q))$ from  Remark \ref{remark for fe} (b) by   $(A,B,C)$, when no confusion may arise.
In \cite[Proposition 2.7]{dimkatz4} it is shown that the image of the homeomorphism in Remark \ref{remark for fe} (b) is a set of the form  \begin{equation}
\label{eqSet}\RR_{>0}^3 \times   \left \{  \begin{array}{c} y_0 - y_1 < 1+\alpha \\  y_0 - y_2 < 1+\min\{\beta, \alpha+\gamma\} \\ y_1 - y_2 < 1+\gamma\end{array} \right \}\,,\end{equation}
where $\alpha, \beta, \gamma\in \ZZ$ are defined  in \cite[(30)]{dimkatz4}. An important role play the exceptional triples. They are:
\begin{proposition}\cite[Corollary 3.12, Remark 3.14]{dimkatz4}
	\label{excpairs}
	\label{exctriples}
 Define the following sets of full exceptional triples in $\mc T$: 
 \begin{enumerate}[(i)]
     \item ${\mk T}_a = \left \{   (M',a^{m},a^{m+1}), (a^m, b^{m+1},a^{m+1}),  (a^m,a^{m+1},M) : m\in \ZZ  \right \}$,
     \item ${\mk T}_b =	\left \{ (M,b^{m},b^{m+1}), (b^m,a^m,b^{m+1}), (b^m,b^{m+1},M')  : m\in \ZZ \right \}$.
 \end{enumerate}	 
  	The set of all full exceptional collections  in $\mathcal{T}$ up to shifts  is:
	\begin{align} \label{exceptional triples}
	{\mk T}=	{\mk T}_a \cup \left \{ (a^m , M, b^{m+1}) : m\in \ZZ \right \} \cup 	\left \{ (b^{m},M',a^{m}) : m\in \ZZ \right \} \cup {\mk T}_b. 
	\end{align}
	
\end{proposition}  

The stability conditions in $(A,B,C)$ are determined by the properties that the objects $A,B,C$ are  in $\sigma^{ss}$ and their phases satisfy the inequalities summarized in the following tables:
\begin{gather} \label{left right M} \begin{array}{| c   | c  |}   \hline  
\mk{T}_{a}^{st}       &  \mk{T}_{b}^{st}  \\
\hline
\left \{  M',a^{j},a^{j+1} \in \sigma^{ss} :   \begin{array}{c} \phi\left (M'\right) < \phi\left (a^{j}\right) \\ \phi\left ( M'\right)+1 < \phi\left (a^{j+1}\right) \\ \phi\left ( a^{j}\right) < \phi\left (a^{j+1} \right)\end{array}\right \}                                                &    \left \{  M,b^{n},b^{n+1} \in \sigma^{ss} :                     \begin{array}{c} \phi\left (M\right) < \phi\left (b^{n}\right) \\ \phi\left ( M\right)+1 < \phi\left (b^{n+1}\right) \\ \phi\left ( b^{n}\right) < \phi\left (b^{n+1} \right)\end{array} \right \}  \\
\hline
\left \{a^p,b^{p+1},a^{p+1} \in \sigma^{ss} :                      \begin{array}{c} \phi\left ( a^{p}\right) < \phi\left (b^{p+1}\right) \\ \phi\left ( a^{p}\right)+1 < \phi\left (a^{p+1}\right) \\ \phi\left ( b^{p+1}\right) < \phi\left (a^{p+1}\right)\end{array}  \right \}                            &       \left \{ b^q,a^q,b^{q+1} \in \sigma^{ss} :                      \begin{array}{c} \phi\left ( b^{q}\right) < \phi\left (a^{q}\right) \\ \phi\left ( b^{q}\right)+1 < \phi\left (b^{q+1}\right) \\ \phi\left ( a^{q}\right) < \phi\left (b^{q+1}\right)\end{array}  \right \}   \\   \hline
\left \{ a^{m},a^{m+1}, M \in \sigma^{ss} :                      \begin{array}{c} \phi\left ( a^{m}\right) < \phi\left (a^{m+1}\right) \\ \phi\left ( a^{m}\right) < \phi\left (M\right) \\ \phi\left ( a^{m+1}\right) < \phi\left (M\right)+1\end{array}  \right \}                            &    \left \{ b^{i},b^{i+1}, M' \in \sigma^{ss}:                       \begin{array}{c} \phi\left ( b^{i}\right) < \phi\left (b^{i+1}\right) \\ \phi\left ( b^{i}\right) < \phi\left (M'\right) \\ \phi\left ( b^{i+1}\right) < \phi\left (M'\right)+1\end{array}    \right \}   .    \\ \hline
\end{array}
\end{gather}
    
\begin{gather} \label{middle M} \begin{array}{| c   | c  |}   \hline  
(a^p,M,b^{p+1})   & (b^q,M',a^{q})  \\
\hline
\left \{a^p,M,b^{p+1}  \in \sigma^{ss} :                      \begin{array}{c} \phi\left ( a^{p}\right) < \phi\left (M\right)+1 \\ \phi\left ( a^{p}\right) < \phi\left (b^{p+1}\right) \\ \phi\left ( M\right) < \phi\left (b^{p+1}\right)\end{array}  \right \}                            &       \left \{b^q,M',a^{q} \in \sigma^{ss} :                      \begin{array}{c} \phi\left ( b^{q}\right) < \phi\left (M'\right) +1\\ \phi\left ( b^{q}\right) < \phi\left (a^{q}\right) \\ \phi\left ( M'\right) < \phi\left (a^{q}\right)\end{array}  \right \}   \\  \hline
\end{array}.
\end{gather}

In \cite{dimkatz3, dimkatz4}, it was shown that 
$ \st(D^b(Q )) = \bigcup_{(A,B,C) \in \mk{T}} (A,B,C)$ (see \cite[(92)]{dimkatz4}). 
Following \eqref{exceptional triples} we reorganize the union as follows  (see \cite[(95) and Section 5]{dimkatz4}):    
\begin{gather} \label{st with T} \st(D^b(Q )) = \mk{T}_{a}^{st} \cup (\_,M,\_) \cup (\_,M',\_) \cup \mk{T}_{b}^{st} \qquad \qquad   \mk{T}_{a}^{st} \cap  \mk{T}_{b}^{st} = \emptyset, \end{gather}
where  $\mk{T}_a^{st}$, resp.  $\mk{T}_b^{st}$, is the union of all the subsets described in the first, resp. second, column of table \eqref{left right M} as $j,p,m$, resp.  $n,q,i$, vary in $\ZZ$. The union  of all the subsets described in the first, resp. second, column of table \eqref{middle M} as $p$ vary in $\ZZ$ will be denoted by $(\_, M,  \_)$, resp $(\_, M', \_)$.

More generally, the subset $(A,B,C)\subset \st(\mc T)$ is described in \cite[Proposition 2.7]{dimkatz4} for any $\mc T$ as those stability conditions such that $A,B,C$ are semistable and the phases $\phi(A)$, $\phi(B)$, $\phi(C)$ satisfy the inequalities in \eqref{eqSet} with $y_0,y_1,y_2$  replaced  by $\phi(A), \phi(B), \phi(C)$. 
\subsection{First remarks on $C_{i,\sigma }(\mc T)$, $C_{i,\sigma \sigma}(\mc T)$} \label{first remarks}
From \ref{derived points}, Proposition  \ref{noncomcurves}, and Remark \ref{remark for helix} we  see that:

$\alpha^m \in C_{0,\sigma \sigma}(\mc T)$ iff $\{M',a^m,b^{m+1}\}\subset \sigma^{ss}$; \ \  
$\beta^m \in C_{0,\sigma \sigma}(\mc T)$ iff $\{M,a^m,b^{m}\}\subset \sigma^{ss}$;

$A \in C_{1,\sigma \sigma}(\mc T)$ iff $\{a^m:m\in \ZZ \}\subset \sigma^{ss}$; \ \  $B \in C_{1,\sigma \sigma}(\mc T)$ iff $\{b^m:m\in \ZZ \}\subset \sigma^{ss}$; 

$A \in C_{1,\sigma }(\mc T)$ iff $\abs{\{m\in \ZZ: a^m \in  \sigma^{ss} \} }=\infty $; \ \  $B \in C_{1,\sigma }(\mc T)$ iff $\abs{\{m\in \ZZ: b^m \in  \sigma^{ss} \} }=\infty$
.

In particular, it follows  (taking into account  \eqref{derived points}) 
\begin{gather} \label{a criterion} C_{D^b(pt),\sigma \sigma}(\mc T)=C_{D^b(pt),\KK}(\mc T) \iff C_{0,\sigma \sigma}(\mc T)=C_{0}(\mc T).  \\
\label{first equality of} \left \{  \sigma :  \abs{C_{D^b(pt), \sigma \sigma}(\mc T)}<\infty \right \} =\left \{  \sigma :  \abs{C_{1, \sigma }(\mc T)}=0 \right \}.\end{gather}
Furthermore, if infinitely many genus zero nc curves are in $C_{0,\sigma \sigma}(\mc T)$, then either for infinitely many $m$ we have $\alpha^m \in C_{0,\sigma \sigma}(\mc T)$ and hence   $\{M',a^m,b^{m+1}\}\subset \sigma^{ss}$ or for infinitely many $m$ we have $\beta^m \in C_{0,\sigma \sigma}(\mc T)$ and hence  $\{M,a^m,b^{m}\}\subset \sigma^{ss}$, and it follows that  $C_{1,\sigma}(\mc T)=C_{1}(\mc T)$. Thus we deduce:
\begin{gather} \label{finitenes of genus zero stables}
C_{1,\sigma}(\mc T) \neq \{A,B\} \ \ \Rightarrow \ \ \abs{C_{0,\sigma \sigma}(\mc T)}<\infty.
\end{gather}

\subsection{Study of $C_{i,\sigma}(\mc T)$, $C_{i,\sigma \sigma}(\mc T)$, $C_{D^b(pt),\sigma \sigma}(\mc T)$ as $\sigma \in \mk{T}_a^{st} \cup \mk{T}_b^{st}$ and $i\geq  -1$} \label{subsection for Tab} Recall that $\mk{T}_a^{st} \cup \mk{T}_b^{st}$ is in   \eqref{st with T}.    It will be enough to study only one of  both the cases, $\sigma \in \mk{T}_a^{st}$ or $\sigma \in \mk{T}_b^{st}$, because  there is an  
auto-equivalence $\zeta$ such that $\zeta(\mk{T}_a^{st})=\mk{T}_b^{st}$. More precisely, there is $\zeta$ such that 
\begin{gather} \label{useful autoequivalence} \zeta \in {\rm Aut}_{\KK}(\mc T) \qquad \zeta(a^m)\cong b^m \ \ \zeta(b^m) \cong a^{m-1} \ \ \zeta(M)\cong M' \ \ \zeta(M')\cong M \end{gather}
for each $m\in \ZZ$. 
Indeed from \cite[Proposition 8.4]{dimkatz2} we see that there is an auto-equivalence $\zeta \in {\rm Aut}(\mc T)$ such that $\zeta(a^m)\cong b^m$, $\zeta(b^m)\cong {a}^{m-1}$ for each $m\in \ZZ$. On the other hand,  from the full exceptional triples $(M',a^m,a^{m+1})$,   $(M,b^m,b^{m+1})$, $(a^m,a^{m+1}, M)$,   $(b^m,b^{m+1}, M')$ it follows that $\zeta(M')\cong M[p]$, $\zeta(M)\cong M'[q]$ for some $p.q \in \ZZ$. However, due to Corollary \ref{nonvanishings} we have also $0\neq \hom(M',a^m) =\hom(\zeta(M'), \zeta(a^m))=\hom(M[p], b^m)$    and from Corollaries \ref{nonvanishings}, \ref{one nonvanishing degree} we see that $p=0$. Also we have  $0\neq \hom(M,b^m) =\hom(\zeta(M), \zeta(b^m))=\hom(M'[q], a^{m-1})$  and again  Corollaries \ref{nonvanishings}, \ref{one nonvanishing degree} imply $q=0$. From Proposition \ref{noncomcurves} and \eqref{useful autoequivalence} we obtain 
\begin{gather} \label{zeta on alpha beta}
\zeta(\alpha^m)=\beta^m \qquad \zeta(\beta^m)=\alpha^{m-1}  \qquad \zeta(A) = B \qquad \zeta(B) = A.
\end{gather}
 Using \eqref{aut properties 1} in  Remark \ref{action} and Subsection \ref{first remarks} we see that 
\begin{remark} \label{action1} For any $\sigma \in \st(\mc T)$ and any $m\in \ZZ$ we have 
	\begin{gather} M \in \sigma^{ss} \iff M' \in (\zeta \cdot \sigma)^{ss} \ \ \Rightarrow \phi_{\zeta \cdot \sigma}(M')= \phi_{ \sigma}(M) \\
	 M' \in \sigma^{ss} \iff M \in (\zeta \cdot \sigma)^{ss} \ \ \Rightarrow \phi_{\zeta \cdot \sigma}(M)= \phi_{ \sigma}(M') \\
	 a^m \in \sigma^{ss} \iff b^m \in (\zeta \cdot \sigma)^{ss} \ \ \Rightarrow \phi_{\zeta \cdot \sigma}(b^m)= \phi_{ \sigma}(a^m) \\ b^m \in \sigma^{ss} \iff a^{m-1} \in (\zeta \cdot \sigma)^{ss} \ \ \Rightarrow \phi_{\zeta \cdot \sigma}(a^{m-1})= \phi_{ \sigma}(b^m) \\ \alpha^m \in C_{0,\sigma \sigma}(\mc T) \iff \beta^m \in C_{0,(\zeta \cdot \sigma) ( \zeta \cdot \sigma)}(\mc T)\\ \beta^m \in C_{0,\sigma \sigma}(\mc T) \iff \alpha^{m-1} \in C_{0,(\zeta \cdot \sigma) ( \zeta \cdot \sigma)}(\mc T)\\ A \in C_{1,\sigma \sigma}(\mc T) \iff B \in C_{1, (\zeta \cdot \sigma) ( \zeta \cdot \sigma)}(\mc T) \ \ A \in C_{1,\sigma}(\mc T) \iff B \in C_{1, \zeta \cdot \sigma}(\mc T).\end{gather}
	
\end{remark}

\begin{lemma} \label{no semistable}  Let $p\in \ZZ$. In any of the following cases holds $C_{1,\sigma \sigma}(\mc T)=C_{1,\sigma}(\mc T)=\emptyset$:
	
{\rm	(a)} $a^p, a^{p+1} \in \sigma^{ss}$ and $\phi(a^{p+1})> \phi(a^{p})+1$. In fact in this case $a^j\not \in \sigma^{ss}$ for $j\not \in \{p,p+1\}$, $b^{j}\not \in \sigma^{ss}$ for $j\neq p+1$,  and it follows that $ C_{0,\sigma \sigma}(\mc T)\subset\{\alpha^{p}, \beta^{p+1}\}$,   $\beta^{p+1} \in  C_{0,\sigma \sigma}(\mc T)$ iff $\{M, b^{p+1}\} \subset \sigma^{ss}$,  $\alpha^{p} \in  C_{0,\sigma \sigma}(\mc T)$ iff $\{M', b^{p+1}\} \subset  \sigma^{ss}$.
	
{\rm	(b)} $b^p, b^{p+1} \in \sigma^{ss}$ and $\phi(b^{p+1})> \phi(b^{p})+1$. In  this case $b^j\not \in \sigma^{ss}$ for $j\not \in \{p,p+1\}$, $a^{j}\not \in \sigma^{ss}$ for $j\neq p$,   and it follows that $ C_{0,\sigma \sigma}(\mc T)\subset\{\alpha^{p}, \beta^{p}\}$, $\beta^{p} \in  C_{0,\sigma \sigma}(\mc T)$ iff $\{M, a^{p}\} \subset \sigma^{ss}$,    $\alpha^{p} \in  C_{0,\sigma \sigma}(\mc T)$ iff $\{M', a^{p}\} \subset \sigma^{ss}$.

{\rm	(c)}  $a^p, b^{p+1} \in \sigma^{ss}$ and $\phi(b^{p+1})> \phi(a^{p})+1$. In fact, in this case $b^j, a^j\not \in \sigma^{ss}$ for $j\not \in \{p,p+1\}$.

\end{lemma} 
\begin{proof} (a) If $b^j\in\sigma^{ss}$ for $j\leq p$, then from  $\hom(b^j, a^p)\neq 0$ in \eqref{nonvanishing3} and axiom (2) in Definition  \ref{slicing} it follows that $\phi(b^j)\leq \phi(a^p)$, and therefore axiom (2) in Definition  \ref{slicing} implies $\hom^1(a^{p+1},b^j)=0$, which contradicts \eqref{nonvanishing4}. 
	
If $b^j\in\sigma^{ss}$ for $j > p+1$,  then from  $\hom(a^{p+1}, b^j)\neq 0$ in \eqref{nonvanishing4} and axiom (2) in Definition  \ref{slicing} it follows that $\phi(a^{p+1})\leq \phi(b^j)$, and therefore axiom (2) in Definition  \ref{slicing} implies $\hom^1(b^{j},a^p)=0$, which contradicts \eqref{nonvanishing3}. 	
The cases of $a^j$ for $j < p$, $j>p+1$ are obtained by similar arguments,  relying on  \eqref{nonvanishing5}.

The statement for $C_{0,\sigma \sigma}(\mc T)$ follows from the already proved part and Subsection \ref{first remarks}. 

(b) follows from (a) due to the auto-equivalence $\zeta$ with \eqref{useful autoequivalence} and \eqref{zeta on alpha beta}.

(c) If $a^j \in \sigma^{ss}$ for some $j<p$, then using Corollary \ref{nonvanishings} and axiom (2) in Definition \ref{slicing} we deduce $\phi(a^j)\leq \phi(a^p)$ and therefore $\phi(a^j)+1 < \phi(b^{p+1})$. Hence vie axiom (2) in Definition \ref{slicing} we have $\hom^1(b^{p+1}, a^j)=0$, which contradicts \eqref{nonvanishing3}. 

If $a^j \in \sigma^{ss}$ for some $j>p+1$, then by $\hom(b^{p+1},a^j)\neq 0$ in \eqref{nonvanishing3}  we get $\phi(b^{p+1})\leq \phi(a^j)$ and therefore $\hom^1(a^j,a^p)=0$, which contradicts \eqref{nonvanishing5}.

If $b^j \in \sigma^{ss}$ for some $j<p$, then  \ref{nonvanishing3} ensures $\phi(b^j)\leq \phi(a^p)$, which however implies  $\hom^1(b^{p+1}, b^j)=0$, which contradicts \eqref{nonvanishing6}. 

If $b^j \in \sigma^{ss}$ for some $j>p+1$, then by $\hom(b^{p+1},b^j)\neq 0$ in \eqref{nonvanishing6}  we get $\phi(b^{p+1})\leq \phi(b^j)$ and therefore $\hom^1(b^j,a^p)=0$, which contradicts \eqref{nonvanishing3}.

\end{proof}

We will prove Propositions \ref{proposition for  B in Tb}, \ref{more careful study}  and \ref{more careful part two}. From these follow  tables  \eqref{table for b b M'}, \eqref{table for M b b}  and then  using Remark \ref{action1} follows  tables  \eqref{table for a a M}, \eqref{table for M' a a}.

We will prove also  Proposition \ref{proposition for  M,Mp in Tb}, from which follows
\begin{corollary} \label{phases of MMprime1} If $\sigma \in \mk{T}_b^{st}$ and $M,M'\in \sigma^{ss}$, then $\phi(M) <\phi(M')$. If $\sigma \in \mk{T}_a^{st}$ and $M,M'\in \sigma^{ss}$, then $\phi(M) >\phi(M')$. 
	\end{corollary}

Now we start proving tables \eqref{table for b b M'}, \eqref{table for M b b}. We first determine completely  the behavior of the stability and semi-stability of the curve $B$ in $\mk{T}_b^{st}$, and mention something about $A$.

\begin{proposition} \label{proposition for  B in Tb} Let $\sigma \in \mk{T}_b^{st}$. 
	
	If $\sigma \in (b^q,a^q,b^{q+1})$ for some $q$, then $C_{1,\sigma}(\mc T)=C_{1,\sigma \sigma}(\mc T)=\emptyset$. 
	
	If $\sigma \in (M,b^q,b^{q+1})$ or $\sigma \in (b^q,b^{q+1}, M')$  for some $q$, then for $\phi(b^{q+1})> \phi(b^{q})+1$ we have $C_{1,\sigma \sigma}(\mc T)=C_{1,\sigma}(\mc T)=\emptyset$ and otherwise for  $\phi(b^{q+1})\leq \phi(b^{q})+1$ holds  $B\in C_{1,\sigma \sigma}(\mc T)$.

\end{proposition}
\begin{proof}
All the cases are in the second column of table  \eqref{left right M} and therefore   $b^q, b^{q+1}\in \sigma^{ss}$.  From Lemma \ref{no semistable} it follows that $A\not \in C_{1,\sigma}(\mc T)$ and $B\not \in C_{1,\sigma}(\mc T)$, when $\phi(b^{q+1})> \phi(b^{q})+1$, furthermore 	for $\sigma \in (b^q,a^q,b^{q+1})$  we see in table \eqref{left right M} that $\phi(b^{q+1})> \phi(b^q)+1$ is automatically satisfied. 

	It remains the case when $\sigma \in (M,b^q,b^{q+1})$ or $\sigma \in (b^q,b^{q+1}, M')$  for some $q$ and $\phi(b^{q+1})\leq \phi(b^{q})+1$. In this case Corollary \ref{stable curve} ensure  $B\in C_{1,\sigma \sigma}(\mc T)$. 
 \end{proof}
 
Now we discuss the behavior of  $\phi_{\sigma}(M)$, $\phi_{\sigma}(M')$ for some stability conditions. 
 \begin{proposition} \label{proposition for  M,Mp in Tb} Let $p \in \ZZ$. 
 	
 {\rm (a) }	Let $\sigma \in (b^p,a^p,b^{p+1})$. If $M,M'\in \sigma^{ss}$, then $\phi(M)<\phi(M')\leq \phi(M)+1$. Furthermore:  
 	\begin{gather}   M \in \sigma^{ss} \iff \phi(a^p)-1 \leq \phi(b^p)< \phi\left (a^{p}\right) \nonumber  \\[-3mm] \label{some inequalities} \\[-3mm]
 	  M' \in \sigma^{ss} \iff \phi(b^{p+1})-1 \leq \phi(a^p)< \phi\left (b^{p+1}\right).  \nonumber \end{gather} 
  For example one can choose $\sigma \in  (b^p,a^p,b^{p+1})$ with  $\phi(b^p)=t$, $\phi(a^p)=t+1$, $\phi(b^{p+1})=t+2$ and in this case both inequalities in \eqref{some inequalities} hold and $\phi(M)=t$, $\phi(M')= t+1$.

  {\rm (b) }	Let $\sigma \in (M,b^p,b^{p+1})$ or $\sigma \in (b^p,b^{p+1}, M')$. If $M,M'\in \sigma^{ss}$ then $\phi(M)<\phi(M')$.
  
\end{proposition}
 \begin{proof} (a) Let $M, M' \in \sigma^{ss}$.  Combining the inequalities $  \begin{array}{c} \phi\left ( b^{p}\right) < \phi\left (a^{p}\right) \\ \phi\left ( b^{p}\right)+1 < \phi\left (b^{p+1}\right) \\ \phi\left ( a^{p}\right) < \phi\left (b^{p+1}\right)\end{array}$ in   \eqref{left right M} with the following inequalities due to  \eqref{nonvanishing1}, \eqref{nonvanishing2}:   	
 	$\begin{array}{c} \phi(b^{p+1})-1 \leq \phi(M')\leq \phi(a^p)   \\  \phi(a^p)-1 \leq \phi(M)\leq \phi(b^p) \end{array} $ we 
 	we obtain the following equivalent system of inequalities: $\phi(a^p)-1 \leq \phi(M)\leq \phi(b^p) <  \phi(b^{p+1})-1 \leq \phi(M')\leq \phi(a^p)$.
 		 In particular, $\phi(M)<\phi(M')\leq \phi(M)+1$, if $M,M'\in \sigma^{ss}$. 
 	
 	Conversely,  let $\sigma \in (b^p,a^p,b^{p+1})$ and some of the inequalities  \eqref{some inequalities} hold.  	
 	 The constants $\alpha, \gamma$ defined in \cite[Lemma 2.7]{dimkatz4} for the triple  $(b^p,a^p,b^{p+1})$ are $\alpha=\gamma=-1$, which follows from Corollary \ref{nonvanishings}.  From Corollary \ref{stable curve} we see that if the first, resp. the second, inequality  in \eqref{some inequalities} holds, then the exceptional objects in the extension closures of $(b^p,a^p[-1])$, resp. in $(a^p,b^{p+1}[-1])$, are in $\sigma^{ss}$. Therefore $M\in \sigma^{ss}$, resp. $M'\in \sigma^{ss}$ (see \eqref{short filtration 1}).  The last sentence follows also from \eqref{short filtration 1}.

(b) Let $M,M'\in \sigma^{ss}$.  If $\sigma \in (M,b^p,b^{p+1})$, then from table \eqref{left right M} we have $\phi(M)+1 < \phi(b^{p+1})$, and from \eqref{nonvanishing2} it follows that $\phi(b^{p+1})\leq \phi(M')+1$. Therefore $\phi(M)< \phi(M')$. If $\sigma \in (b^p,b^{p+1}, M')$, then from table \eqref{left right M} we have $ \phi(b^{p})< \phi(M')$, and from \eqref{nonvanishing1} it follows that $\phi(M)\leq \phi(b^{p})$. Therefore $\phi(M)< \phi(M')$.
 \end{proof}
 \begin{corollary} \label{genus 0 in bab}
 	Let $\sigma \in (b^p,a^p,b^{p+1})$, then $C_{0,\sigma \sigma}(\mc T)\subset \{\alpha^p,\beta^p\}$ and $\alpha^p \in C_{0,\sigma \sigma}(\mc T)$ iff $\phi(b^{p+1})-1 \leq \phi(a^{p})< \phi(b^{p+1})$,  $\beta^p \in C_{0,\sigma \sigma}(\mc T)$ iff $\phi(a^{p})-1 \leq \phi(b^{p})<\phi(a^{p})$. In particular, the first row of table \eqref{table for b a b} follows, taking into account also Proposition \eqref{proposition for  B in Tb}.
 	
 		Let $\sigma \in (a^p,b^{p+1},a^{p+1})$, then $C_{0,\sigma \sigma}(\mc T)\subset \{\alpha^p,\beta^{p+1}\}$ and $\beta^{p+1} \in C_{0,\sigma \sigma}(\mc T)$ iff $\phi(a^{p+1})-1 \leq \phi(b^{p+1})$,  $\alpha^p \in C_{0,\sigma \sigma}(\mc T)$ iff $\phi(b^{p+1})-1 \leq \phi(a^{p})$. In particular, the second  row of table \eqref{table for b a b} follows.
 \end{corollary}
 \begin{proof}
	Let $\sigma \in (b^p,a^p,b^{p+1})$. In table \ref{left right M} we have $b^p,a^p,b^{p+1} \in \sigma^{ss}$ and$  \begin{array}{c} \phi\left ( b^{p}\right) < \phi\left (a^{p}\right) \\ \phi\left ( b^{p}\right)+1 < \phi\left (b^{p+1}\right) \\ \phi\left ( a^{p}\right) < \phi\left (b^{p+1}\right)\end{array}$.	From Lemma \ref{no semistable} (b) and  Proposition \ref{proposition for  M,Mp in Tb} (a) we obtain the statement. 
	
	The second part follows from the first with the help of \eqref{useful autoequivalence}, \eqref{zeta on alpha beta} and Remark \ref{action1}.
 \end{proof}

  \begin{corollary} \label{coro for b b M'} Let $\sigma \in (b^p,b^{p+1}, M')$ and $\phi(b^{p+1})>\phi(b^p)+1$. Then
  	
  	{\rm (a)}  $ C_{0,\sigma \sigma}(\mc T) \subset \{\alpha^p, \beta^{p}\}$ and
  	\begin{gather} \label{equivalences for nonempty Co}  C_{0,\sigma \sigma}(\mc T)\not = \emptyset \iff \phi\left (M'\right)+1 > \phi(b^{p+1}) \geq \phi(M') \iff \alpha^p \in  C_{0,\sigma \sigma}(\mc T)\iff a^p \in \sigma^{ss}. \end{gather} 
  	Furthermore for $ \phi\left (M'\right)+1 > \phi(b^{p+1}) \geq \phi(M') $ we have  $\phi(a^p)=\arg_{[\phi(M'), \phi(M')+1)}(Z(b^{p+1})+ Z(M'))$ and $\beta^p \in C_{0,\sigma \sigma}(\mc T) \iff  \arg_{[\phi(M'), \phi(M')+1)}(Z(b^{p+1})+ Z(M')) \leq   \phi(b^{p})+1 $.
  	
  	{\rm (b)}  $M\in \sigma^{ss}$ if and only if $\beta^p \in C_{0,\sigma \sigma}(\mc T)$, which in turn by (a)  is equivalent to  $ \phi\left (M'\right)+1 > \phi(b^{p+1}) \geq \phi(M') $ and $ \arg_{[\phi(M'), \phi(M')+1)}(Z(b^{p+1})+ Z(M')) \leq   \phi(b^{p})+1 $.
  	
  {\rm (c)}	From {\rm (a)}  and {\rm (b)}  follow the third  and the second  row of table \eqref{table for b b M'} (taking into account Proposition \ref{proposition for  B in Tb}).
  \end{corollary}	
  \begin{proof} Recall that in table  \eqref{left right M} we have $\phi\left ( b^{p+1}\right) < \phi\left (M'\right)+1$.
  	From Lemma \ref{no semistable} (b) it follows that  $ C_{0,\sigma \sigma}(\mc T)\subset\{\alpha^{p}, \beta^{p}\}$, $\beta^{p} \in  C_{0,\sigma \sigma}(\mc T)$ iff $\{M, a^{p}\} \subset \sigma^{ss}$,    $\alpha^{p} \in  C_{0,\sigma \sigma}(\mc T)$ iff $\{M', a^{p}\} \subset \sigma^{ss}$. In particular $ C_{0,\sigma \sigma}(\mc T)\not = \emptyset$ iff $a^p \in \sigma^{ss}$. On the other hand,  if $a^p \in \sigma^{ss}$, then  \eqref{nonvanishing1},  \eqref{nonvanishing4}  imply that $\phi(M')\leq \phi(a^p)\leq \phi(b^{p+1})$ and vice versa, if $\phi(b^{p+1}) \geq  \phi(M') $, then Corollary \ref{stable curve} and \eqref{nonvanishing2} imply that $\alpha^p = \langle M', a^p, b^{p+1} \rangle \in C_{0,\sigma \sigma}(\mc T)$ and hence $a^p \in \sigma^{ss}$. So we proved \eqref{equivalences for nonempty Co}. Before continuing with analyzing the case   $\phi(b^{p+1}) \geq \phi(M')$, we note that when $\phi(b^{p+1}) < \phi(M')$ we have $ M\not \in \sigma^{ss}$. Indeed, the non-vanishing $\hom(M, b^p)\neq 0$ in \eqref{nonvanishing1}  combined with inequalities  $\phi(b^{p+1}) < \phi(M')$, $\phi(b^{p+1})>\phi(b^p)+1$, and $M\in \sigma^{ss}$, recalling (2) in Definition \ref{slicing}, would imply $\hom^1(M', M)=0$, which contradicts \eqref{nonvanishing7}.  Recalling that $\beta^{p} \in  C_{0,\sigma \sigma}(\mc T)$ iff $\{M, a^{p}\} \subset \sigma^{ss}$ and using the already proved  \eqref{equivalences for nonempty Co} we see that $\beta^{p} \in  C_{0,\sigma \sigma}(\mc T)$ iff $M \in \sigma^{ss}$.

  	Let $\phi(b^{p+1}) \geq \phi(M')$ and $\phi(b^{p+1})>\phi(b^p)+1$. It remains to determine when  $\beta^p \in C_{0,\sigma \sigma}(\mc T)$, i.e. when $M\in \sigma^{ss}$.   We already explained that $a^p \in \sigma^{ss}$ and   $\phi(M')\leq \phi(a^p)\leq \phi(b^{p+1})$. In table \eqref{left right M} we have also   $\begin{array}{c} \phi\left ( b^{p}\right) < \phi\left (b^{p+1}\right) \\ \phi\left ( b^{p}\right) < \phi\left (M'\right) \\ \phi\left ( b^{p+1}\right) < \phi\left (M'\right)+1\end{array}$. Therefore $\phi(M')\leq \phi(a^p)\leq \phi(b^{p+1})<\phi(M')+1$. Recalling that $Z(a^p)= Z(b^{p+1})+ Z(M')$  (the first triangle in \eqref{short filtration 1}) we deduce that  $\phi(a^p)=\arg_{[\phi(M'), \phi(M')+1)}(Z(b^{p+1})+ Z(M'))$, Furthermore,  $\phi(a^p)= \phi(b^{p+1})$ iff $\phi(b^{p+1}) = \phi(M')$ otherwise $\phi(M')<\phi(a^p)<\phi(b^{p+1})$. If $\phi(b^{p+1}) = \phi(M')$,  then  $\phi(a^p)= \phi(b^{p+1})$ and therefore, taking into account $\phi(b^p)+1 < \phi(b^{p+1})$, we deduce $ \phi(a^{p}[-1])>\phi(b^p)$ and then the second triangle in \eqref{short filtration 1} implies $M \not \in \sigma^{ss}$. Otherwise   $\phi\left ( b^{p}\right) < \phi(M')<\phi(a^p)<\phi(b^{p+1})$ and from table \eqref{left right M} if follows that $\sigma \in (b^p, a^p, b^{p+1})$ and now applying  Corollary \ref{genus 0 in bab} we complete the proof.   
  \end{proof}

 \begin{corollary} \label{coro for M b b} Let $\sigma \in (M,b^p,b^{p+1})$ and $\phi(b^{p+1})>\phi(b^p)+1$. Then
 	
 	{\rm (a)}  $ C_{0,\sigma \sigma}(\mc T) \subset \{\alpha^p, \beta^{p}\}$ and
 	\begin{gather} \label{equivalences fo C_0 etc}  C_{0,\sigma \sigma}(\mc T)\not = \emptyset \iff  \phi\left (b^{p}\right) > \phi(M)\geq \phi(b^p)-1 \iff \beta^p \in  C_{0,\sigma \sigma}(\mc T)\iff a^p \in \sigma^{ss}. \end{gather} 
 	Furthermore for $ \phi\left (b^{p}\right) > \phi(M)\geq \phi(b^p)-1$ we have $\phi(a^p)=\arg_{(\phi(M), \phi(M)+1]}(Z(b^p)- Z(M))$ and $\alpha^p \in C_{0,\sigma \sigma}(\mc T) \iff \phi(b^{p+1})-1 \leq   \arg_{(\phi(M), \phi(M)+1]}(Z(b^p)- Z(M))$. 
 	
 	{\rm (b)}  $M'\in \sigma^{ss}$ if and only if $\alpha^p \in C_{0,\sigma \sigma}(\mc T)$, which in turn by (a)  is equivalent to  $ \phi\left (b^{p}\right) > \phi(M)\geq \phi(b^p)-1 $ and $ \phi(b^{p+1})-1 \leq   \arg_{(\phi(M), \phi(M)+1]}(Z(b^p)- Z(M)) $.
 	
 	{\rm (c)}	From 	{\rm(a)} and {\rm(b)} follow the second and the third row of table \eqref{table for M b b} (taking into account Proposition \ref{proposition for  B in Tb}).
 \end{corollary}	
 \begin{proof} In table \eqref{left right M} we have  $\phi\left (M\right) < \phi\left (b^{p}\right) $.
 	From Lemma \ref{no semistable} (b) it follows that  $ C_{0,\sigma \sigma}(\mc T)\subset\{\alpha^{p}, \beta^{p}\}$, $\beta^{p} \in  C_{0,\sigma \sigma}(\mc T)$ iff $\{M, a^{p}\} \subset \sigma^{ss}$,    $\alpha^{p} \in  C_{0,\sigma \sigma}(\mc T)$ iff $\{M', a^{p}\} \subset \sigma^{ss}$. In particular $ C_{0,\sigma \sigma}(\mc T)\not = \emptyset$ iff $a^p \in \sigma^{ss}$. On the other hand,  if $a^p \in \sigma^{ss}$, then  \eqref{nonvanishing2},  \eqref{nonvanishing3} and \eqref{nonvanishing4} imply that $\phi(b^p)\leq \phi(a^p)\leq \phi(M)+1$,  $\phi(a^p)\leq \phi(b^{p+1})$ and vice versa, if $\phi(M)\geq \phi(b^p)-1$, then Corollary \ref{stable curve} implies that $\beta^p = \langle M, b^p, a^p \rangle \in C_{0,\sigma \sigma}(\mc T)$ and hence $a^p \in \sigma^{ss}$.  Thus, we proved \eqref{equivalences fo C_0 etc}.  
 	
 	 Note that when $ \phi(M)< \phi(b^{p}) -1$  we have $ M'\not \in \sigma^{ss}$. Indeed, the non-vanishing $\hom^1( b^{p+1}, M')\neq 0$ in \eqref{nonvanishing2}  combined with inequalities  $\phi(b^{p})-1 > \phi(M)$, $\phi(b^{p+1})>\phi(b^p)+1$, and $M'\in \sigma^{ss}$, recalling (2) in Definition \ref{slicing}, would imply $\hom^1(M', M)=0$, which contradicts \eqref{nonvanishing7}. Having \eqref{equivalences fo C_0 etc} and  $\alpha^{p} \in  C_{0,\sigma \sigma}(\mc T)$ iff $\{M', a^{p}\} \subset \sigma^{ss}$ we see that $M'\in \sigma^{ss} \iff \alpha^{p} \in  C_{0,\sigma \sigma}(\mc T)$. 
 	
 	Let $\phi(M)\geq \phi(b^p)-1 $ and $\phi(b^{p+1})>\phi(b^p)+1$. It remains to determine when  $\alpha^p \in C_{0,\sigma \sigma}(\mc T)$, i.e. when $M'\in \sigma^{ss}$ ?   We already explained that $a^p \in \sigma^{ss}$ and  $\phi(b^p)\leq \phi(a^p)\leq \phi(M)+1$,  $\phi(a^p)\leq \phi(b^{p+1})$. In table \eqref{left right M} we have also $\begin{array}{c}\phi\left (M\right) < \phi\left (b^{p}\right) \\ \phi\left ( M\right)+1 < \phi\left (b^{p+1}\right) \\ \phi\left ( b^{p}\right) < \phi\left (b^{p+1} \right) \end{array}$. Therefore $\phi(M)< \phi\left (b^{p}\right)\leq \phi(a^p)\leq \phi(M)+1$. Recalling that $Z(a^p)= Z(b^p)- Z(M)$  (the second triangle in \eqref{short filtration 1}) we deduce that  $\phi(a^p)=\arg_{(\phi(M), \phi(M)+1]}(Z(b^p)- Z(M))$, Furthermore,  $\phi(a^p)= \phi(b^p)$ iff $\phi(b^p)-1 = \phi(M)$ otherwise $\phi(b^p)<\phi(a^p)< \phi(M)+1<\phi(b^{p+1})$. If $\phi(b^p)-1 = \phi(M)$,  then  $\phi(a^p)= \phi(b^p)$ and therefore, taking into account $\phi(b^p)+1 < \phi(b^{p+1})$, we deduce $ \phi(b^{p+1}[-1])>\phi(a^p)$ and then the first triangle in \eqref{short filtration 1} implies $M'\not \in \sigma^{ss}$. Otherwise    $\phi(b^p)<\phi(a^p)< \phi(M)+1<\phi(b^{p+1})$ and from table \eqref{left right M} if follows that $\sigma \in (b^p, a^p, b^{p+1})$ and now applying  Corollary \ref{genus 0 in bab} we complete the proof.   
 \end{proof}

 In Proposition \ref{proposition for  B in Tb} and Corollaries \ref{coro for M b b}, \ref{coro for b b M'}  we filled  in some boxes of tables \eqref{table for b b M'}, \eqref{table for M b b}. 
 The rest boxes in these tables are proved in  the next two propositions.

 \begin{proposition} \label{more careful part two} Let $\sigma \in (b^n,b^{n+1}, M')$ for some $n\in \ZZ$ and  $\phi(b^{n+1})\leq \phi(b^{n})+1$. Then  $A\not \in C_{1,\sigma \sigma}(\mc T)$.
 		
 		{\rm (a)} If    $\phi(b^{n+1}) < \phi(b^{n})+1$, we denote\footnote{Recall that $Z(b^n)-Z(b^{n+1})=Z(\delta)$ - see Remark \ref{remark for delta}.} \begin{gather} \label{the case of} t=\arg_{(\phi(b^{n})-1,\phi(b^{n}))}(Z(b^n)-Z(b^{n+1})), \end{gather}  and then consider two subcases:
 		
 	{\rm (a.1)}	  $\phi(M')\geq t+1$: in this sub-case $a^j\not \in \sigma^{ss}$ for all $j\in \ZZ$ and $C_{0,\sigma \sigma}(\mc T)=\emptyset$ and  $A\not \in C_{1,\sigma}(\mc T)$; 
 	
 		{\rm (a.2)}	  $\phi(M')< t+1$: in this sub-case if we denote $U=\arg_{(t,t+1)}(Z(M')+Z(b^{n+1})-Z(b^n))$, then it holds $U>\phi(M')$, $M\in \sigma^{ss}$, $U=\phi(M)+1$ and there are uniquely determined  integers $N\leq u$ such that \begin{gather} \label{tbmprime}   t < \phi(b^N)<\phi(M')\leq \phi(b^{N+1})<t+1\\ \label{small u and big U}
 		t< \phi(b^u)\leq U < \phi(b^{u+1}) < t+1,
 		 \end{gather} 
 		 and for these integers $ N, u$  holds 
 		 $ a^j\in \sigma^{ss} \iff N\leq j \leq u $ and $C_{0,\sigma \sigma}(\mc T)=\{\alpha^j, \beta^j: N\leq j \leq u\}$.  In particular here we have   $A\not \in C_{1,\sigma}(\mc T)$ again. 
 		 
 		{\rm (b) } If    $\phi(b^{n+1}) = \phi(b^{n})+1$, then we  consider three  sub-cases
 		
 		{\rm (b.1) }  $\phi(M')> \phi(b^{n})+1$: in this case $a^j\not \in \sigma^{ss}$ for all $j\in \ZZ$ and $A\not \in C_{1,\sigma}(\mc T)$, $C_{0,\sigma \sigma}(\mc T)=\emptyset$. Furthermore, in this case $M\not \in \sigma^{ss}$.
 		
 			{\rm (b.2) }  $\phi(M')< \phi(b^{n})+1$: in this case $a^j\in  \ \sigma^{ss}$ iff $j=n$, $A\not \in C_{1,\sigma}(\mc T)$, $C_{0,\sigma \sigma}(\mc T)=\{\alpha^n, \beta^n \}$. Furthermore, in this case $M\in \sigma^{ss}$ and  $\phi(M')<\phi(M)+1$.
 		
 			{\rm (b.3) } $\phi(M')=  \phi(b^{n+1})=\phi(b^n)+1$:  in this case     $M\in \sigma^{ss}$, $\phi(M')=\phi(M)+1$, $j\geq n$ iff  $a^j\in \sigma^{ss}$,   $C_{0,\sigma \sigma}(\mc T)=\{\alpha^j, \beta^j: j\geq n\}$,  $A \in C_{1,\sigma}(\mc T)$.

 		 In particular, follows the rest of table \eqref{table for b b M'} and then from Remark \ref{useful autoequivalence} follows table \eqref{table for a a M},
 \end{proposition}	
 \begin{proof} From   Proposition \ref{proposition for  B in Tb} we see that $\{b^j\}_{j\in \ZZ}\subset \sigma^{ss}$.  	From Corollary \ref{nonvanishings} we see that $\hom(b^{j},a^{j})\neq 0$, $\hom(a^{j},b^{j+1})\neq 0$, $\hom( M', a^j)\neq 0$ for any $j\in \ZZ$  and it follows \begin{gather} \label{help inequalities 3'} \forall j \in \ZZ \qquad a^{j} \in \sigma^{ss}  \ \Rightarrow \ \phi(b^{j})\leq \phi(a^{j})\leq \phi(b^{j+1}) \ \ \mbox{and} \ \ \phi(M') \leq \phi(a^j). \end{gather}	
 	
  So, if $a^j \in \sigma^{ss}$ for some $j\in \ZZ$,  then  $\phi(M') \leq \phi(a^j) $ and from table \eqref{left right M} it follows that $\phi(b^n)< \phi(a^{j})$, hence $\hom(a^{j},b^n)=0$. Now   the non-vanishing $ \hom(a^p,b^{q+1})\neq 0$ for $p\leq q$ in   corollary \ref{nonvanishings} implies that $j \geq n $. So we see that  	
 	\begin{gather} \label{non-ss a for big j}  j<n  \ \ \Rightarrow  \ \  a^j \not \in \sigma^{ss} \ \ \Rightarrow \ \ A\not \in C_{1,\sigma \sigma}(\mc T).
 	\end{gather} 
 	
(a)  Here we  have 	$\phi(b^{n})<\phi(b^{n+1})< \phi(b^{n})+1$. Recalling that (see \cite[(84)]{dimkatz4}) $Z(\delta) = Z(b^n)-Z(b^{n+1})$ we see that  choosing  $t$ as in  \eqref{the case of} and  combining the facts: $\{b^j:j\in \ZZ\}\subset \sigma^{ss}$, $\hom(b^i,b^j)\neq 0$ for $i\leq j$, $\hom^1(b^i,b^j)\neq 0$ for $i>j+1$ and \cite[Corollaries 3.19, 3.18]{dimkatz4} one obtains:  
 	\begin{gather} \label{inequalities help 2'}  0\neq Z(\delta)=\abs{Z(\delta)} \exp({\rm i} \pi t)  \ \lim_{j\rightarrow + \infty } \phi(b^j)  = t+1 \    \lim_{j\rightarrow -\infty } \phi(b^j) = t \qquad  \forall j \in \ZZ \ \ \phi(b^j) < \phi(b^{j+1}).
 	\end{gather} 
 	Having \eqref{inequalities help 2'} and \eqref{help inequalities 3'} and using again \cite[Corollary 3.19]{dimkatz4} we deduce:
 	\begin{gather}  \{Z(a^j)\}_{j\in \ZZ} \subset Z(\delta)^c_+ \ \    \lim_{k\rightarrow +\infty } \arg_{(t,t+1)}(Z(a^j)) = t +1 \nonumber  \\[-2mm] \label{inequalities help 5'}  \\[-2mm] \forall k \in \ZZ \ \ \ t < \arg_{(t,t+1)}(Z(a^k)) < \arg_{(t,t+1)}(Z(a^{k+1})) < t+1. \nonumber
 	\end{gather}

 (a.1) If $\phi(M')\geq t+1$, then  using   \eqref{help inequalities 3'} and \eqref{inequalities help 2'} it follows that $a^j\not \in \sigma^{ss}$ for any $j\in \ZZ$ and hence in Subsection \ref{first remarks} we deduce that  $C_{0,\sigma \sigma}(\mc T) = \emptyset $.
 
 (a.2)  So, assume that $\phi(M') < t+1$. In table \eqref{left right M} we have $\phi(b^n)<\phi(M')$ and due to \eqref{inequalities help 2'} we see that for some $N\in \ZZ$ holds $ n \leq N$ and \eqref{tbmprime}.
 
 Table \eqref{left right M} ensure that $\sigma \in (b^N,b^{N+1},M')$ and recalling that $\hom^1(b^{N+1},M')\neq 0$ in \eqref{nonvanishing2} we can use Corollary \ref{stable curve} to deduce that \begin{gather}  \alpha^{N}=\langle b^{N+1},M' \rangle = \langle a^N, b^{N+1},M' \rangle \in C_{0,\sigma \sigma}(\mc T). \end{gather}
 For $j<N$ we use \eqref{help inequalities 3'} and $\phi(b^N)<\phi(M')$ to deduce that $a^j\not \in \sigma^{ss}$ for  $j<N$. From the first triangle in \eqref{short filtration 1} we have actually that for $j>N$ holds $Z(a^j)= Z(M')+Z(b^{j+1})$,  $a^j\in \mc P([\phi(M'),\phi(b^{j+1})])\subset \mc P((t,t+1])$ hence from \eqref{tbmprime} and \cite[Remark 2.1]{dimkatz4} we see that \begin{gather}\label{phases of ajbjp1} N<j \ \ \Rightarrow  \ \ \arg_{(t,t+1)}(Z(a^j)) < \arg_{(t,t+1)}(Z(b^{j+1}))=\phi(b^{j+1}) \\ 
 \label{phiminusaaj} N<j, \ a^j\not \in \sigma^{ss}  \ \ \Rightarrow  \ \   \phi_-(a^j) < \arg_{(t,t+1)}(Z(a^j)) ,  \ \  \phi(M')\leq \phi_-(a^j)\leq \phi(b^{j+1}) \end{gather}
 Furthermore, we will show \eqref{if aj is not ss} below with the help of  \cite[Lemma 7.2]{dimkatz4}. Let  $N<j$ and $ \ a^j\not \in \sigma^{ss}$. One of the five  cases given there must appear. In case (a) of \cite[Lemma 7.2]{dimkatz4}   we have $a^k\in \sigma^{ss}$,   $\phi_-(a^j) =\phi(a^k)+1 $ for some $k<j-1$, however then \eqref{phiminusaaj} and \eqref{property of stability} imply $Z(a^k)\in Z(\delta)_-^c$, which contradicts \eqref{inequalities help 5'}. In case (b) of \cite[Lemma 7.2]{dimkatz4} we have   $\phi_-(a^j) =\phi(a^k) $ for some $k>j$, however  then \eqref{phiminusaaj}\ \eqref{help inequalities 3'}, and \eqref{property of stability} imply $\arg_{(t,t+1)}(Z(a^k))<\arg_{(t,t+1)}(Z(a^j))$ with $k>j$, which contradicts \eqref{inequalities help 5'}.  In case (c) of \cite[Lemma 7.2]{dimkatz4}  we have  $\phi_-(a^j) = \phi(b^k)+1$  for some $k<j$, which would imply  $Z(b^k)\in Z(\delta)_-^c$, which contradicts \eqref{inequalities help 2'}. Case  (d)  of \cite[Lemma 7.2]{dimkatz4} ensures  $\phi_-(a^j) =\phi(b^k) $  for some $ j<k$, and then  \eqref{phiminusaaj},  \eqref{property of stability} ensure   $\arg_{(t,t+1)}(Z(b^k))=\phi(b^k)<\arg_{(t,t+1)}(Z(a^j))$, which contradicts \eqref{phases of ajbjp1}, \eqref{inequalities help 2'} and $j<k$. So only case (e) in Lemma \cite[Lemma 7.2]{dimkatz4} remains possible, hence 
 \begin{gather} \label{if aj is not ss}
 j> N, \ \ a^j \not \in \sigma^{ss} \ \ \Rightarrow \ \ M  \in \sigma^{ss}, \ \ \phi_-(a^j) =\phi(M)+1. 
 \end{gather}

 In Remark \ref{remark for delta} we see that $-Z(M)=Z(M')-Z(\delta)$ and \eqref{tbmprime} implies 
 
 \begin{gather} \label{u is bigger than mprime} -Z(M) \in Z(\delta)^c_+ \ \  \qquad  \phi(M')< U=\arg_{(t,t+1)}(-Z(M)) < t+1  \end{gather} increasing the superscript starting from $N$ and taking into account \eqref{tbmprime} we get an integer $u\geq N$ such that \eqref{small u and big U} holds.
 In \eqref{short filtration 1}  we see that $Z(a^j)= Z(b^j)-Z(M)$ and from \eqref{small u and big U}, \eqref{u is bigger than mprime} it follows that  for $j>u$ holds  $\arg_{(t,t+1)}(Z(a^j))<\arg_{(t,t+1)}(Z(b^j))$. Therefore using  \eqref{help inequalities 3'}, \eqref{property of stability} we deduce that $a^j \in \sigma^{ss}$ for $j>u$ implies $\phi(a^j)<\phi(b^j)$ and $\hom(b^j,a^j)=0$, which contradicts \eqref{nonvanishing3}. Hence we get (using also \eqref{if aj is not ss}):
 \begin{gather} \label{j bigger than u} j >u \ \ \Rightarrow a^j \not \in \sigma^{ss} \Rightarrow \ \ M \in \sigma^{ss}, \phi_-(a^j)=\phi(M)+1.  \end{gather}

 For $N<j\leq u$ we have $\phi(b^j)\leq U$ due to \eqref{small u and big U} and  the formula $Z(a^j)= Z(b^j)-Z(M)$ implies that $\arg_{(t,t+1)}(Z(a^j))\leq U$.  Using  \eqref{phiminusaaj}, \eqref{if aj is not ss} we see that $N<j\leq u$, $a^j\not \in \sigma^{ss}$ imply $\phi(M')=\phi(M)+1 < U$, which is impossible, since $U=\arg_{(t,t+1)}(-Z(M))$ and we proved that $M\in \sigma^{ss}$. Therefore we proved $a^j \in \sigma^{ss} \iff N \leq j \leq u$. Taking into account that $\{b^j: j \in \ZZ\}\subset \sigma^{ss}$ and $M\in \sigma^{ss}$ we use Subsection \ref{first remarks} to deduce $C_{0,\sigma \sigma}(\mc T)=\{\alpha^j, \beta^j: N\leq j \leq u\}$.

 (b)  Assume now  	$\phi(b^{n+1})= \phi(b^{n})+1$.  Let us denote $\phi(b^n)=t$, then $\phi(b^{n+1})= t+1$.  Using that  $\{b^j\}_{j\in \ZZ}\subset \sigma^{ss}$, $\hom(b^n,b^m)\neq 0$ for $n\leq m$ and $\hom^1(b^m,b^n)\neq 0$ for $m >n+1$ it follows that:  
 \begin{gather} \label{help equalities 1'}   \phi(b^j)=t \ \ \mbox{for} \ j\leq n \ \  \qquad  \phi(b^j)=t+1  \ \ \mbox{for} \ \  j\geq n+1. \end{gather}
 
  (b.1) If $\phi(M')>t+1$, then in \eqref{help inequalities 3'} we see that  $a^j\in \sigma^{ss}$ implies $\phi(a^j)>t+1$ and hence $\hom^1(a^j,b^k)=0$ for small enough $k$, which contradicts \eqref{nonvanishing4}. Furthermore, if $M\in \sigma^{ss}$, then from $\hom(M,b^n)\neq 0$ in \eqref{nonvanishing1} it follows that $\phi(M)\leq \phi(b^n)$ and therefore    $ \phi(M')> \phi(M)+1$ which contradicts $\hom^1(M',M)\neq 0$ in \eqref{nonvanishing7}.
  
  (b.2) If  $\phi(M')<t+1$,  due to table \ref{left right M} we get $t<\phi(M')<t+1=\phi(b^n)+1$ and from Corollary  \ref{stable curve} it follows $\alpha^n=\langle M', a^n, b^{n+1}\rangle \in C_{0,\sigma \sigma}(\mc T)$, $a^n \in \sigma^{ss}$. From $Z(a^n)=Z(M')+Z(b^{n+1})$ in Remark \ref{remark for delta} it follows that $\phi(M')<\phi(a^n)<t+1$.  Hence $\phi(M')-1<\phi(a^n)-1<t=\phi(b^n)$.

   If $a^j\in \sigma^{ss}$ for some  $j > n$, then from \eqref{help inequalities 3'} and \eqref{help equalities 1'}, \eqref{property of stability} we see that   $Z(b^{j+1})$, $Z(a^{j})$ are non-zero and collinear and  from Remark \ref{remark for delta} we know that $Z(M')=Z(a^{j})- Z(b^{j+1})$, which contradicts $t<\phi(M')<t+1=\phi(b^n)+1$. So for  $j>n$ we have $a^j \not \in \sigma^{ss}$.  Recalling \eqref{non-ss a for big j}  and Subsection \ref{first remarks} we deduce that $\alpha^j,\beta^j \not \in C_{0,\sigma \sigma}(\mc T)$ for $j\neq n$. We will show that $M\in \sigma^{ss}$ and this would imply that  $\beta^n=\langle M,a^n,b^n \rangle \in C_{0,\sigma \sigma}(\mc T) $. 
  
  Indeed,  let  $j >n$.  From the first triangle in \eqref{short filtration 1} we have       $a^j\in \mc P([\phi(M'),\phi(b^{j+1})])\subset \mc P((t,t+1])$, therefore $\phi(M')\leq \phi_-(a^j)<\phi_+(a^j)\leq \phi(b^{j+1})$   and \cite[Remark 2.1]{dimkatz4} implies
  \begin{gather} \label{again phimaj} j >n \ \ \Rightarrow \ \
  t<\phi_-(a^j) < \arg_{(t,t+1)}(Z(a^j))< t+1.  \ \  
  \end{gather}
  One of the five cases given in \cite[Lemma 7.2]{dimkatz4}   must appear. In case (a) of \cite[Lemma 7.2]{dimkatz4}   we have $a^k\in \sigma^{ss}$,   $\phi_-(a^j) =\phi(a^k)+1 $ for some $k<j-1$, and since we already showed that only $j=n$ gives $a^j\in \sigma^{ss}$ it would follow that $\phi(a^n)+1=\phi_-(a^j)$, which contradicts \eqref{again phimaj} and   $\phi(M')<\phi(a^n)<t+1$.  In case (b) of \cite[Lemma 7.2]{dimkatz4} we have   $\phi_-(a^j) =\phi(a^k) $ for some $k>j$, which contradicts the already proved instability of $a^k$ for $k>n$.   In cases (c) and (d) of \cite[Lemma 7.2]{dimkatz4}  we have    $\phi_-(a^j) = \phi(b^k)+1$  or    $\phi_-(a^j) =\phi(b^k) $  for some $ k \in \ZZ$,  which contradicts \eqref{again phimaj} and the already proved fact that $\phi(b^k) = t$ or $\phi(b^k) = t+1$.  So only case (e) in Lemma \cite[Lemma 7.2]{dimkatz4} remains possible, hence  $M \in \sigma^{ss}$. 
  
  Finally, from $\hom(M,b^n)\neq 0$ in  \eqref{nonvanishing1} and $\hom^1(a^n,M)\neq 0$ in \eqref{nonvanishing2}  we have $\phi(a^n)-1\leq \phi(M) \leq t$ and then the second triangle in \eqref{short filtration 1} implies that $\phi(M')-1<\phi(a^n)-1<\phi(M)<t=\phi(b^n)$.
  
  (b.3) Let $\phi(M')=t+1$, then  for $p\geq n$ the first  distinguished triangle in \eqref{short filtration 1}  ensures that $a^p\in \sigma^{ss}$ and $\phi(a^p)=t+1$. In particular, we have $\phi(a^n) = \phi(b^{n})+1$ and from the second triangle in \eqref{short filtration 1}  it follows that $M\in \sigma^{ss}$, $\phi(M)=\phi(b^n)$.  Recalling \eqref{non-ss a for big j} and using Subsection \ref{first remarks} we obtain $C_{0,\sigma \sigma}(\mc T)=\{\alpha^j, \beta^j: j\geq n\}$.
 \end{proof} 
 
 \begin{lemma}  \label{lemma for semistability of M} Let us assume that we are in the situation of {\rm (a.1)} of  Proposition \ref{more careful part two}.In this case  $M\in \sigma^{ss} \iff \phi(M')=t+1$.
 \end{lemma} 
 \begin{proof}	
 	First,  we note that for $\phi(M')>t+1$ holds $M\not \in  \sigma^{ss}$. Indeed, if $M\in \sigma^{ss}$, then from \eqref{nonvanishing1} and \eqref{inequalities help 2'} it follows that $\phi(M)\leq t$ and $\hom^1(M', M)=0$, which contradicts \eqref{nonvanishing7}.
 	
 	So, it remains to consider the case   $\phi(M')=t+1$.  Combining the triangles in \eqref{short filtration 1}  we get a diagram  for  any $p \in \ZZ$:
 	\be \label{T12Zcap(E_1)the filtation} 
 	\begin{diagram}[size=1em] 
 		0 & \rTo      &        &       &   M'[-1] & \rTo      &             &        & a^{p}[-1] &   \rTo      &           &       & M \\
 		& \luDashto &        & \ldTo &         &\luDashto  &             &  \ldTo &           &   \luDashto &           & \ldTo &       \\
 		&           &  M'[-1] &       &         &           &  b^{p+1}[-1]  &        &           &             & b^{p} &
 	\end{diagram}
 	\ee
Therefore $M$ is in the extension closure of $M'[-1]$, $b^{p+1}[-1]$, $b^{p}$ for each $p\in \ZZ$.  Since for any $p\in \ZZ$ we have $\phi(b^{p+1}[-1])<t=\phi(M')-1< \phi(b^{p})$, it follows that $M\in \mc P[\phi(b^{p+1}[-1]),\phi(b^{p})]$ for each $p\in \ZZ$, hence $\phi(b^{p+1}[-1])\leq \phi_-(M)$  and $\phi_+(M)\leq \phi(b^{p})$ for each $p\in \ZZ$.  Now using \eqref{inequalities help 2'} and limiting $p$ to $+\infty$ we obtain $t\leq  \phi_-(M)$, whereas limiting $p$ to  $-\infty$ ensures $\phi_+(M)\leq t$, therefore $\phi_-(M)=\phi_+(M)=t$ and indeed $M\in \sigma^{ss}$.
 	   \end{proof}
 
 \begin{proposition} \label{arbitrary even integer}  Varying $\sigma \in (b^n,b^{n+1}, M')$ in the region with  $\phi(b^{n+1}) < \phi(b^{n})+1$ and $\phi(M')<t+1$, where $t$ is as in \eqref{the case of}, one can arrange that the pair $(N,u)$   is  $(n,n+i)$ for any $i\in \ZZ_{\geq 0}$, where $u,N$ are as in (a.2) of Proposition \ref{more careful part two}. Therefore for any $i\in \ZZ_{\geq 0}$ we have $\sigma$ such that  $C_{0,\sigma \sigma}(\mc T)=\{\alpha^j, \beta^j: n\leq  j \leq n+i \}$. Hence  $\abs{C_{0,\sigma \sigma}(\mc T)}$ can be any  positive even integer.  	
 \end{proposition}
 \begin{proof}
   From Remark \ref{remark for fe}  (b) and table \eqref{left right M} we know that  the map \eqref{assignment} gives a homeomorphism  from $(b^n,b^{n+1},M')\subset \st(\mc T)$ to  $\RR_{>0}^3 \times \left \{  \begin{array}{c} y_0 - y_1 < 0 \\  y_0 - y_2 < 0 \\ y_1 - y_2 < 1\end{array}\right\}$. Taking into account also \eqref{property of stability}, it  follows that for any $(r_0,r_1,r_2) \in \RR_{>0}^3$ and any $(y_0, y_1, y_2)$ satisfying the specified  inequalities we have $\sigma \in (b^n,b^{n+1},M')$ such that $Z(b^n)=r_0 \exp(\ri \pi y_0)$,  $Z(b^{n+1})=r_1 \exp(\ri \pi y_1)$, $Z(M')=r_2 \exp(\ri \pi y_2)$. These arguments and the equalities $Z(b^{n+1})=Z(b^n)-Z(\delta)$, $Z(\delta)= Z(M)+Z(M')$ imply that we have a stability condition $\sigma \in (b^n,b^{n+1},M')$, such that $Z(b^{n})$, $Z(b^{n+1})$, $Z(M')$, $-Z(M)$, $Z(\delta)$ are as in Figure \ref{Figure for several stable genus zero nc curves first}, and for this $\sigma$ we have $\phi(b^n)<\phi(M')<\phi(b^{n+1})<\phi(b^n)+1$. Furthermore, choosing $t$ as in \eqref{the case of}, we see from Figure \ref{Figure for several stable genus zero nc curves first} that $\phi(M')<t+1$. Hence the chosen  stability condition $\sigma$ falls in the case (a.2) of Proposition \ref{more careful part two}.   In particular, denoting $U=\arg_{(t,t+1)}(Z(M')+Z(b^{n+1})-Z(b^n))$, we have  $M\in \sigma^{ss}$, $U=\phi(M)+1$, and from  Figure \ref{Figure for several stable genus zero nc curves first} we see that not only $\phi(b^n)<\phi(M')<\phi(b^{n+1})$ but also $\phi(b^n)<\phi(M)+1=U<\phi(b^{n+1})$. Now  the defining properties of $N$, $u$ in \eqref{tbmprime} and \eqref{small u and big U} imply  that $u=N=n$ for this $\sigma$.
   
    The constructed $\sigma$ corresponds to a point $(r_0, r_1, r_2,y_0,y_1,y_2)$ via the  map \eqref{assignment}. Varying $r_2$ to arbitrary $s>0$  without changing $r_0, r_1,y_0,y_1,y_2$ corresponds, via the homeomorphism   specified above, to varying the stability condition $\sigma(s)$ in $(b^n,b^{n+1},M')$ and $\sigma(s)$  falls in the case (a.2) of Proposition \ref{more careful part two} for any $s>0$. Since we change only $r_2$, the entities   $\phi_{\sigma(s)}(M')$, $\phi_{\sigma(s)}(b^n)$,   $\phi_{\sigma(s)}(b^{n+1})$, $Z_{\sigma(s)}(b^n)$, $Z_{\sigma(s)}(b^{n+1})$ remain the same as for the initial $\sigma=\sigma(r_2)$ for any $s>0$. It follows that $Z_{\sigma(s)}(\delta) = \arg_{(\phi_{\sigma(s)}(b^{n})-1,\phi_{\sigma(s)}(b^{n}))}(Z_{\sigma(s)}(b^n)-Z_{\sigma(s)}(b^{n+1}))$ and $\{Z_{\sigma(s)}(b^{j})\}_{j\in \ZZ}$ remain the same as well  for any $s>0$. By the arguments used  before \eqref{inequalities help 2'} it follows that for any $j\in \ZZ$ the real number  $\phi_{\sigma(s)}(b^{j})$ equals  $  \arg_{(t,t+1)}(Z_{\sigma(s)}(b^{j})) $ and therefore it does not change as $s$ varies in $\RR_{>0}$.  Now from \eqref{tbmprime} we  see that  $N(s)$ does not change as $s$ varies in $\RR_{>0}$ and remains equal to $n$.  
    
     On the other hand from Figures \ref{Figure for several stable genus zero nc curves second} and  \eqref{Figure for several stable genus zero nc curves} we see that $U(s)=\phi_{\sigma(s)}(M)+1$ increases monotone and continuously as $s$ decreases,  and that $\lim_{s\rightarrow 0}U(s) = t+1$. Recalling \eqref{inequalities help 2'} and the defining property of $u(s)$ in \eqref{small u and big U}   which hold for $\sigma(s)$ and each $s>0$ we see that we can obtain an infinite sequence $r_0=s_0>s_1>s_2 > s_3 > \dots >0 $  such that $u(s_i)=n+i$ for each $i\in \ZZ_{\geq 0}$.

 	\begin{figure} \centering
 		\includegraphics[scale=0.8]{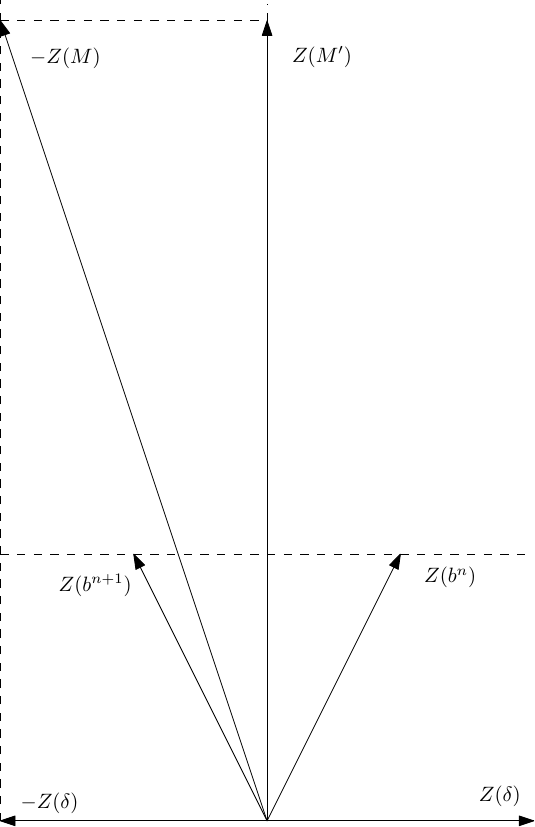}
 		\centering
 		\caption{$n=N=u$}\label{Figure for several stable genus zero nc curves first}
 	\end{figure}
 	
 	\begin{figure} \centering
 		\includegraphics{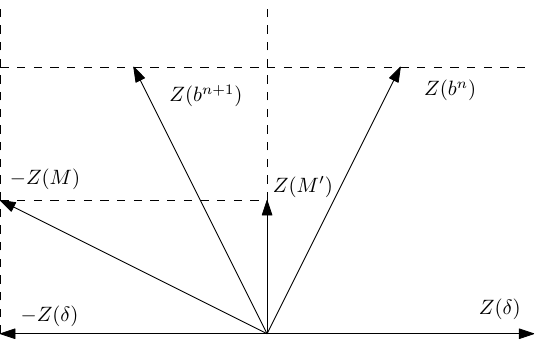}
 		\centering
 		\caption{$n=N<u$}\label{Figure for several stable genus zero nc curves second}
 	\end{figure}

 	\begin{figure}[h]
 		\includegraphics[scale=0.6]{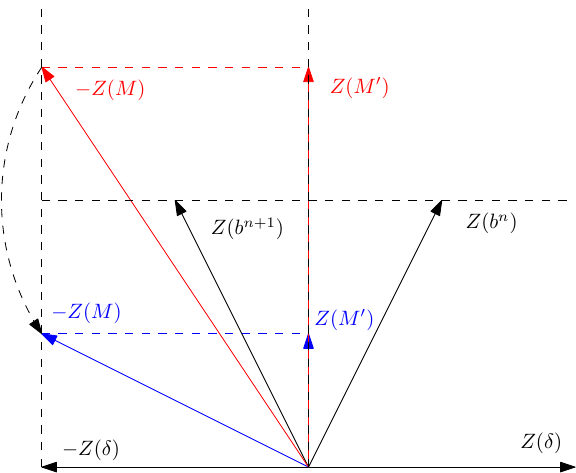}
 		\centering
 		\caption{}\label{Figure for several stable genus zero nc curves}
 	\end{figure}
 	
 \end{proof}

 \begin{proposition}  \label{more careful study}Let $\sigma \in (M,b^n,b^{n+1})$ for some $n\in \ZZ$ and  $\phi(b^{n+1})\leq \phi(b^{n})+1$. Then $A\not \in C_{1,\sigma \sigma}(\mc T)$.

 	{\rm (a)} If    $\phi(b^{n+1}) < \phi(b^{n})+1$,  we denote $ t=\arg_{(\phi(b^{n})-1,\phi(b^{n}))}(Z(b^n)-Z(b^{n+1}))$  and then must hold $\phi(M) < t$ and  consider two sub-cases:
 		
 	{\rm (a.1)}  $\phi(M)\leq t-1$, then in this sub-case $a^j\not \in \sigma^{ss}$ for all $j\in \ZZ$ and $C_{0,\sigma \sigma}(\mc T)=\emptyset$ and  $A\not \in C_{1,\sigma}(\mc T)$; 
 	
 	{\rm (a.2)} $\phi(M)> t-1$: in this sub-case if we denote $V=\arg_{(t,t+1)}(Z(b^{n})-Z(b^{n+1})-Z(M))$, and then holds $V<\phi(M)+1$, $M'\in \sigma^{ss}$, $\phi(M')=V$, and  there are integers $N\leq u$ such that
 	\begin{gather}  \label{tbmprime1} t < \phi(b^N)<V\leq \phi(b^{N+1})<t+1\\   \label{small u and big U1}  
 	t< \phi(b^u)\leq \phi(M)+1 < \phi(b^{u+1}) < t+1.
 	\end{gather} 
 	 and for these integers $ N, u$  holds 
 	 $ a^j\in \sigma^{ss} \iff N\leq j \leq u $ and $C_{0,\sigma \sigma}(\mc T)=\{\alpha^j, \beta^j: N\leq j \leq u\}$.  In particular here we have   $A\not \in C_{1,\sigma}(\mc T)$ again.

 	{\rm (b)} If $\phi(b^{n+1})=\phi(b^n)+1$, then we consider three sub-cases
 	
 	{\rm	(b.1) }  $ \phi(M)< \phi(b^n)-1 $: in this case $a^j\not \in \sigma^{ss}$ for all $j\in \ZZ$ and $A\not \in C_{1,\sigma}(\mc T)$, $C_{0,\sigma \sigma}(\mc T)=\emptyset$. Furthermore, in this case  $M' \not \in \sigma^{ss}$.
 		
 		{\rm	(b.2) }  $ \phi(M)> \phi(b^n)-1 $: in this case $a^j\in  \ \sigma^{ss}$ iff $j=n$, $A\not \in C_{1,\sigma}(\mc T)$, $C_{0,\sigma \sigma}(\mc T)=\{\alpha^n, \beta^n \}$. Furthermore, in this case  $M'\in \sigma^{ss}$ and $\phi(M')<\phi(M)+1$.
 		
 		{\rm	(b.3) }  $ \phi(M)=\phi(b^n)-1 $: in this sub-case   for  $j\leq n$ hold  $a^j\in \sigma^{ss}$, $\phi(a^j)=\phi(b^j)=\phi(b^n)$ and for  $j> n$ hold  $a^j\not \in \sigma^{ss}$,  $M'\in \sigma^{ss}$, $\phi(M')=\phi(M)+1$. In particular,  $A\in C_{1,\sigma}(\mc T)$ and  $C_{0,\sigma \sigma}(\mc T)=\{\alpha^j, \beta^j: j\leq n\}$.

In particular, the rest of table \eqref{table for  M b b}  follows,  and then from Remark \ref{useful autoequivalence} follows table \eqref{table for M' a a}, 	
 	
 \end{proposition}	
 \begin{proof} From the  proposition \ref{proposition for  B in Tb}  we see that $\{b^j\}_{j\in \ZZ}\subset \sigma^{ss}$. 
 	
 	From Corollary \ref{nonvanishings} we see that $\hom(b^{j},a^{j})\neq 0$, $\hom(a^{j},b^{j+1})\neq 0$, $\hom^1(  a^j, M)\neq 0$ for any $j\in \ZZ$  and it follows \begin{gather} \label{help inequalities 3} \forall j \in \ZZ \qquad a^{j} \in \sigma^{ss}  \ \Rightarrow \ \phi(b^{j})\leq \phi(a^{j})\leq \phi(b^{j+1}) \ \ \mbox{and} \ \ \phi(a^j) \leq \phi(M)+1. \end{gather}	
 	
 	So, if $a^j \in \sigma^{ss}$ for some $j\in \ZZ$,  then  $\phi(a^j)\leq \phi(M)+1 $ and from table \eqref{left right M} it follows that $\phi(a^j)< \phi(b^{n+1})$, hence $\hom(b^{n+1},a^j)=0$. Now the non-vanishing $ \hom(b^p,a^q)\neq 0$ for $p\leq q$ in  Corollary \ref{nonvanishings} implies that $n\geq j$. So we see that 
 	
 	\begin{gather} \label{non-ss a for small j}  j>n  \ \ \Rightarrow  \ \  a^j \not \in \sigma^{ss}
 	\end{gather}  and therefore indeed $A\not \in C_{1,\sigma \sigma}$. 
 	
 	(a) Here we  have 	$\phi(b^{n})<\phi(b^{n+1})< \phi(b^{n})+1$. The  arguments as in the beginning of the proof of (a) in Proposition \ref{more careful part two} are applicable in this case and they imply the same formulas \eqref{inequalities help 2'}, \eqref{inequalities help 5'} in this case. From \eqref{inequalities help 2'} and the inequality $\phi(M)+1 <\phi(b^{n+1})$ in \eqref{left right M} imply that $\phi(M)< t$.
 	
 	(a.1) If in this sub-case $a^j\in \sigma^{ss}$ for some $j\in \ZZ$,  \eqref{help inequalities 3} implies $t<\phi(a^j)$ and $\phi(a^j)\leq \phi(M)+1\leq t$, which is a contradiction. Therefore in this case $a^j \not \in \sigma^{ss}$ for each $j\in \ZZ$. 
 	
 		(a.2) In this case $t-1 <\phi(M)<t$. In table \eqref{left right M} we have $\phi(M)+1<\phi(b^{n+1})$ and due to \eqref{inequalities help 2'} we see that for some $u\in \ZZ$ holds $ u \leq n$ and \eqref{small u and big U1}.
 		
 		Table \eqref{left right M} ensure that $\sigma \in (M,b^u,b^{u+1})$ and recalling that $\hom(M,b^u)\neq 0$ in \eqref{nonvanishing1} we can use Corollary \ref{stable curve} to deduce that \begin{gather} \label{betau}  \beta^{u}=\langle M,b^{u} \rangle = \langle M,b^{u}, a^u \rangle \in C_{0,\sigma \sigma}(\mc T). \end{gather}

 		For $j>u$ we use \eqref{help inequalities 3} and $\phi(M)+1<\phi(b^{u+1})$ to deduce that $a^j\not \in \sigma^{ss}$ for  $j>u$. From the second triangle in \eqref{short filtration 1} we have  for $j<u$ holds $Z(a^j)= Z(b^{j})-Z(M)$,  $a^j\in \mc P([\phi(b^{j}), \phi(M)+1])\subset \mc P((t,t+1])$ hence from \eqref{small u and big U1} and \cite[Remark 2.1]{dimkatz4} we see that \begin{gather}\label{phases of ajbjp111} j<u \ \ \Rightarrow  \ \ \phi(M)+1 > \arg_{(t,t+1)}(Z(a^j)) > \arg_{(t,t+1)}(Z(b^{j}))=\phi(b^{j}) \\ 
 		\label{phiminusaaj1} j<u, \  a^j\not \in \sigma^{ss}  \ \ \Rightarrow  \ \   \phi_-(a^j) < \arg_{(t,t+1)}(Z(a^j)) ,  \ \  \phi(b^{j+1})\leq \phi_-(a^j)\leq \phi(M)+1 \end{gather}
 		
 		In Remark \ref{remark for delta} we see that $Z(M')=-Z(M)+Z(\delta)$ and \eqref{small u and big U1} implies 
 		
 		\begin{gather} \label{u is bigger than mprime1} Z(M') \in Z(\delta)^c_+ \ \  \qquad  t<V=\arg_{(t,t+1)}(Z(M')) < \phi(M)+1<  t+1  \end{gather} decreasing the superscript starting from $u$ and taking into account \eqref{small u and big U1} we get an integer $N\leq u$ such that \eqref{tbmprime1} holds. In Remark \ref{remark for delta} we see that 
 		$Z(a^j)= Z(M')+Z(b^{j+1})$ and from \eqref{tbmprime1} we see that
 		 \begin{gather} j<N \Rightarrow \arg_{(t,t+1)}(Z(a^j))>\arg_{(t,t+1)}(Z(b^{j+1}))=\phi(b^{j+1}) \\ \label{phases of ajbjp11}
 		 j\geq N \Rightarrow \arg_{(t,t+1)}(Z(a^j))\leq \arg_{(t,t+1)}(Z(b^{j+1}))=\phi(b^{j+1})
 		 \end{gather} Therefore, using \eqref{property of stability} and \eqref{help inequalities 3} we deduce that for $j<N$ the inclusion $a^j\in \sigma^{ss}$ implies $\phi(a^j) > \phi(b^{j+1})$, $\hom(a^j,b^{j+1} )=0$ which contradicts \eqref{nonvanishing4}.

 		Now we will show that $a^j \in \sigma^{ss}$ for $N\leq j<u$ with the help of  \cite[Lemma 7.2]{dimkatz4}. Let  $N\leq j<u$ and $ \ a^j\not \in \sigma^{ss}$. One of the five  cases given there must appear. In case (a) of \cite[Lemma 7.2]{dimkatz4}   we have $a^k\in \sigma^{ss}$,   $\phi_-(a^j) =\phi(a^k)+1 $ for some $k<j-1$, however then \eqref{phiminusaaj1} and \eqref{property of stability} imply $Z(a^k)\in Z(\delta)_-^c$, which contradicts \eqref{inequalities help 5'}. In case (b) of \cite[Lemma 7.2]{dimkatz4} we have   $\phi_-(a^j) =\phi(a^k) $ for some $k>j$, however  then \eqref{phiminusaaj1}\ \eqref{help inequalities 3'}, and \eqref{property of stability} imply $\arg_{(t,t+1)}(Z(a^k))<\arg_{(t,t+1)}(Z(a^j))$ with $k>j$, which contradicts \eqref{inequalities help 5'}.  In case (c) of \cite[Lemma 7.2]{dimkatz4}  we have  $\phi_-(a^j) = \phi(b^k)+1$  for some $k<j$, which would imply  $Z(b^k)\in Z(\delta)_-^c$, which contradicts \eqref{inequalities help 2'}. Case  (d)  of \cite[Lemma 7.2]{dimkatz4} ensures  $\phi_-(a^j) =\phi(b^k) $  for some $ j<k$, and then  \eqref{phiminusaaj1},  \eqref{property of stability} ensure   $\arg_{(t,t+1)}(Z(b^k))=\phi(b^k)<\arg_{(t,t+1)}(Z(a^j))$, which contradicts \eqref{phases of ajbjp11}, \eqref{inequalities help 2'} and $j<k$. In the last  case (e) of  Lemma \cite[Lemma 7.2]{dimkatz4}  holds $\phi_-(a^j) =\phi(M)+1$, however \eqref{phases of ajbjp111} and \eqref{phiminusaaj1} imply that  $\phi_-(a^j) <\phi(M)+1$ (here we take into account $j<u$), which is again contradiction and we proved that $a^j \in \sigma^{ss}$ for $N\leq j<u$. And   in \eqref{betau}  we also proved that $a^u \in \sigma^{ss}$.
 		
 		Finally,  we want  to prove that $M'\in \sigma^{ss}$.  We will use \cite[Lemma 2.12]{dimkatz4} with either the triple $(a^u,M, b^{u+1}[-1])$ or the triple  $(a^{u-1},M, b^{u}[-1])$ and taking into account the first triangle in \eqref{short filtration 1}. From Corollary \ref{nonvanishings} we see that this triple is Ext-exceptional triple.    From the second triangle in \eqref{short filtration 1}   it follows that $\phi(a^u)-1 = \phi(M) = \phi(b^u)-1$ if $\phi(M)=\phi(b^u)-1$ and $ \phi(b^u)-1<\phi(a^u)-1 < \phi(M) $ if  $\phi(M)>\phi(b^u)-1$. The latter case combined with \eqref{small u and big U1} and \eqref{help inequalities 3}  imply  $\phi(a^u)-1 < \phi(M) <\phi(b^{u+1}[-1])< \phi(a^u)$ and indeed the second system of inequalities in the conditions of \cite[Lemma 2.12]{dimkatz4} holds, and then this lemma ensures that $M'\in \sigma^{ss}$. Note that, if $N=u$, then the already proved \eqref{tbmprime1}, \eqref{small u and big U1} and $V<\phi(M)+1$ imply the latter case.  Thus, it remains to consider the case  $\phi(a^u)-1 = \phi(M) = \phi(b^u)-1$ and $N<u$. From \eqref{inequalities help 5'} and since $a^{u-1}, a^u,\in \sigma^{ss}$  (recall also \eqref{help inequalities 3}) we see that $\phi(a^{u-1})<\phi(a^{u})=\phi(b^{u})=\phi(M)+1$, hence recalling also that $\phi(M)<t<\phi(b^{u-1})\leq \phi(a^{u-1})$ we derive $\phi(a^{u-1})-1<\phi(b^{u}[-1])=\phi(M) < \phi(a^{u-1})$, and now \cite[Lemma 2.12]{dimkatz4} applied to the triple $(a^{u-1},M, b^{u}[-1])$ ensures $M'\in \sigma^{ss}$. Now Subsection \ref{first remarks} shows that $\alpha^j \in C_{o,\sigma\sigma}(\mc T)$ for $N\leq j \leq u$.

 	(b)  Assume now  	$\phi(b^{n+1})= \phi(b^{n})+1$.  Let us denote $\phi(b^n)=t$, then $\phi(b^{n+1})= t+1$.  Using that  $\{b^j\}_{j\in \ZZ}\subset \sigma^{ss}$, $\hom(b^n,b^m)\neq 0$ for $n\leq m$ and $\hom^1(b^m,b^n)\neq 0$ for $m >n+1$ it follows that:  
 		\begin{gather} \label{help equalities 1}   \phi(b^j)=t \ \ \mbox{for} \ j\leq n \ \  \qquad  \phi(b^j)=t+1  \ \ \mbox{for} \ \  j\geq n+1. \end{gather}

 	(b.1)   If $\phi(M)< \phi(b^n)-1$ and $a^j \in \sigma^{ss}$ for some $j\in \ZZ$, then \eqref{nonvanishing2}  imply  $\phi(a^j)<\phi(b^n)$, hence \eqref{help equalities 1}  imply that  $\phi(a^j)<\phi(b^k)$, $\hom(b^k, a^j)$ for small enough $k$, which contradicts \eqref{nonvanishing3}.
 	
 	 If $\phi(M)< \phi(b^n)-1$ and $M' \in \sigma^{ss}$, then from $\hom^1(b^{n+1}, M')\neq 0$ in \eqref{nonvanishing2} it follows that $\phi(b^{n+1})-1\leq \phi(M')$ and using that  $\phi(b^n)=\phi(b^{n+1})-1$ we obtain  $\phi(M)< \phi(M')-1$ and hence $\hom^1(M', M)=0$, which contradicts \eqref{nonvanishing7}.

 	(b.2) In this case $t-1 < \phi(M)$ and from table \eqref{left right M} we have also $\phi(M)<t$. Recalling that $\hom(M,b^n)\neq 0$ (see \eqref{nonvanishing1}) we can use  Corollary \ref{stable curve} to deduce that $\beta^n=\langle M, b^n, a^n \rangle \in C_{0,\sigma \sigma}(\mc T)$.  Since $Z(a^n) = Z(b^n)-Z(M)$ and  $\hom(b^n, a^n)\neq 0$ \eqref{nonvanishing3},   $\hom^1(a^n, M )\neq 0$ \eqref{nonvanishing2} we have $\phi(b^{n+1})-1 =\phi(b^n)<\phi(a^n)<\phi(M)+1$.

 	If $a^j\in \sigma^{ss}$ for some  $j < n$, then from \eqref{help inequalities 3} and \eqref{help equalities 1}, \eqref{property of stability} we see that   $Z(b^{j+1})$, $Z(a^{j})$, $Z(b^{j})$ are non-zero and collinear and  from Remark \ref{remark for delta} we know that $Z(M) = Z(b^j)-Z(a^j)$, which contradicts $t-1<\phi(M)<t=\phi(b^n)$. So for  $j<n$ we have $a^j \not \in \sigma^{ss}$.  Recalling \eqref{non-ss a for small j}  and Subsection \ref{first remarks} we deduce that $\alpha^j,\beta^j \not \in C_{0,\sigma \sigma}(\mc T)$ for $j\neq n$. If we prove that $M' \in \sigma^{ss}$, then using again  Subsection \ref{first remarks} we would obtain $  C_{0,\sigma \sigma}(\mc T)=\{\alpha^n,\beta^n\}$. To prove that  $M' \in \sigma^{ss}$ we use  \cite[Lemma 2.12]{dimkatz4} with  the Ext-exceptional  triple $(a^n,M, b^{n+1}[-1])$ and the first triangle in \eqref{short filtration 1}. We showed above that $\phi(b^n)<\phi(a^n)<\phi(M)+1$, and recalling that $t-1< \phi(M)<t=\phi(b^n)=\phi(b^{n+1})-1$  it follows  $\phi(a^n)-1<\phi(M)<\phi(a^n)$, $\phi(a^n)-1<\phi(b^{n+1}[-1])<\phi(a^n)$. Thus, the second system of the inequalities in the presumptions in \cite[Lemma 2.12]{dimkatz4} is satisfied and this lemma ensures that $M' \in \sigma^{ss}$ due to the first triangle in \eqref{short filtration 1} wit $q=n$.
 	
 	Finally, 	from $\hom^1(b^{n+1}, M')\neq 0$ in \eqref{nonvanishing2} and $\hom(M', a^n)\neq 0$ in \eqref{nonvanishing1} it follows that $\phi(b^{n+1})-1 \leq \phi(M')\leq \phi(a^n)$ and now from the first triangle in \eqref{short filtration 1} we get $\phi(b^{n+1})-1 <\phi(M')<\phi(a^n)<\phi(M)+1$.
 	
 		(b.3)   Now we have $\phi(M)=\phi(b^n)-1 $ and hence $\phi(M)=t-1$. Now for $q\leq n$ the second distinguished triangle in \eqref{short filtration 1}  ensures that $a^q\in \sigma^{ss}$ and $\phi(a^q)=t$. In particular, we have $\phi(a^n)+1 = \phi(b^{n+1})$ and from the first triangle in \eqref{short filtration 1}  it follows that $M'\in \sigma^{ss}$, $\phi(M')=\phi(a^n)$. Recalling \eqref{non-ss a for small j} and using Subsection \ref{first remarks} we obtain $C_{0,\sigma \sigma}(\mc T)=\{\alpha^j, \beta^j: j\leq n\}$.
 		\end{proof}

  \begin{lemma} \label{lemma for semistability of M'} Let us assume that we are in the situation of {\rm (a.1)} of  Proposition \ref{more careful study}. In this case  $M'\in \sigma^{ss}\iff \phi(M)=t-1$ .
  	\end{lemma} \begin{proof}
  	First,  we note that for $\phi(M)<t-1$ holds $M'\not \in  \sigma^{ss}$. Indeed, if $M'\in \sigma^{ss}$, then from \eqref{nonvanishing2} and \eqref{inequalities help 2'} it follows that $\phi(M')\geq t$ and $\hom^1(M', M)=0$, which contradicts \eqref{nonvanishing7}.
  	
  	So, it remains to consider the case  $\phi(M)=t-1$. From the second   triangle in \eqref{short filtration 1} we have a distinguished triangle  $\begin{diagram}[size=1em] 
  	b^p  & \rTo      &        &       & a^p     \\
  	& \luDashto &        & \ldTo &        \\
  	&           &  	M[1]   &       &        
  	\end{diagram}$   for each $p\in \ZZ$, and since $\phi(b^p)> \phi(M[1])=t$, it follows that this triangle is the HN filtration of $a^p$, hence $a^p\in \mc P[t,\phi(b^p)]$ for each $p\in \ZZ$. 	From the first triangle in    \eqref{short filtration 1} we see that $M'$ is in the extension closure of $b^{p+1}[-1]$ and $a^p$, and taking into account that $\phi(b^{p+1}[-1])<t$, we deduce  that $M'\in \mc P[\phi(b^{p+1}[-1]),\phi(b^{p})]$ for each $p\in \ZZ$. Therefore $\phi(b^{p+1}[-1])\leq \phi_-(M')$  and $\phi_+(M')\leq \phi(b^{p})$ for each $p\in \ZZ$.  Now using \eqref{inequalities help 2'} and limiting $p$ to $+\infty$ we obtain $t\leq  \phi_-(M')$, whereas limiting $p$ to  $-\infty$ ensures $\phi_+(M')\leq t$, therefore $\phi_-(M')=\phi_+(M')=t$ and indeed $M'\in \sigma^{ss}$.
  \end{proof}

 From  tables \eqref{table for a a M}, \eqref{table for M' a a}, \eqref{table for b b M'}, \eqref{table for M b b}, \eqref{table for b a b} and subsection \ref{homeomoeprhism} follow the following two corollaries:  
   \begin{corollary} \label{in TaTb only one} For $\sigma \in \mk{T}_{a}^{st} \cup \mk{T}_{b}^{st}$ we have $\abs{C_{1,\sigma \sigma}(\mc T)}\leq 1$.
   	
   \end{corollary}

 \begin{corollary} \label{infinite genus zero part 1} The set $\{\sigma \in \mk{T}_a^{st}\cup  \mk{T}_b^{st} : \abs{C_{0,\sigma \sigma}(\mc T)} = \infty \}$ is a disjoint  union (any two summands are disjoint) 
 	
 	$\bigcup_{p\in \ZZ} \{\sigma \in (M,b^p, b^{p+1}) : \phi(b^{p+1})=\phi(b^p)+1, \phi(b^p)=\phi(M)+1 \} \cup $
 	
 	$\bigcup_{p\in \ZZ} \{ \sigma \in (M,a^p, a^{p+1}): \phi(a^{p+1})=\phi(a^p)+1, \phi(a^p) =\phi(M')+1 \} \cup$

 		$\bigcup_{p\in \ZZ} \{ \sigma \in (b^p, b^{p+1}, M'): \phi(b^{p+1})=\phi(b^p)+1=\phi(M') \}  \cup$
 		
 		$\bigcup_{p\in \ZZ} \{\sigma \in (a^p, a^{p+1}, M):\phi(a^{p+1})=\phi(a^p)+1=\phi(M) \}$.
 	
 \end{corollary}
 \begin{proof}
 	In the tables  we see that $\abs{C_{0,\sigma \sigma}(\mc T)}=\infty$ if and only iff $\sigma$ is in the shown union. The fact that the union is disjoint follows from the behavior of $C_{0, \sigma \sigma}(\mc T)$. In the tables we see that when $\sigma$ varies in any subset, which is a  summand of  the given union,  the set $C_{0, \sigma \sigma}(\mc T)$  does not change and that for  any two different summands the corresponding sets of stable genus zero non-commutative curves are different (recall also that $\mk{T}_a^{st} \cap \mk{T}_b^{st}=\emptyset$,  see \eqref{st with T}). 
 \end{proof}

 \subsection{Study of $C_{1,\sigma}(\mc T)$, $C_{1,\sigma \sigma}(\mc T)$ as $\sigma \in (\_,M,\_) \cup (\_,M',\_)$ and results  for the genus zero case. } \label{section for middle MMprime}

 Here Remark \ref{action1} shows that the auto-equivalence $\zeta$ transforms $(\_,M,\_)$ to $(\_,M',\_)$ and more precisely $ \sigma \in  (a^p,M,b^{p+1})$ is mapped to  $\zeta \cdot \sigma \in  (b^p,M',a^{p})$ as $\phi_{\zeta \cdot \sigma}(b^p)=\phi_{ \sigma}(a^p)$, $\phi_{\zeta \cdot \sigma}(M')=\phi_{ \sigma}(M)$, $\phi_{\zeta \cdot \sigma}(a^p)=\phi_{ \sigma}(b^{p+1})$.  Therefore it is enough to   study the behavior of $C_{l,\sigma}(\mc T)$, $C_{l,\sigma \sigma}(\mc T)$ in $ (a^p,M,b^{p+1})$. To that end, we  need first to strengthen \cite[
 lemma 7.5 (e)]{dimkatz4}.  So, we prove first:
\begin{lemma} \label{semi-stability of b} Let $\sigma \in (a^p,M,b^{p+1})$,  and let the following inequalities hold:  \begin{gather} \label{semi-stability of b ineq}    \phi\left (a^{p}\right)-1< \phi\left ( M\right)<\phi\left (a^{p}\right)\\  \label{semi-stability of b ineq1}  \phi\left (a^{p}\right)-1< \phi\left ( b^{p+1}\right)-1<\phi\left (a^{p}\right). \end{gather} 	
	
	Then  $ M' \in \sigma^{ss}$ and  $\phi(b^{p+1})-1<\phi(M')=\arg_{(\phi\left (a^{p}\right)-1,\phi\left (a^{p}\right))}(Z(a^p)-Z(b^{p+1}))<\phi(a^{p})$.
	
	 Furthermore, if  $\phi(M)=\phi(M')$, then $\sigma \in (a^{j}, M, b^{j+1})\cap (b^j, M', a^j)$ for all $j\in \ZZ$, and therefore $C_{D^b(pt),\sigma \sigma}(\mc T)= C_{D^b(pt), \KK}(\mc T)$. In particular  $\cap_{j\in \ZZ}  (a^{j}, M, b^{j+1})\cap (b^j, M', a^j) \neq \emptyset$. 
	 
\end{lemma}
 \begin{proof}
 From \cite[Lemma 7.5 (a), (c)]{dimkatz4}	we obtain   $b^{p}, M' \in \sigma^{ss}$, and   $\phi(M)<\phi(b^{p})<\phi(a^{p})$, $\phi(b^{p+1})-1<\phi(M')=\arg_{(\phi\left (a^{p}\right)-1,\phi\left (a^{p}\right))}(Z(a^p)-Z(b^{p+1}))<\phi(a^{p})$.
 
Assume that $\phi(M)=\phi(M')$ holds.
 	
   In this case   we have $\phi(a^{p})-1<\phi(M)=\phi(M')<\phi(a^{p})$. Recalling that  $Z(\delta)=Z(M')+Z(M)$, we see that   $t=\phi(M)=\phi(M')$ satisfies    $Z(\delta)=\abs{Z(\delta)}\exp(\ri \pi t)$ and $t<\phi(a^{p})<t+1$.   From the already obtained  $\phi(M)<\phi(b^{p})<\phi(a^{p})$  we get $t<\phi(b^{p})<\phi(a^{p})<t+1$.   Now we can apply \cite[Corollary 3.20]{dimkatz4}, which,  besides   $\{Z(a^j), Z(b^j)\}_{j\in \ZZ}\subset Z(\delta)^c_+$, gives us the formulas 
    \begin{gather} \label{noncolinear ab1}  \forall j \in \ZZ \ \ \  \arg_{(t,t+1)}(Z(x^j))<\arg_{(t,t+1)}(Z(x^{j+1})) \\
    \label{noncolinear ab2}    \lim_{j\rightarrow -\infty } \arg_{(t,t+1)}(Z(x^j))=t; \quad  \lim_{j\rightarrow+ \infty } \arg_{(t,t+1)}(Z(x^j))=t+1,  \end{gather}
   where  $\{x^i\}_{i\in \ZZ}$ is   either the sequence $\{a^i\}_{i\in \ZZ}$ or   the sequence $\{b^i\}_{i\in \ZZ}$. On the other hand, \cite[Corollary 3.20]{dimkatz4} claims   that  for any three integers $i,j,m$ we have:
    \begin{gather} \label{j<mleqi} j<m\leq i \ \ \Rightarrow \ \  \arg_{(t,t+1)}(Z(a^j))< \arg_{(t,t+1)}(Z(b^m))<\arg_{(t,t+1)}(Z(a^i)).\end{gather}

   We extend  the inequality $t<\phi(b^{p})<\phi(a^{p})<t+1$  to \eqref{semi-stability of b ineq3} as follows.    We already have  that $a^p, b^p, b^{p+1}\in \sigma^{ss}$.   In \eqref{semi-stability of b ineq1} is given that  $\phi(a^{p})<\phi(b^{p+1})$. 
   From $\hom^1(b^{p+1},M')\neq 0$ (see \eqref{nonvanishing2}) it follows $\phi(b^{p+1})\leq t+1$, and from $Z(b^{p+1})\in Z(\delta)^c_+$ we get  $\phi(b^{p+1})< t+1=\phi(M)+1$.  We have also  $\phi(M)<\phi(b^{p+1})$ (due to $\sigma \in (a^p,M,b^{p+1})$ and \eqref{middle M}). Therefore $\phi(b^{p+1})-1<\phi(M)< \phi(b^{p+1})$, and from \cite[Lemma 7.4 (a)]{dimkatz4}   we get  $a^{p+1}\in \sigma^{ss}$ and   $\phi(b^{p+1})-1<\phi(a^{p+1})-1<\phi(M)$. Thus, we derive:
   \begin{gather} \label{semi-stability of b ineq3} \phi(a^{p})-1<\phi(b^{p+1})-1<\phi(a^{p+1})-1<t<\phi(b^{p})<\phi(a^{p})<\phi(b^{p+1})<\phi(a^{p+1})<t+1.\end{gather}

   In  \cite[Lemma 7.5  (e)]{dimkatz4} it is shown that $ \sigma \in (a^j,M,b^{j+1})$ for each $j\leq p$, and therefore  $a^j, b^{j+1} \in \sigma^{ss}$   for each $j\leq p$.

 Let $j\geq p+1$. We will work now to prove that $a^j, b^j \in \sigma^{ss}$.  From \cite[Remark 3.9]{dimkatz4} we see that,  if $x^j$ is $a^j$ or $b^j$, then  $x^{j}[-1]$ is in the extension closure of $x^p,x^{p+1}[-1]$, therefore we have:
 \begin{gather} \label{phipm} j\geq p+1   \qquad x^j\in\mc P[\phi(x^{p+1}),\phi(x^p)+1]   \Rightarrow \phi_\pm(x^j) \in [\phi(x^{p+1}),\phi(x^p)+1].\end{gather}

  Suppose that $a^j\not \in \sigma^{ss}$ for some $j\geq p+1$. Due to \eqref{semi-stability of b ineq3} and \eqref{phipm} we have  $a^j\in\mc P[\phi(a^{p+1}),\phi(a^p)+1] \subset  \mc P(\phi(b^{p+1}),\phi(b^{p+1})+1]$, and  \cite[Remark 2.1(c)]{dimkatz4} shows that  $\arg_{(\phi(b^{p+1}),\phi(b^{p+1})+1]}(Z(a^j))=\arg_{(t,t+1)}(Z(a^j))$. Combining with \cite[Remark 2.1 (a)]{dimkatz4}, we put together: 
  \begin{gather} \label{semi-stability of b ineq5} t<\phi(a^{p+1})\leq \phi_-(a^j) < \arg_{(t,t+1)}(Z(a^j))<\phi_+(a^j)\leq  \phi(a^p)+1 <\phi(a^{p+1})+1. \end{gather}
  
  We use \cite[Lemma 7.2]{dimkatz4} to get a contradiction.  One of the five  cases given there must appear.

  In case (a) of \cite[Lemma 7.2]{dimkatz4}   we have $a^k\in \sigma^{ss}$,   $\phi_-(a^j) =\phi(a^k)+1 $ for some $k<j-1$, and since we already have that $a^q\in \sigma^{ss}$  and $\phi(a^{q-1})<\phi(a^{q})$ for  $q \leq p$, it follows that $\phi(a^{q})<\phi(a^{k})$ for small enough $q$, hence   $\hom^1(a^{j},a^q)=0$ for small enough $q$, which contradicts \eqref{nonvanishing3}. 
  
  Case (b) of \cite[Lemma 7.2]{dimkatz4}  ensures  $\phi_-(a^j) =\phi(a^k) $ for some $k>j$. Using \eqref{semi-stability of b ineq5} and \cite[Remark 2.1 (c)]{dimkatz4} we get  $\phi(a^k)=\arg_{(t,t+1)}(Z(a^k))$, hence   by \eqref{semi-stability of b ineq5} we get $\arg_{(t,t+1)}(Z(a^k))<\arg_{(t,t+1)}(Z(a^j))$, which contradicts  \eqref{noncolinear ab1}.
  
  In case (c) of \cite[Lemma 7.2]{dimkatz4}  we have  $\phi_-(a^j) = \phi(b^k)+1$  for some $k<j$, and since we already have that $b^q\in \sigma^{ss}$  and $\phi(b^{q-1})<\phi(b^{q})$ for  $q \leq p$, it follows that $\phi(b^{q})<\phi(b^{k})$ for small enough $q$, hence   $\hom^1(a^{j},b^q)=0$ for small enough $q$, which contradicts \eqref{nonvanishing4}.

  Case  (d)  of \cite[Lemma 7.2]{dimkatz4} ensures  $\phi_-(a^j) =\phi(b^k) $  for some $ j<k$, and then  \eqref{semi-stability of b ineq5},  \cite[Remark 2.1 (c)]{dimkatz4} ensure   $\arg_{(t,t+1)}(Z(b^k))=\phi(b^k)$, 
  hence by  \eqref{semi-stability of b ineq5} we get $\arg_{(t,t+1)}(Z(b^k))<\arg_{(t,t+1)}(Z(a^j))$, which contradicts \eqref{j<mleqi} and $j<k$.

  In case  (e) of \cite[Lemma 7.2]{dimkatz4} we have   $\phi_-(a^j) =\phi(M)+1=t+1$, and  \eqref{semi-stability of b ineq5}, \eqref{semi-stability of b ineq3} imply  $t+1 < {\rm arg}_{(t,t+1)}(Z(a^j))<t+2$, which contradicts the incidence $Z(a^j)\in Z(\delta)^c_+$.
 
  So far, we proved that $a^j \in \sigma^{ss}$ for all $j\in \ZZ$. 
 
 Suppose that $b^j\not \in \sigma^{ss}$ for some $j\geq p+1$.  Due to \eqref{semi-stability of b ineq3} and \eqref{phipm} we have  $b^j\in\mc P[\phi(b^{p+1}),\phi(b^p)+1] \subset  \mc P(\phi(a^{p}),\phi(a^{p})+1]$, and  \cite[Remark 2.1(c)]{dimkatz4} shows that  $\arg_{(\phi(a^{p}),\phi(a^{p})+1]}(Z(b^j))=\arg_{(t,t+1)}(Z(b^j))$. Combining with \cite[Remark 2.1 (a)]{dimkatz4} we put together: 
 \begin{gather} \label{semi-stability of b ineq4} t<\phi(b^{p+1})\leq \phi_-(b^j) < \arg_{(t,t+1)}(Z(b^j))<\phi_+(b^j)\leq  \phi(b^p)+1 <\phi(b^{p+1})+1. \end{gather}

 We use \cite[Lemma 7.3]{dimkatz4}  to obtain a  contradiction.  Some of the five cases given there must appear. 
 
 Case (a) in \cite[Lemma 7.3]{dimkatz4} ensures $\phi_-(b^j) = \phi(a^k)+1$  for some $k<j-1$, and since we have already proved that $a^q \in \sigma^{ss}$ and $\phi(a^{q-1})<\phi(a^{q-1})$ for $q\in \ZZ$,  it follows that \eqref{semi-stability of b ineq4} implies that $\hom^1(b^{j},a^q)=0$ for small enough $q$, which contradicts \eqref{nonvanishing3}.

 Case  (b) in \cite[Lemma 7.3]{dimkatz4} ensures  $\phi_-(b^j) =\phi(a^k) $  for some $k\geq j$, and then  \eqref{semi-stability of b ineq4}  and  $Z(a^k) \in Z(\delta)^c_+$ 
 imply  $\arg_{(t,t+1)}(Z(a^k))=\phi(a^k)$, hence by  \eqref{semi-stability of b ineq4} we get $\arg_{(t,t+1)}(Z(a^k))<\arg_{(t,t+1)}(Z(b^j))$, which contradicts \eqref{j<mleqi} and  $k\geq j$. 
 
 Case (c) in \cite[Lemma 7.3]{dimkatz4} ensures   $\phi_-(b^j) =\phi(b^k)+1 $ for some $k<j-1$, and since we already have that $b^q\in \sigma^{ss}$  and $\phi(b^{q-1})<\phi(b^{q})$ for  $q \leq p$, it follows that $\hom^1(b^{j},b^q)=0$ for small enough $q$, which contradicts \eqref{nonvanishing6}. 
 
 In case (d) in \cite[Lemma 7.3]{dimkatz4} we have   $\phi_-(b^j) =\phi(b^k) $ for some $k>j$. By $Z(b^k)\in Z(\delta)^c_+$ and \eqref{semi-stability of b ineq4} it follows  that  $\phi(b^k)=\arg_{(t,t+1)}(Z(b^k))$, and then \eqref{semi-stability of b ineq4} gives $\arg_{(t,t+1)}(Z(b^k))<\arg_{(t,t+1)}(Z(b^j))$, which contradicts   \eqref{noncolinear ab1} and $k>j$.

 In case  (e)  using \eqref{semi-stability of b ineq4} we obtain $\phi_-(b^j)= \phi(M')+1 =t+1  < \arg_{(t,t+1)}(Z(b^j))<\phi_+(b^j) < t+2$, which contradicts  the incidence $Z(b^j)\in Z(\delta)^c_+$.

So, we proved that $b^j \in \sigma^{ss}$ for all $j\in \ZZ$. 
 
 We claim that  $\phi(M')= \phi(M)<\phi(a^{j})<\phi(b^{j+1})<\phi(a^{j+1})<\phi(M)+1$  for each $j \in \ZZ$, and then from table \eqref{middle M} follows that   $\sigma \in (a^{j}, M, b^{j+1})\cap (b^j, M', a^j)$ for all $j\in \ZZ$. The desired inequality for $j=p$ is in \eqref{semi-stability of b ineq3}. In \cite[p. 869]{dimkatz4} is shown that $\phi(a^j)=\arg_{(t,t+1)}(Z(a^j))$ and  $\phi(b^j)=\arg_{(t,t+1)}(Z(b^j))$ for each $j<p$, if we show that this holds for   $j\geq  p+1$, then the desired inequalities follow from \eqref{j<mleqi}.   Let $j>p$.  Since  $a^j, b^j \in \sigma^{ss}$ for each  $j$,  then by  \eqref{nonvanishing5}, \eqref{nonvanishing6} we get $\phi(a^{p+1}) \leq \phi(a^j)\leq \phi(a^{p+1})+1$ and $\phi(b^{p+1})\leq \phi(b^j)\leq \phi(b^{p+1})+1$, which combined with  $t<\phi(b^{p+1})<\phi(a^{p+1})<t+1$ and $Z(a^j), Z(b^j)\in Z(\delta)^c_+$ implies that $\phi(a^j),\phi(b^j) \in (t,t+1)$, in particular   $\phi(a^j)=\arg_{(t,t+1)}(Z(a^j))$ and  $\phi(b^j)=\arg_{(t,t+1)}(Z(b^j))$ for each $j\geq  p+1$. 
 \end{proof}
 \begin{lemma} \label{equality of phases oh Ms and interesection} If  $\sigma \in \cap_{j\in \ZZ}  (a^{j}, M, b^{j+1})\cap (b^j, M', a^j)$, then $\phi(M)=\phi(M')$ and \eqref{semi-stability of b ineq}, \eqref{semi-stability of b ineq1} hold for each $p\in \ZZ$. In particular for any $p\in \ZZ$ holds 
 	
 	\begin{gather} \nonumber
 	\cap_{j\in \ZZ}  (a^{j}, M, b^{j+1})\cap (b^j, M', a^j) = \left \{ \sigma \in (a^p, M, b^{p+1}):  
 	\begin{array}{c} \phi\left (a^{p}\right)-1< \phi\left ( M\right)<\phi\left (a^{p}\right)\\   \phi\left (a^{p}\right)-1< \phi\left ( b^{p+1}\right)-1<\phi\left (a^{p}\right) \\ \phi(M)=\arg_{(\phi\left (a^{p}\right)-1,\phi\left (a^{p}\right))}(Z(a^p)-Z(b^{p+1}))
 	\end{array}
  \right \}
 	\end{gather} 
 	\end{lemma}
 	\begin{proof} Let $\sigma \in \cap_{j\in \ZZ}  (a^{j}, M, b^{j+1})\cap (b^j, M', a^j)$. From table \ref{middle M}, and the non-vanishings $\hom(M,b^j)\neq 0$, $\hom^1(b^j, M')\neq 0$ in Corollary \ref{nonvanishings}  it follows that for each $j\in \ZZ$ we have \begin{gather} \phi(M') < \phi(a^j) < \phi (b^{j+1})< \phi(a^{j+1}) < \phi(M')+1 \quad \phi(M) < \phi(a^j) < \phi (b^{j+1})< \phi(a^{j+1}) < \phi(M)+1 \nonumber  \end{gather} From \cite[Corollaries 3.18,Corollary 3.19]{dimkatz4} it follows that there exists $t\in \RR$ such that 
 		\begin{gather} \lim_{j\rightarrow +\infty}\phi(a^j) = t+1, \lim_{j\rightarrow -\infty}\phi(a^j) = t. \end{gather} And therefore $t = \phi(M)=\phi(M')$. The last sentence now follows from Lemma \ref{semi-stability of b}.
 		\end{proof}
 
 Now we pass to the behavior of $C_{1,\sigma}(\mc T)$, $C_{1,\sigma \sigma}(\mc T)$ in $ (a^p,M,b^{p+1})$. 
 
 \footnotetext[1]{for some $x\in \ZZ$}
 \begin{proposition} \label{porposition for middle M} Let $\sigma \in (a^p,M,b^{p+1})$ for some $p\in \ZZ$. Then $a^p,M,b^{p+1}\in \sigma^{ss}$ and $ \phi\left ( M\right) < \phi\left (b^{p+1}\right)$, $\phi(a^p) < \phi(b^{p+1})$, $\phi(a^p)<\phi(M)+1$.

 	{\rm (a)} If   $	\phi(b^{p+1})-1<\phi(M) <\phi(b^{p+1})$, then  $a^{p+1} \in \sigma^{ss}$ and  \begin{gather} \label{arg of equals phase of app1} \phi(a^{p+1})={\rm arg}_{(\phi(b^{p+1}), \phi(b^{p+1})+1]}(Z(b^{p+1})-Z(M)). \end{gather}
 	
 	{\rm(a.1)} if  $\phi(a^p)<\phi(M)$, then  $\sigma\in (a^p,a^{p+1},M) \subset  \mk{T}_a^{st}$. 
 	
 		{\rm(a.2)} if  $\phi(a^p)\geq \phi(M)$, then  $M'\in \sigma^{ss}$ and $\phi(M')={\rm arg}_{(\phi(b^{p+1})-1, \phi(b^{p+1}))}\left (Z(a^p)-Z(b^{p+1})\right )$ and:
 		\begin{gather} \label{middle M a2} \phi(b^{p+1})-1 < \phi(a^{p+1})-1 < \phi(M)\leq \phi(a^{p}) <  \phi(b^{p+1}); \ \ \phi(b^{p+1})-1 <  \phi(M') < \phi(a^{p})
 		\end{gather}

 	{\rm(a.2.1)} if $\phi(M') < \phi(M)$, then $\sigma \in (a^j,a^{j+1}, M)$ for   $j \in \ZZ_{\leq x}$\footnotemark[1] and  $C_{1,\sigma \sigma}(\mc T)=C_{1,\sigma }(\mc T) = \{A\}$,
 	
 		{\rm(a.2.2)}  if $\phi(M') > \phi(M)$, then $\sigma \in (b^j,b^{j+1}, M')$ for   $j \in \ZZ_{\leq x}$\footnotemark[1] and   $C_{1,\sigma \sigma}(\mc T)=C_{1,\sigma }(\mc T) = \{B\}$,
 	
 		{\rm(a.2.3)}  if $\phi(M') = \phi(M)$, then $\sigma \in \bigcap_{j\in \ZZ}  (a^{j}, M, b^{j+1})\cap (b^j, M', a^j)$  and $C_{1,\sigma \sigma}(\mc T)=C_{1,\sigma }(\mc T) = \{A, B\}=C_{1}(\mc T)$ and $C_{0,\sigma \sigma}(\mc T)=C_{0}(\mc T)$.
 		
 	{\rm (b)} If   $\phi(M)+1 <	\phi(b^{p+1})$ and $\phi(M) < \phi(a^p)$, then $\sigma \in (M,b^p,b^{p+1})$ 	and we use table \eqref{table for M b b} with
 	\begin{gather} \label{the phase of bp} \phi(b^p)={\rm arg}_{(\phi(M), \phi(M)+1]}(Z(M)+Z(a^p)), \qquad \phi(M)<\phi(b^p)<\phi(a^p)<\phi(M)+1 \end{gather}

 		{\rm (c)} If   $\phi(M)+1 <	\phi(b^{p+1})$ and $\phi(a^{p})\leq  \phi(M)$, then $C_{1,\sigma}(\mc T) = \emptyset$, $ M' \not \in \sigma^{ss}$, and if $\sigma  \in (A,B,C)$, then $(A,B,C) = (a^p,M,b^{p+1})$.	Furthermore $C_{0,\sigma \sigma}(\mc T) =\{\beta^{p}\}$, when $\phi(a^{p})=  \phi(M)$, whereas for  $\phi(a^{p})<  \phi(M)$, we have  $C_{0,\sigma \sigma}(\mc T) =\emptyset$
 		
 			{\rm (d)} If   $\phi(M)+1 =	\phi(b^{p+1})$ and $\phi(a^{p})<  \phi(M)$, then $C_{1,\sigma}(\mc T) = \emptyset$, $ M' \not \in \sigma^{ss}$, and if $\sigma  \in (A,B,C)$, then $(A,B,C) = (a^p,M,b^{p+1})$.	Furthermore $C_{0,\sigma \sigma}(\mc T) =\{\beta^{p+1}\}$.
 			
 				{\rm (e)} If   $\phi(M)+1 =	\phi(b^{p+1})$ and $\phi(a^{p})>  \phi(M)$, then $\sigma \in (b^j,b^{j+1},M')$ for  $j\in \ZZ_{\leq x}$\footnotemark[1] and  $C_{1,\sigma \sigma}(\mc T)=C_{1,\sigma }(\mc T) = \{B\}$.

	{\rm (f)} If   $\phi(M)+1 =	\phi(b^{p+1})=  \phi(a^p)+1$, then  $a^j,b^j \in \sigma^{ss}$ for all $j\in \ZZ$ and  $\phi(a^j)=\phi(b^j)=t+1$   for each $j\geq p+1$,  $\phi(a^j)=\phi(b^j)=t$ for $j\leq p$. In particular $C_{1,\sigma \sigma}(\mc T)=C_{1,\sigma }(\mc T) = C_{1}(\mc T)$. In this case $M'\in \sigma^{ss}$ and $\phi(M)=\phi(M')$, and in particular $C_{0,\sigma \sigma}(\mc T)=C_{0 }(\mc T)$. In this case,  if $\sigma  \in (A,B,C)$, then $(A,B,C) = (a^p,M,b^{p+1})$.

 \end{proposition}	
 \begin{proof}
 	Looking in table \eqref{middle M} we see that  $a^p,M,b^{p+1}  \in \sigma^{ss}$ and  the inequalities given there hold.
 	
  (a)  \cite[Lemma 7.4  (a)]{dimkatz4} ensures that $a^{p+1} \in \sigma^{ss}$ and 
 		\begin{gather} \label{middle M (a)} \phi(b^{p+1})<\phi(a^{p+1})< \phi(M)+1. \end{gather} Recalling from Remark \ref{remark for delta} that $Z(a^{p+1})= Z(b^{p+1})-Z(M)$ it follows \eqref{arg of equals phase of app1}.
 	
 	(a.1) Now $\phi(a^p)<\phi(M)$ and \cite[Lemma 7.4 (b)]{dimkatz4}  ensure that  $\sigma\in (a^p,a^{p+1},M)$ and hence $\sigma \in \mk{T}_a^{st}$.
 		
 	(a.2) Now combining \eqref{middle M (a)} and \eqref{middle M} we see that hold  the first series of inequalities in  \eqref{middle M a2}, and \cite[Lemma 7.4 (c)]{dimkatz4} ensure that $M'\in \sigma^{ss}$, $\phi(M')={\rm arg}_{(\phi(b^{p+1})-1, \phi(b^{p+1}))}\left  (Z(a^p)-Z(b^{p+1})\right )$ and the second series of inequalities in \eqref{middle M a2}. 
 	
 	(a.2.1) In the \cite[proof of Lemma 7.4 (d)]{dimkatz4} is shown that $\sigma \in  (a^{j-1}, a^j, M)$ for small enough $j$. From \cite[Lemma 5.2.]{dimkatz4} it follows that $\sigma \in  (a^{j-1}, a^j, M)$ and  $\sigma \in (a^{j-1}, a^j, M)$ and $\phi(a^{j})-1 < \phi(a^{j-1})< \phi(a^{j})$ for all small enough  $j\in \ZZ$.  From  Propositions \eqref{proposition for  B in Tb}, \eqref{more careful part two}  we see that $A$ is stable and $B$ is not semistable. 
 	
 	(a.2.2) Now combining with \eqref{middle M a2}  we obtain 
 	$ \phi(b^{p+1})-1 < \phi(a^{p+1})-1 < \phi(M) <\phi(M') < \phi(a^{p}) <  \phi(b^{p+1})
 		$ and therefore all the three conditions  $\phi(a^{p})-1 < \phi(M)< \phi(a^{p})$, $ \phi(a^{p})-1 < \phi(b^{p+1})-1< \phi(a^{p})$, $\phi(M) <\phi(M')$ required in \cite[Lemma 7.5 (d)]{dimkatz4} are satisfied, which ensures that $\sigma \in (b^j,b^{j+1}, M')$ for some $j\in \ZZ$.  Furthermore, since we have 	$ \phi(b^{p+1}) < \phi(M)+1$, then in the  \cite[proof of Lemma 7.5 (d)]{dimkatz4} is actually shown that $\sigma \in (b^j,b^{j+1}, M')$ for all small enough  $j\in \ZZ$.  From \cite[Lemma 5.2.]{dimkatz4} it follows that $\sigma \in  (b^{j-1}, b^j, M')$ and  $\phi(b^{j})-1 < \phi(b^{j-1})< \phi(b^{j})$ for all small enough  $j\in \ZZ$.  From  Propositions \eqref{proposition for  B in Tb}, \eqref{more careful part two}  we see that $B$ is stable and $A$ is not semistable. 
 	
 	(a.2.3) In \eqref{middle M} we have $\phi(a^p)<\phi(b^{p+1})$ and \eqref{middle M a2} together with $\phi(M)=\phi(M')$ give rise to  \eqref{semi-stability of b ineq},    \eqref{semi-stability of b ineq1}.  Now we apply  Lemma    \ref{semi-stability of b} and \eqref{a criterion}. 
 	
 	(b) The given inequalities and \eqref{middle M} ensure that $\phi(a^p)-1 < \phi(M) < \phi(a^p)$ and \cite[Lemma 7.5 (a), (b)]{dimkatz4} shows that $b^p \in \sigma^{ss}$, $\phi(M)<\phi(b^p)<\phi(a^p)<\phi(M+1)$, $\sigma \in (M,b^{p}, b^{p+1}) $. Recalling that in Remark \ref{remark for delta} we have $Z(b^p)=Z(M)+ Z(a^{p})$ and \eqref{the phase of bp} follows.
 	
 	(c) and (d): From the given inequalities follows  $\phi(b^{p+1})>\phi(a^p)+1$ and  Lemma \ref{no semistable} we see that both $A$, $B$ are not semistable. Furthermore, using \cite[Lemmas 7.6,  7.7]{dimkatz4} it follows that  
 	 	 $\sigma \not \in \mk{T}_a^{st} \cup \mk{T}_b^{st}$.	From Lemma \ref{no semistable} (c) it follows that $b^j, a^j\not \in \sigma^{ss}$ for $j\not \in \{p,p+1\}$ and hence $\sigma \not \in (a^j, M, b^{j+1})$ for $j\neq p$. On the other hand the first triangle in \eqref{short filtration 1} for $q=p$ is a non-trivial HN filtration of $M'$ and therefore $M'\not \in \sigma^{ss}$, hence $\sigma \not \in (\_,M',\_)$.  From \eqref{st with T} we deduce that $\sigma \not \in \st(\mc T) \setminus (a^p, M,b^{p+1})$.
 From Subsection \ref{first remarks} it follows that $\alpha^m \not \in C_{0,\sigma \sigma}(\mc T)$ for $m\in \ZZ$. On the other hand, since 	$b^j, a^j\not \in \sigma^{ss}$ for $j\not \in \{p,p+1\}$, it follows that 
 $\beta^p \in C_{0,\sigma \sigma}(\mc T)$ iff $b^p \in \sigma^{ss}$ and $\beta^{p+1} \in C_{0,\sigma \sigma}(\mc T)$ iff $a^{p+1} \in \sigma^{ss}$.
 In case (d) we have $\phi(M)=\phi(b^{p+1})-1$ and we can use Lemma \ref{stable curve} and Corollary \ref{nonvanishings}  to get $a^{p+1}\in \sigma^{ss}$,  $\beta^{p+1} \in C_{0,\sigma \sigma}(\mc T)$. In case (d) we have also $\phi(a^p)<\phi(M)$, and due to the non-vanishings $\hom(M,b^p)\neq 0$, $\hom(b^p, a^p) \neq  0$ in \eqref{nonvanishing1}, \eqref{nonvanishing3} we see that $b^p \not \in \sigma^{ss}$, therefore $\beta^{p} \not \in C_{0,\sigma \sigma}(\mc T)$. In case (c),  $a^{p+1} \not \in \sigma^{ss}$, since we have $\hom^1(a^{p+1},M)\neq 0$, $\hom(b^{p+1}, a^{p+1}) \neq  0$ in \eqref{nonvanishing2}, \eqref{nonvanishing3}. In case (c) with $\phi(a^p)<\phi(M)$ we have also $b^p \not \in \sigma^{ss}$ by the arguments used in case (d). Finally, in case (c) with $\phi(a^p)=\phi(M)$ we use Corollary \ref{stable curve} and \eqref{nonvanishing2} to deduce that $\beta^p \in C_{0,\sigma \sigma}(\mc T)$.

 	(e)  The given inequalities and \eqref{middle M} give rise to $\phi(a^p)<\phi(M)+1=\phi(b^{p+1})< \phi(a^p)+1$. Therefore $\phi(a^p)-1<\phi(M)=\phi(b^{p+1})-1< \phi(a^p)$ and from \cite[Lemma 7.5 (c)]{dimkatz4} we see that $M'\in \sigma^{ss}$ and  $\phi(b^{p+1})-1<\phi(M')=\arg_{(\phi\left (a^{p}\right)-1,\phi\left (a^{p}\right))}(Z(a^p)-Z(b^{p+1}))<\phi(a^{p})$. On the other hand $\phi(M)=\phi(b^{p+1})-1$ implies that $\phi(M')=\arg_{(\phi\left (a^{p}\right)-1,\phi\left (a^{p}\right))}(Z(a^p)-Z(b^{p+1}))> \phi(M)$, hence we apply \cite[Lemma 7.5 (d)]{dimkatz4}. Furthermore, we are given now that $\phi(b^{p+1})=\phi(M)+1$. Looking in the proof of \cite[Lemma 7.5 (d)]{dimkatz4} we see that    $\sigma \in (b^j,b^{j+1},M')$ for  small enough $j\in \ZZ$ and now the same  arguments as in the proof of 
 	(a.2.2) imply that $C_{1,\sigma \sigma}(\mc T)=C_{1,\sigma }(\mc T) = \{B\}$.

 	(f) Let us denote $t=\phi(M)$. In this case we have $t=\phi(M)=\phi(a^p)=\phi(b^{p+1})-1$ and by the  triangles in \eqref{short filtration 1}  it follows that $M', b^p, a^{p+1} \in \sigma^{ss}$ and $\phi(b^p)=\phi(M)=\phi(M')=\phi(a^p)=t$,  $\phi(a^{p+1})=t+1$. Using induction and \eqref{short filtration 1} one sees that  $a^j,b^j \in \sigma^{ss}$,  $\phi(a^j)=\phi(b^j)=t+1$   for each $j\geq p+1$,  and  $a^j,b^j \in \sigma^{ss}$,  $\phi(a^j)=\phi(b^j)=t$ for $j\leq p$.  Due to the last sentence and table \eqref{middle M} it follows that $\sigma \not \in (\_,M',\_)$ and $\sigma \not \in (a^j,M,b^{j+1})$ for $j\neq p$. We showed that $C_{1,\sigma \sigma}(\mc T)=\{A,B\}$ and from Corollary \ref{in TaTb only one} it follows that $\sigma \not \in \mk{T}_a^{st} \cup  \mk{T}_b^{st}$.  From \eqref{st with T} we deduce that $\sigma \not \in \st(\mc T) \setminus (a^p, M,b^{p+1})$.
 		\end{proof}

 	\begin{corollary}\label{phases of MMprime2}  For $\sigma \in (\_, M, \_)\cup  (\_, M', \_)$ we have $C_{1,\sigma \sigma}(\mc T)= \{A,B\}$ iff $M,M'\in \sigma^{ss}$ and $\phi(M)=\phi(M')$. 
 	\end{corollary}
 	\begin{proof} Let $\sigma \in (\_, M, \_)$.  Then  $\sigma \in (a^p, M, b^{p+1})$ for some $p\in \ZZ$.   This $\sigma$ must satisfy exactly  one of the cases (a.1), (a.2.1), (a.2.2), (a.2.3), (b), (c), (d), (e),(f) in Proposition \ref{porposition for middle M}.  Taking into account tables \eqref{table for a a M}, \eqref{table for M' a a},  \eqref{table for M b b}, \eqref{table for b b M'} we see that $C_{1,\sigma \sigma}(\mc T)= \{A,B\}$ if and only if $\sigma$ is in case (a.2.3) or in case (f). On the other hand using Corollary \ref{phases of MMprime1} we see that $\sigma$ satisfies case (a.2.3) or  case (f) if and only if $M,M' \in \sigma^{ss}$ and $\phi(M)=\phi(M')$. If $\sigma \in (\_, M', \_)$ we use the already proved case and Remark \ref{action1}.  \end{proof}

	Now we can give also proof of  Proposition \ref{sigma outside Tab}.
\textit{Proof of Proposition} \ref{sigma outside Tab}
		If $\sigma \not \in \mk{T}_a^{st} \cup \mk{T}_b^{st} $  due to \eqref{st with T} we have  either $\sigma \in (a^p, M,b^{p+1})$ or $\sigma \in (b^p, M',a^{p})$ for some $p\in \ZZ$. If  $\sigma \in (a^p, M,b^{p+1})$ we go through all possible cases  (a.1), (a.2.1), (a.2.2), (a.2.3), (b), (c), (d), (e),(f) in Proposition \ref{porposition for middle M} and from $\sigma \not \in \mk{T}_a^{st} \cup \mk{T}_b^{st} $ we must be in some of the cases (a.2.3),  (c), (d), (f). 
		
		In cases (a.2.3) and (f) we have 	$\{M, M'\} \subset \sigma^{ss}$,   $\phi(M)=\phi(M')$ and $C_{1,\sigma \sigma}(\mc T)=\{A,B\}$, $C_{0,\sigma \sigma}(\mc T)=C_{0}(\mc T)$. Furthermore using also Lemma \ref{equality of phases oh Ms and interesection} we see that $\sigma$ satisfies case (a.2.3) if and only if  $\sigma \in \bigcap_{p\in \ZZ} (a^p, M, b^{p+1}) \cap (b^p, M', a^{p})$. 
		
		In cases (c), (d) we have  	$\{M, M'\} \not \subset \sigma^{ss}$ and $C_{1,\sigma \sigma}(\mc T)=C_{1,\sigma }(\mc T)=\emptyset$.

	 The case  $\sigma \in (b^p, M',a^{p})$ follows from the already proved case and Remark \ref{action1}. 
		
	The fact that the unions are disjoint follows from the properties $\sigma \not \in \st(\mc T)\setminus (a^p, M,b^{p+1}) $  obtained in (c), (d), (f) of Proposition \ref{porposition for middle M}. 
	
	The behavior of $C_{0,\sigma \sigma}(\mc T)$ specified in the first three sentences after \eqref{second disjoint union} is the same as explained in (c), (d) of Proposition \ref{porposition for middle M}.
	
		If  $\sigma \in (b^p, M',a^{p})$ we use \eqref{zeta on alpha beta} and  Remark \ref{action1}. \qed \\
	
\section{The set of stabilities, where $\abs{C_{0,\sigma \sigma}(\mc T)}<\infty$ is open and  dense in $\st(\mc T)$} \label{dense open}

Note first that due to \eqref{finitenes of genus zero stables} we have that $\abs{C_{0,\sigma \sigma}(\mc T)}=\infty$ implies $C_{1,\sigma}(\mc T)=\{A,B\}$. On the other hand from the already proved tables  \eqref{table for M' a a}, \eqref{table for M b b}, \eqref{table for a a M}, \eqref{table for b b M'}, \eqref{table for b a b}, Proposition \ref{sigma outside Tab}, and \eqref{st with T} we see that $C_{1,\sigma}(\mc T)=\{A,B\}$ implies  $\abs{C_{0,\sigma \sigma}(\mc T)}=\infty$. Therefore, we see that 
\begin{gather} \label{inffinitely genus zero stable}
C_{1,\sigma}(\mc T)=\{A,B\} \iff \abs{C_{0,\sigma \sigma}(\mc T)}=\infty.
\end{gather}
Note that from the results in previous sections follows Proposition \ref{MeqMprime} (see the proof  in the end of Subsection \ref{details on}), i. e. $C_{1,\sigma \sigma}(\mc T)=\{A,B\} \iff M,M' \in \sigma^{ss} \ \mbox{and} \ \phi(M)=\phi(M')$.  From  \eqref{inffinitely genus zero stable} tables \eqref{table for a a M}, \eqref{table for M' a a}, \eqref{table for b b M'}, \eqref{table for M b b}, \eqref{table for b a b}, Proposition \ref{sigma outside Tab}, \eqref{st with T} and Subsection \ref{homeomoeprhism} it follows that the set $\{\sigma \in \st(\mc T):  \abs{C_{0,\sigma \sigma}(\mc T)}=\infty\}$ is union of the following three subsets in \eqref{closed subset 1}, \eqref{closed subset 2}, \eqref{closed subset 3}
\begin{gather} \label{closed subset 1} \{\sigma \in  \st(\mc T): C_{1,\sigma \sigma}(\mc T)=\{A,B\}\}  \end{gather}   \begin{gather} 
 \{\sigma \in  \st(\mc T): B \in C_{1,\sigma \sigma}(\mc T) \ \mbox{and} \ M,M' \in \sigma^{ss} \ \mbox{and} \ \phi(M)+1=\phi(M') \}  \setminus \nonumber  \\[-2mm] \label{closed subset 2} \\ \nonumber  \left (\bigcup_{p\in \ZZ} \{\sigma \in (M,b^p, b^{p+1}): \phi(b^{p+1})<  \phi(b^{p})+1 \} \cup  \{\sigma \in (b^p, b^{p+1}, M'): \phi(b^{p+1})<  \phi(b^{p})+1 \} \right ) \nonumber 
\end{gather}

\begin{gather}
\{\sigma \in  \st(\mc T): A \in C_{1,\sigma \sigma}(\mc T) \ \mbox{and} \ M,M' \in \sigma^{ss} \ \mbox{and} \ \phi(M')+1=\phi(M) \}  \setminus \nonumber  \\ \label{closed subset 3} \\ \nonumber  \left (\bigcup_{p\in \ZZ} \{\sigma \in (M',a^p, a^{p+1}): \phi(a^{p+1})<  \phi(a^{p})+1 \} \cup  \{\sigma \in (a^p, a^{p+1}, M): \phi(a^{p+1})<  \phi(a^{p})+1 \} \right ) \nonumber 
\end{gather}
From Corollary \ref{closed subset} we know that $\{\sigma \in  \st(\mc T): B \in C_{1,\sigma \sigma}(\mc T) \}$ is a closed subset in $\st(\mc T)$.  From \cite[p. 342]{Bridgeland} is known that for any object $X$ in $\mc T$ the set $\{\sigma \in \st(\mc T): X\in \sigma^{ss}\}$ is closed subset and the function assigning $\phi_{\sigma}(X)$ to any $\sigma$ in this subset, is continuous in this subset. Therefore  $\{\sigma \in  \st(\mc T): B \in C_{1,\sigma \sigma}(\mc T) \ \mbox{and} \ M,M' \in \sigma^{ss} \ \mbox{and} \ \phi(M)+1=\phi(M') \}$ is a closed subset in $\st(\mc T)$. 

\begin{remark} \label{remark for a homeomorphism} From Remark \ref{remark for fe}  (b) and table \eqref{left right M} we know that  the map \eqref{assignment} gives a homeomorphism  from $(M,b^p,b^{p+1})\subset \st(\mc T)$ to  $\RR_{>0}^3 \times \left \{  \begin{array}{c} y_0 - y_1 < 0 \\  y_0 - y_2 < -1 \\ y_1 - y_2 < 0\end{array}\right\}$. It follows that  the subset $\{\sigma \in (M,b^p, b^{p+1}): \phi(b^{p+1})<  \phi(b^{p})+1 \}$ is the inverse image of the open  subset  $\RR_{>0}^3 \times \left \{  \begin{array}{c} y_0 - y_1 < 0 \\  y_0 - y_2 < -1 \\ -1<y_1 - y_2 < 0\end{array}\right\}$ with respect to this  homeomorphism, therefore $\{\sigma \in (M,b^p, b^{p+1}): \phi(b^{p+1})<  \phi(b^{p})+1 \}$ is an open subset in $\st(\mc T)$.
	\end{remark}
	 By similar arguments one shows that  $\{\sigma \in (b^p, b^{p+1}, M'): \phi(b^{p+1})<  \phi(b^{p})+1 \}$ is also  an open subset in $\st(\mc T)$. Therefore \eqref{closed subset 2} is a closed subset in $\st (\mc T)$. Since \eqref{closed subset 3} is obtained from \eqref{closed subset 2} via the auto-equivalence  $\zeta$ in \eqref{useful autoequivalence} it follows that \eqref{closed subset 3} is closed subset as well. From Corollary \ref{closed subset} it follows that \eqref{closed subset 1} is a closed subset as well (see also Proposition \ref{MeqMprime}). We proved that \eqref{closed subset 1}, \eqref{closed subset 2}, \eqref{closed subset 3} are closed subset,\textbf{ therefore}  \textbf{ $\{\sigma \in \st(\mc T):  \abs{C_{0,\sigma \sigma}(\mc T)}=\infty\}$  is a closed subset in $\st(\mc T)$, hence  $\{\sigma \in \st(\mc T):  \abs{C_{0,\sigma \sigma}(\mc T)}<\infty\}$  is an open subset in $\st(\mc T)$.}
 \begin{corollary}  \label{infinite genus zero} The set $\{\sigma \in \st(\mc T) : \abs{C_{0,\sigma \sigma}(\mc T)} < \infty \}$ is  dense subset in $\st(\mc T)$. 
\end{corollary}
\begin{proof} We already proved that the set in question is open subset, therefore it is enough to show that for any $\sigma \in \st(\mc T)$ such that $\abs{C_{0,\sigma \sigma}(\mc T)}=\infty$ and for any neighborhood $U$ of $\sigma$ there exists $\sigma'\in U$ such that  $\abs{C_{0,\sigma' \sigma'}(\mc T)}<\infty$. So let $\sigma \in \st(\mc T)$ be such that $\abs{C_{0,\sigma \sigma}(\mc T)}=\infty$ and let $\sigma \in U$ for some open subset $U$. 
We use Corollary \ref{infinite genus zero part 1}, Proposition \ref{sigma outside Tab},  Lemma \ref{equality of phases oh Ms and interesection} and deduce that the set  $\{\sigma \in \st(\mc T) : \abs{C_{0,\sigma \sigma}(\mc T)} = \infty \}$ is the following disjoint  union (in the last summand $q\in \ZZ$):

 $\bigcup_{p\in \ZZ} \{\sigma \in (M,b^p, b^{p+1}) : \phi(b^{p+1})=\phi(b^p)+1, \phi(b^p)=\phi(M)+1 \} \cup $
 
 $\bigcup_{p\in \ZZ} \{ \sigma \in (M,a^p, a^{p+1}): \phi(a^{p+1})=\phi(a^p)+1, \phi(a^p) =\phi(M')+1 \} \cup$

 $\bigcup_{p\in \ZZ} \{ \sigma \in (b^p, b^{p+1}, M'): \phi(b^{p+1})=\phi(b^p)+1=\phi(M') \}  \cup$
 
 $\bigcup_{p\in \ZZ} \{\sigma \in (a^p, a^{p+1}, M):\phi(a^{p+1})=\phi(a^p)+1=\phi(M) \} \cup$
 \begin{gather} \bigcup_{p\in \ZZ} \{\sigma \in (a^p, M, b^{p+1}):  \phi(M)+1 =	\phi(b^{p+1})=  \phi(a^p)+1 \} \nonumber \\ \cup  \bigcup_{p\in \ZZ} \{\sigma \in (b^p, M', a^{p}):  \phi(M')+1 =	\phi(a^{p})=  \phi(b^p)+1 \} \nonumber \\ \cup \left \{ \sigma \in (a^q, M, b^{q+1}):  
 \begin{array}{c} \phi\left (a^{q}\right)-1< \phi\left ( M\right)<\phi\left (a^{q}\right)\\   \phi\left (a^{q}\right)-1< \phi\left ( b^{q+1}\right)-1<\phi\left (a^{q}\right) \\ \phi(M)=\arg_{(\phi\left (a^{q}\right)-1,\phi\left (a^{q}\right))}(Z(a^q)-Z(b^{q+1}))
 \end{array}
 \right \}. \nonumber 
 \end{gather}

Assume  that  $\sigma \in (M,b^p, b^{p+1})$  and  $\phi(b^{p+1})=\phi(b^p)+1, \phi(b^p)=\phi(M)+1$ (the arguments in the other cases are analogous).  

The subset  $\{\sigma \in (M,b^p, b^{p+1}) : \phi(b^{p+1})=\phi(b^p)+1, \phi(b^p)=\phi(M)+1 \}$   is the inverse image of the subset  $\RR_{>0}^3 \times \left \{  \begin{array}{c} y_1= y_0+1 \\  y_2= y_1+1\end{array}\right\}$ in  $\RR_{>0}^3 \times \left \{  \begin{array}{c} y_0 - y_1 < 0 \\  y_0 - y_2 < -1 \\ y_1 - y_2 < 0 \end{array}\right\}$  with respect to the homeomorphism discussed in Remark \ref{remark for a homeomorphism}. Denote this homeomorphism by $f$.\footnote{$f$ can be viewed as a chart in a $C^{\infty} $ manifold.} In particular, $f(\sigma)=(r_0, r_1, r_2, y_0, y_0+1, y_0+2)$ for some $r_i > 0$ and $y_0 \in \RR$. Since $(M,b^p, b^{p+1})$ is an open subset in $\st(\mc T)$, it follows that $f(U\cap (M,b^p, b^{p+1}))$ is an open subset in $\RR^6$ containing  $f(\sigma)$, therefore  for some $\varepsilon>0$ we have open subset $V \ni \sigma$, such that $V\subset U \cap (M,b^p, b^{p+1})$ and such that $f(V)=(r_0-\varepsilon,r_0+\varepsilon)\times(r_1-\varepsilon,r_1+\varepsilon)\times(r_2-\varepsilon,r_2+\varepsilon)\times(y_0-\varepsilon,y_0+\varepsilon)\times (y_0+1-\varepsilon,y_0+1+\varepsilon)\times (y_0+2-\varepsilon,y_0+2+\varepsilon)$. Hence $f(v) \not \in \RR_{>0}^3 \times \left \{  \begin{array}{c} y_1= y_0+1 \\  y_2= y_1+1\end{array}\right\}$ for many $v \in V$, hence by table \eqref{table for M b b} we have $\abs{C_{0,v v}(\mc T)} < \infty$ for many $v \in V$.
\end{proof}

\section{Walls and chambers for $C_{1,\sigma\sigma}$} \label{walls and chambers} 
\label{section walls and chambers}
In tables \eqref{table for b b M'}, \eqref{table for a a M}, \eqref{table for M b b}, \eqref{table for M' a a} and \eqref{table for b a b} and Proposition \ref{sigma outside Tab}, we have summed up the local behaviour of $C_{1,\sigma \sigma}$. However, to understand how many walls there are, how they attach to one another and how the chambers with given subset of $C_{1,\sigma\sigma }$ are separated by them, we need to combine these results with knowing how two subset $(A,B,C),(A',B',C')\subset \textnormal{Stab}(\mathcal{T})$  intersect or equivalently how the corresponding charts given by the homoemorphism in Remark \ref{remark for fe} glue. We begin by labeling some of the subsets appearing in the aforementioned tables and proposition. 
\begin{definition}
The following subsets of $\textnormal{Stab}(\mathcal{T})$ and their connected unions will be called walls
\begin{align*}
    \mathcal{W}(a^m,a^{m+1},M) &= \{\sigma\in (a^m,a^{m+1},M): \phi(a^{m+1}) = \phi(a^m)+1\}\,,\\
     \mathcal{W}(b^m,b^{m+1},M') &= \{\sigma\in (b^m,b^{m+1},M'): \phi(b^{m+1}) = \phi(b^m)+1\}\,,\\
      \mathcal{W}(M',a^m,a^{m+1}) &= \{\sigma\in (M',a^m,a^{m+1}): \phi(a^{m+1}) = \phi(a^m)+1\}\,,\\
       \mathcal{W}(M,b^m,b^{m+1}) &= \{\sigma\in (M,b^m,b^{m+1}): \phi(b^{m+1}) = \phi(b^m)+1\}\,,\\
       \mathcal{W}^0(A,B) &= \bigcap_{p\in\mathbbm{Z}}\big((a^p,M,b^{p+1})\cap (b^p,M',a^p)\big)\,,\\
       \mathcal{G}(a^m,M,b^{m+1}) &= \{\sigma\in (a^m,M,b^{m+1}): \phi(a^m) + 1 = \phi(M) + 1 = \phi(b^{m+1}) \}\,,\\
       \mathcal{G}(b^m,M',a^{m}) &= \{\sigma\in (b^m,M,a^{m}): \phi(b^m) + 1 = \phi(M') + 1 = \phi(a^{m}) \}\,,\\
\end{align*}
for all $m\in \mathbbm{Z}$.\\

If $X$ is a submanifold and $Y$ a submanifold with a boundary in $\textnormal{Stab}(\mathcal{T})$ both of real codimension 1 and disjoint. Let $Z\subset X$ be a submanifold of real codimension 2 contained in the boundary $\partial Y$, then we say that $X$ \textit{attaches to} $Y$ \textit{in} $Z$ if $Z=\bar{X}\cap Y$. 
\end{definition}

Notice that except for the last 2 walls which are of real codimension 2 submanifolds the other walls are submanifolds of real codimension 1. For $\mathcal{W}^0(A,B)$ this follows from Lemma \ref{equality of phases oh Ms and interesection}. For our purposes the following definition of a chamber will be sufficient:

\begin{definition}
A \textit{closed chamber with $N\subset C_{1}(\mathcal{T})$ stable} is a closure of a non-empty maximal open connected subset $\mathcal{C}h^0$ of $\textnormal{Stab}(\mathcal{T})$ with $C_{1,\sigma\sigma} =  N$ for all $\sigma\in \mathcal{C}h^0$.
\end{definition}
We now state the main result summing up the global behaviour of $C_{1,\sigma \sigma}$. 

\begin{proposition}
\label{proposition walls and chambers}
There is a unique closed chamber $\mathcal{C}h(A)$ with $A$ stable and a unique closed chamber $\mathcal{C}h(B)$ with $B$ stable . The complement of $\mathcal{C}h(A)\cup\mathcal{C}h(B)$ is an open subset of $\textnormal{Stab}(\mathcal{T})$ with $C_{1,\sigma \sigma} = \emptyset$ whenever $\sigma$ lies in it. It is given by the disjoint union of open subsets $\mathcal{O}_m$ (see \eqref{nullchambers}) going over all $m\in \mathbbm{Z}$. In particular, the closures $\overline{\mathcal{O}_m}$ are the closed chambers with $\emptyset$ stable.\\

There is a unique connected set with $|C_{1,\sigma\sigma}| = 2$ given by
$$
\mathcal{C}h(A)\cap\mathcal{C}h(B) = \overline{\mathcal{W}^0(A,B)} = \mathcal{W}^0(A,B)\cup\bigcup_{p}\big(\mathcal{G}(a^p,M,b^{p+1})\cup\mathcal{G}(b^p,M',a^p)\big)\,.
$$

The boundary of $\mathcal{C}h(A)$ is the union of the walls:
\begin{equation}
\label{boundary}
  \bigcup_{p} \big(\mathcal{W}(a^p,a^{p+1},M)\cup \mathcal{W}(M',a^p,a^{p+1})\big)\cup \overline{\mathcal{W}^0(A,B)}\,, 
\end{equation}

where $\mathcal{W}(a^p,a^{p+1},M)\cup \mathcal{W}(M',a^p,a^{p+1})$ is a connected sub-manifold of real codimension 1. The union of the sets $\mathcal{W}(a^p,a^{p+1},M)\cup \mathcal{W}(M',a^p,a^{p+1})$ is disjoint, and moreover it is disjoint from $\overline{\mathcal{W}^0(A,B)}$. However, the intersection $\mathcal{W}(a^p,a^{p+1},M)\cap \mathcal{W}(M',a^p,a^{p+1})$ attaches to $\overline{\mathcal{W}^0(A,B)}$ in $\mathcal{G}(a^m,M,b^{m+1})$ and $\mathcal{G}(b^m,M', a^m)$. Therefore, \eqref{boundary} is a connected real codimension 1 wall. 

The boundary of $\mathcal{C}h(B)$ is the union of the walls:
\begin{equation}
\label{boundary1}
  \bigcup_{p} \big(\mathcal{W}(b^p,b^{p+1},M')\cup \mathcal{W}(M,b^p,b^{p+1})\big)\cup \overline{\mathcal{W}^0(A,B)}\,, 
\end{equation}

where $\mathcal{W}(b^m,b^{m+1},M')\cup \mathcal{W}(M,b^m,a^{m+1})$ is a connected sub-manifold of real codimension 1. The union of the sets $\mathcal{W}(b^m,b^{m+1},M')\cup \mathcal{W}(M,b^m,b^{m+1})$ is disjoint, and moreover it is disjoint from $\overline{\mathcal{W}^0(A,B)}$. However, the intersection $\mathcal{W}(b^m,b^{m+1},M')\cap \mathcal{W}(M,b^m,a^{m+1})$ attaches to $\overline{\mathcal{W}^0(A,B)}$ in $\mathcal{G}(a^m,M,b^{m+1})$ and $\mathcal{G}(b^m,M', a^m)$. Therefore, \eqref{boundary1} is a connected real codimension 1 wall. 
\end{proposition}
\begin{proof}
Recall that $\mathfrak{T}^{\textnormal{st}}_a\cap \mathfrak{T}^{\textnormal{st}}_b = \emptyset$ and there exists an auto-equivalence $\xi\in\textnormal{Aut}(\textnormal{Stab}(\mathcal{T}))$ given by \eqref{useful autoequivalence} which identifies the two disjoint sets. We will therefore first understand the picture in $\mathfrak{T}^{\textnormal{st}}_a$ and then glue $(a^p,M,b^{p+1})$ and $(b^p,M',a^p)$ (where the last two sets are again identified using the auto-equivalence). \\

First, let us describe the situation in $\bigcup_p(a^p,a^{p+1}, M)$. Using \cite[Lemma 5.4]{dimkatz4}, we can describe the intersection $(a^p,a^{p+1}, M)\cap(a^j,a^{j+1},M)$ for $p>j$ as 
$$ \{\sigma\in (a^p,a^{p+1},M) : \phi(a^{p+1})<\phi(a^p)+1 \}\,.$$
By looking at Table \eqref{table for a a M}, we see that this corresponds to the interior of the unique chamber with $A$ stable in $(a^p,a^{p+1},M)$. It implies that this set is contained in the interior of the chamber with $A$ stable in $(a^j,a^{j+1},M)$. The union
$$\bigcup_p \{\sigma\in (a^p,a^{p+1},M) : \phi(a^{p+1}<\phi(a^p)+1 \}\,.$$
is connected and therefore contained in a single chamber which we denote by $\mathcal{C}h(A,1)$. Moreover, the walls $\mathcal{W}(a^p,a^{p+1},M)$ and $\mathcal{W}(a^j,a^{j+1},M)$ are disjoint whenever $p\neq j$. \\

We do the same for $\bigcup_p(M',a^p,a^{p+1})$.  By \cite[Lemma 5.6]{dimkatz4}, we see that 
$$
(M',a^p,a^{p+1})\cap (M',a^j,a^{j+1}) = \{\sigma\in (M',a^p,a^{p+1}) : \phi(a^{p+1})<\phi(a^p)+1 \}
$$
for $p<j$. By Table \eqref{table for M' a a} this coincides with the interior of the unique chamber with $A$ stable in $(M',a^p,a^{p+1})$. The union 
$$\bigcup_p \{\sigma\in (M',a^p,a^{p+1}) : \phi(a^{p+1})<\phi(a^p)+1 \}\,.$$
is connected and thus contained in a single chamber which we denote by $\mathcal{C}h(A,2)$. The walls $\mathcal{W}(M',a^p,a^{p+1})$ and $\mathcal{W}(M',a^j,a^{j+1})$ are disjoint whenever $p\neq j$. \\

Let us now glue the sets $(-,-,M)$ and $(M',-,-)$. The intersection $(a^m,a^{m+1},M)\cap (M',-,-)$ is contained by \cite[Lemma 6.3]{dimkatz4} in the following subset of $(a^m,a^{m+1},M)$:
\begin{equation}
\label{intersection a a M M a a}
    \{\sigma\in (a^m,a^{m+1},M):\phi(a^{m+1})<\phi(a^m)+1 \textnormal{ or } \phi(M)<\phi(a^{m+1})\}\,.
\end{equation}

From the proof of \cite[Lemma 6.3]{dimkatz4}, we see that if the condition $\phi(a^{m+1})<\phi(a^m)+1$ is true, then there exists a $j\gg m$  with $\sigma \in (M',a^j,a^{j+1})$. This shows that the interior of the unique chamber with $A$ stable in $(a^m,a^{m+1},M)$ intersects with the interior of the chamber $\mathcal{C}h(A,2)$ non-trivially and therefore $\mathcal{C}h(A,1) = \mathcal{C}h(A,2)=:\mathcal{C}h(A)$. Notice that we have already considered all subsets of $\textnormal{Stab}(\mathcal{T})$ with $C_{1,\sigma\sigma} = \{A\}$ as either walls of $\mathcal{C}h(A)$ or as contained in its interior. We can conclude that $\mathcal{C}h(A)$ is the unique closed chamber with $A$ stable.  \\

If in \eqref{intersection a a M M a a} the condition $\phi(M)<\phi(a^{m+1})$ is true and $\phi(a^{m+1})=\phi(a^m)+1$, then by the proof of \cite[Lemma 6.3]{dimkatz4}, $\sigma$ lies in $(M',a^{m},a^{m+1})$ and in particular in $\mathcal{W}(M',a^{m},a^{m+1})$. Therefore, we see that 
$$
\mathcal{W}(a^{m},a^{m+1},M) \cap \mathcal{W}(M',a^{m},a^{m+1}) = \{\sigma\in \mathcal{W}(a^{m},a^{m+1},M): \phi(M)<\phi(a^{m+1})\}\,,
$$
which is an open subset of $\mathcal{W}(a^m,a^{m+1},M)$. Moreover, as we see from \eqref{intersection a a M M a a} that $\mathcal{W}(a^m,a^{m+1},M)\cap (M',-,-) = \mathcal{W}(a^{m},a^{m+1},M) \cap \mathcal{W}(M',a^{m},a^{m+1})$, we conclude that $$\mathcal{W}(a^m,a^{m+1},M)\cap \mathcal{W}(M',a^j,a^{j+1}) = \emptyset\,,$$ whenever $j\neq m$. \\

Let us now attach $(a^p,b^{p+1},a^{p+1})$ which will help us understand chambers with $\emptyset$ stable of which they are open subsets by Table \eqref{table for b a b}. From \cite[Lemma 6.1 and 6.2]{dimkatz4}, we know that they are pairwise disjoint, and that 
$$
(a^m,b^{m+1},a^{m+1})\cap (a^{l},a^{l+1},M)=(a^m,b^{m+1},a^{m+1})\cap (M',a^{l},a^{l+1}) = \emptyset\,,
$$
whenever $m\neq l$, while if $m=l$ the intersections are non-trivial. Looking at the proof of \cite[Lemma 6.3]{dimkatz4} we see that if $\sigma\in(a^m,a^{m+1},M)$ and $\phi(M)<\phi(a^{m+1}), \phi(a^{m+1})>\phi(a^m)+1$, then $\sigma\in (a^m,b^{m+1},a^{m+1})$. Therefore, we can say that 
$$
(a^{m},a^{m+1},M) \cap (a^{m},b^{m+1},a^{m+1})=\{\sigma\in (a^m,a^{m+1},M):\phi(M)<\phi(a^{m+1}),\phi(a^{m+1})>\phi(a^m)+1\}\,.
$$
Additionally, as $(a^{m},a^{m+1},M)\cap(M',a^{m},a^{m+1})$ has to contain a neighborhood of $\mathcal{W}(a^m,a^{m+1},M)\cap\mathcal{W}(M',a^{m},a^{m+1})$, it must intersect with $(a^{m},a^{m+1},M) \cap (a^{m},b^{m+1},a^{m+1})$ in a nonempty open set. If a set of the form $(X,Y,Z)\subset \textnormal{Stab}(\mathcal{T})$ for $(X,Y,Z)$ full exceptional triple has a unique closed chamber with $\emptyset$ stable, we denote by $(X,Y,Z)_{\emptyset}$ its interior. We can conclude from the above that 
$$
(a^m,a^{m+1},M)_{\emptyset}\cup (M',a^{m},a^{m+1})_{\emptyset}\cup (a^{m},b^{m+1},a^{m+1})_{\emptyset}
$$
are connected open subsets of $\textnormal{Stab}(\mathcal{T})$ for all $m\in\mathbb{Z}$ and they are pairwise disjoint. \\

We now need to consider what happens in $(a^p,M,b^{p+1})$. This is partly handled by Proposition \ref{porposition for middle M}. However, we do some extra work to obtain the qualitative picture given in Figure \eqref{a M b}, where we use $\phi(a^m)-\phi(M)$ and $\phi(M) + 1 - \phi(b^{m+1})$ as our coordinates under the assumption that they are the only variables that are allowed to vary. The reader should note that this picture is meant as a helpful tool for the discussion below and is not meant to depict it precisely. From now on, when we talk about a sub-case without mentioning its proposition, we mean a sub-case of Proposition \ref{porposition for middle M}.\\

In case (a.2.1), the stability condition $\sigma$ lies in the interior of $\mathcal{C}h(A)$ as it lies  in the intersection of multiple $(a^j,a^{j+1},M')$ and it connects to the case (a.1) which lies in $(a^p,a^{p+1},M)$: this can be seen, because if $\phi(a^m)-\phi(M)$ decreases while keeping $\phi(M)+1 - \phi(b^{m+1})$ constant, then from the formula for $\phi(M')$ in (a.2), we see that $\phi(M)-\phi(M')$ increases. Therefore by doing so, we eventually land in (a.2.1) on our way to $\phi(a^p) = \phi(M)$. 
Now as (a.1) connects to both (a.2.1) and (d) with $C_{1,\sigma\sigma}=\emptyset$, it needs to contain an open subset of $\mathcal{W}(a^p,a^{p+1},M)$. We show that $\mathcal{W}(a^p,a^{p+1},M)\cap (a^p,M,b^{p+1})$ is contained in $\mathcal{W}(a^p,a^{p+1},M)\cap \mathcal{W}(M',a^p,a^{p+1})$ and that it attaches to the closure of $\mathcal{W}^0(A,B)$ in $(a^p,M,b^{p+1})$ through $\mathcal{G}(a^p,M,b^{p+1})$. \\

\sloppy From Proposition \ref{porposition for middle M}, we see that in (a.1) $\sigma\in (a^p,a^{p+1},M)$ and $\phi(a^{p+1})>\phi(b^{p+1})>\phi(M)$. Using again the proof of \cite[Lemma 6.3]{dimkatz4}, we see that if additionally $\sigma\in \mathcal{W}(M',a^p,a^{p+1})$, then $\sigma\in \mathcal{W}(a^p,a^{p+1},M)\cap \mathcal{W}(M',a^p,a^{p+1})$. Let us now assume that we are in case (a.1) and that $\phi(M)+1<\phi(b^{m+1})+\epsilon_1$ and $\phi(a^m)>\phi(M) + \epsilon_2$ for some small $\epsilon_1,\epsilon_2>0$. This corresponds to taking a small neighborhood of $\mathcal{G}(a^p,a^{p+1},M)$ and intersecting it with (a.1). It follows then using Table \eqref{left right M} that
$$
 \phi(a^m)+1-\epsilon_1<\phi(M)+1-\epsilon_1<\phi(b^{p+1})<\phi(a^{p+1})<\phi(M)+1<\phi(a^m)+1+\epsilon_2\,.
$$
and we see that we can chose $\phi(a^{m+1}) = \phi(a^m)+1$ in this, to obtain the intersection with $\mathcal{W}(M', a^p,a^{p+1})$. Therefore, each neighborhood of $\mathcal{G}(a^p,M,b^{p+1})$ contains a point of $\mathcal{W}(M', a^p,a^{p+1})\cap \mathcal{W}( a^p,a^{p+1},M)$.

Similar arguments show that (a.2.2) connects through (e) to (b) and that the open subset of $\mathcal{W}(M,b^p,b^{p+1})$ attaches to the closure of $\mathcal{W}^0(A,B)$ in $(a^p,M,b^{p+1})$ through $\mathcal{G}(a^p,M,b^{p+1})$. We additionally want to show that $\mathcal{W}(M,b^p,b^{p+1})\cap (a^p,M,b^{p+1})$ is contained in $\mathcal{W}(b^p,b^{p+1},M')\cap \mathcal{W}(M,b^p,b^{p+1})$. Let us first show that in case (b) with $\phi(b^{p+1})=\phi(b^p)+1$, we have $M'\in\sigma^{ss}$. This follows from $\phi(b^p)<\phi(a^p)<\phi(M) + 1$ and looking at the Table \eqref{table for M b b}. To show that the inequalities for $(b^p,b^{p+1},M')$ from Table \eqref{left right M} hold, we only need to show $\phi(M')+1>\phi(b^{p+1})$, but this follows from looking at the proof of (b.2) in Proposition \ref{more careful study}.\\

Let us now conclude the final results of the proposition. As the closure of $\mathcal{W}^0(A,B)$ in $(a^p,M,b^{p+1})$ is given by $\mathcal{W}^0(A,B)\cup\mathcal{G}(a^p,M,b^{p+1)}$, we see that 
$$
\overline{\mathcal{W}^0(A,B)} =\mathcal{W}^0(A,B)\cup\bigcup_{p}\big(\mathcal{G}(a^p,M,b^{p+1})\cup\mathcal{G}(b^p,M',a^p)\big)\,.
$$
Moreover, this lies in the intersection of $\mathcal{C}h(A)$ and $\mathcal{C}h(B)$. As it is the only set, where both $A$ and $B$ are stable, we see that it is equal to it. Using $\xi$, we also know that $\mathcal{W}(a^p,a^{p+1},M)\cap\mathcal{W}(M',a^p,a^{p+1})$ and $\mathcal{W}(b^p,b^{p+1},M')\cap\mathcal{W}(M,b^p,b^{p+1})$ attach to $\overline{\mathcal{W}^0(A,B)}$ in $\mathcal{G}(a^p,M,b^{p+1})\cup \mathcal{G}(b^p,M',a^p)$. \\

Finally, the sets 
\begin{align*}
    \mathcal{O}_m &= (a^m,a^{m+1},M)_{\emptyset}\cup (M',a^{m},a^{m+1})_{\emptyset}\cup (a^{m},b^{m+1},a^{m+1})_{\emptyset}\cup(b^m,b^{m+1},M')_{\emptyset}\\
    &\cup (M',b^{m},b^{m+1})_{\emptyset}\cup (b^{m},a^{m},b^{m+1})_{\emptyset}
    \cup (a^m,M,b^{m+1})_{\emptyset}\cup(b^m,M',a^m)_{\emptyset}
    \numberthis
    \label{nullchambers}
\end{align*}
are open connected subsets with $C_{1,\sigma\sigma} = \emptyset$ and they are pairwise disjoint as the cases $(c)$ and $(d)$ in Proposition \ref{porposition for middle M} are only contained in $(a^p,M,b^{p+1})$. Therefore, the set of closed chambers with $\emptyset$ stable is given by $\{\overline{\mathcal{O}_m}\}_{m\in\mathbb{Z}}$. This completes the proof of the statements in the proposition.

\begin{figure}[h]
		\includegraphics[scale=0.7]{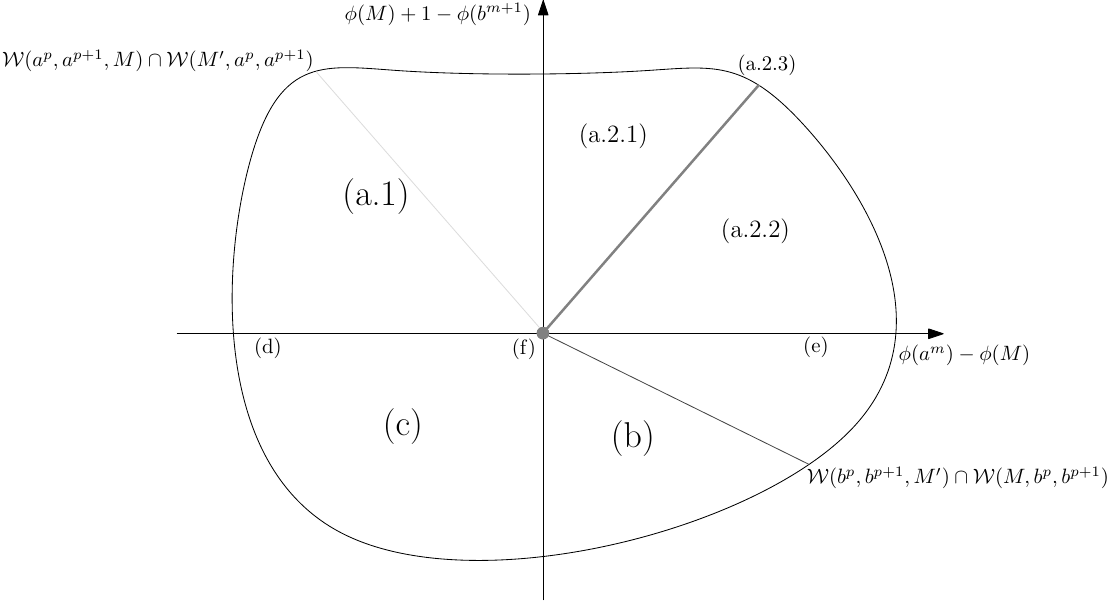}
		\centering
		\caption{$(a^p,M,b^{p+1})$ with subsets as in Proposition \ref{porposition for middle M} and its walls}\label{a M b}
	\end{figure}

\end{proof}
 \begin{remark}
     Notice that to go from the chamber $\mathcal{C}h(A)$ into the chamber $\mathcal{C}h(B)$, one needs to either pass through the wall $\overline{\mathcal{W}^0(A,B)}$ or to go through $\mathcal{O}_m$ for some $m$ and pass through the walls $\mathcal{W}(a^m,a^{m+1},M)\cup \mathcal{W}(M',a^m,a^{m+1})$ and $\mathcal{W}(b^m,b^{m+1},M')\cup \mathcal{W}(M,b^m,b^{m+1})$. 
 \end{remark}

\appendix
\section{Details on $C_{l,\sigma \sigma}(\mc T)$, $C_{l,\sigma}(\mc T)$ as $\sigma$ varies in $\st(\mc T)$ and $l\geq -1$} \label{details on} 
\label{sectab}

In Subsection \ref{subsection for Tab} we determine  $C_{l,\sigma}(\mc T)$,  $C_{l,\sigma \sigma}(\mc T)$ for    $\sigma \in \mk{T}_{a}^{st}  \cup \mk{T}_{b}^{st}$ and $l\geq -1$. The results are contained in the tables  \eqref{table for b b M'}, \eqref{table for a a M}, \eqref{table for M b b}, \eqref{table for M' a a} and  \eqref{table for b a b}. The sets specified in the  first column in each table  are  non-empty.

The following table collects results from  Propositions \ref{proposition for B in Tb}, \ref{more careful part two},  Corollary \ref{coro for b b M'}, Lemma \ref{lemma for semistability of M}: 
\begin{gather} \label{table for b b M'} \begin{array}{| c   | c  | c| c |}   \hline  
\begin{array}{c} \sigma \in  (b^p,b^{p+1},M') \\ v=\phi(b^{p})+1,   q=\phi(M')\\ 
t = \arg_{(\phi(b^{p})-1,\phi(b^{p}))}(Z(b^p)-Z(b^{p+1})) \end{array} &    C_{0,\sigma \sigma}(\mc T) = ?  \ \mbox{and} \ M \in \sigma^{ss} ? & C_{1,\sigma }(\mc T) &  C_{1,\sigma \sigma}(\mc T) \\
\hline
\phi(b^{p+1})> v\ \mbox{and} \ q> \phi(b^{p+1})  &  C_{0,\sigma \sigma}(\mc T) =\emptyset \ \mbox{and} \ M \not \in \sigma^{ss} & \emptyset& \emptyset \\
\hline
\phi(b^{p+1})> v \ \mbox{and} \ q\leq \phi(b^{p+1})  &  \begin{array}{c} \{\alpha^p \} \subset  C_{0,\sigma \sigma}(\mc T) \subset \{\alpha^p, \beta^p \} \\ \hline  \beta^p \in C_{0,\sigma \sigma}(\mc T) \iff  M \in \sigma^{ss} \iff  \\   \arg_{[q, q+1)}(Z(b^{p+1})+ Z(M')) \leq v \end{array}  & \emptyset& \emptyset\\
\hline
\phi(b^{p+1})= v\ \mbox{and} \ q> v& \begin{array}{c} C_{0,\sigma \sigma}(\mc T)=\emptyset \ \ \mbox{and}  \ \ M\not \in \sigma^{ss} \end{array}  & \{B\} & \{B\}\\   \hline 
\phi(b^{p+1})=v\ \mbox{and} \ q < v & \begin{array}{c} C_{0,\sigma \sigma}(\mc T)=\{\alpha^p, \beta^p \}  \\ M\in \sigma^{ss} \ \mbox{and} \  q<\phi(M)+1  \textnormal{ \footnotemark}   \end{array} &\{B\} &\{B\} \\   \hline 
\phi(b^{p+1})= v\ \mbox{and} \ q = v & \begin{array}{c}C_{0,\sigma \sigma}(\mc T)= \{\alpha^j, \beta^j : j \geq p \} \\  M  \in \sigma^{ss} \ \mbox{and} \   q=\phi(M)+1   \end{array} &\{A,B\} &\{B\} \\   \hline  
\begin{array}{c} \phi(b^{p+1})<v \ \mbox{and} \   q-1 \geq t \end{array}&  C_{0,\sigma \sigma}(\mc T) = \emptyset, M\in \sigma^{ss} \iff q-1 = t  &\{B\} &\{B\} \\
\hline 
\begin{array}{c} \phi(b^{p+1})<v \ \mbox{and} \   q - 1 < t \end{array}&\begin{array}{c} C_{0,\sigma \sigma}(\mc T) = \{\alpha^j, \beta^j: N\leq j \leq u\}, \\  \mbox{where $N\leq u$ are determined } \\ \mbox{ in  Proposition \ref{more careful part two} (a.2)} \end{array}& \{B\} &\{B\}  \\
\hline 
\end{array}   
\end{gather} \footnotetext{In details, the second row of this box says that for  $\sigma \in  (b^p,b^{p+1},M')$, $\phi(b^{p+1})= \phi(b^{p})+1$, $\phi(M')< \phi(b^{p})+1$ hold  both $M\in \sigma^{ss}$ and $\phi(M')<\phi(M)+1$. We will use this  for the proof of \eqref{closed subset 2} in Section \ref{dense open}.}
The last  table and Remark \ref{action1} imply:
\begin{gather} \label{table for a a M} \begin{array}{| c   | c  |c| c |}   \hline  
\begin{array}{c} \sigma \in  (a^p,a^{p+1},M) \\ v=\phi(a^{p})+1,   q=\phi(M) \\ 
t= \arg_{(\phi(a^{p})-1,\phi(a^{p}))}(Z(a^p)-Z(a^{p+1})) \end{array}
 &   C_{0,\sigma \sigma}(\mc T) = ? \ \mbox{and} \ M' \in \sigma^{ss} ? & C_{1,\sigma }(\mc T) &  C_{1,\sigma \sigma}(\mc T)\\
\hline
\phi(a^{p+1})> v\ \mbox{and} \ q> \phi(a^{p+1})  & C_{0,\sigma \sigma}(\mc T) = \emptyset \ \mbox{and} \ M'\not \in \sigma^{ss} & \emptyset  &  \emptyset \\
\hline
\phi(a^{p+1})> v \ \mbox{and} \ q\leq \phi(a^{p+1})  &  \begin{array}{c} \{\beta^{p+1} \} \subset  C_{0,\sigma \sigma}(\mc T)  \subset \{\alpha^p, \beta^{p+1} \}    \\ \hline \alpha^p \in C_{0,\sigma \sigma}(\mc T) \iff M' \in \sigma^{ss} \iff   \\  \arg_{[q, q+1)}(Z(a^{p+1})+ Z(M)) \leq v \end{array} & \emptyset & \emptyset \\
\hline
\phi(a^{p+1})= v\ \mbox{and} \ q> v&  \begin{array}{c} C_{0,\sigma \sigma}(\mc T) =\emptyset \ \ \mbox{and} \ \ M' \not \in \sigma^{ss}   \end{array} & \{A\} & \{A\} \\   \hline 
\phi(a^{p+1})= v\ \mbox{and} \ q < v & \begin{array}{c} C_{0,\sigma \sigma}(\mc T) = \{\beta^{p+1}, \alpha^p \} \\ M'\in \sigma^{ss} \ \mbox{and} \  q<\phi(M')+1 \end{array} & \{A\} & \{A\} \\   \hline 
\phi(a^{p+1})= v\ \mbox{and} \ q = v & \begin{array}{c} C_{0,\sigma \sigma}(\mc T) = \{\beta^{j+1}, \alpha^j : j \geq p \} \\ M'\in \sigma^{ss} \ \mbox{and} \  q=\phi(M')+1  \end{array} & \{A, B \} & \{A\} \\   \hline  
\begin{array}{c} \phi(a^{p+1})<v \ \mbox{and} \   q-1 \geq t \end{array}&  C_{0,\sigma \sigma}(\mc T) = \emptyset, M'\in \sigma^{ss} \iff q-1=t & \{A\} & \{A\} \\
\hline 
\begin{array}{c} \phi(a^{p+1})<v \ \mbox{and} \ q - 1 < t \end{array}&\begin{array}{c} C_{0,\sigma \sigma}(\mc T) = \{\beta^{j+1}, \alpha^j: N\leq j \leq u\}, \\  \mbox{where $N\leq u$ are determined in} \\ \mbox{Proposition \ref{more careful part two} (a.2) by replacing}  \\  \mbox{$b$ with $a$ and exchanging  $M'$ and $M$}\end{array} & \{A\} & \{A\}  \\
\hline 
\end{array}   
\end{gather}
The following table collects results from  Propositions \ref{proposition for B in Tb}, \ref{more careful study},   Corollary \ref{coro for M b b}, Lemma \ref{lemma for semistability of M'}: 
\begin{gather} \label{table for M b b} \begin{array}{| c   | c  | c  | c  |}   \hline  
\begin{array}{c}\sigma \in  (M,b^p,b^{p+1}) \\   v = \phi(b^{p+1}) -1,  s=\phi(M)+1, \\ t=\arg_{(\phi(b^{p})-1,\phi(b^{p}))}(Z(b^p)-Z(b^{p+1}))  \end{array}  &    C_{0,\sigma \sigma}(\mc T) = ?  \ \mbox{and} \ M' \in \sigma^{ss} ?	&   C_{1,\sigma}(\mc T) &  C_{1,\sigma \sigma}(\mc T)  \\
\hline
v> \phi(b^{p})\ \mbox{and} \ s< \phi(b^p)  &  C_{0,\sigma \sigma}(\mc T) =\emptyset \ \mbox{and} \ M' \not \in \sigma^{ss}  & \emptyset & \emptyset \\
\hline
v> \phi(b^{p}) \ \mbox{and} \ s\geq \phi(b^p)  &  \begin{array}{c} \{\beta^p \} \subset C_{0,\sigma \sigma}(\mc T) \subset \{\alpha^p, \beta^{p}\} \\  \hline \alpha^p \in  C_{0,\sigma \sigma}(\mc T) \iff M'  \in \sigma^{ss} \iff  \\ v \leq   \arg_{(s-1, s]}(Z(b^p)- Z(M)) \end{array} &  \emptyset & \emptyset \\ 
\hline
v= \phi(b^{p})\ \mbox{and} \ s< \phi(b^{p})&  \begin{array}{c} C_{0,\sigma \sigma}(\mc T) = \emptyset \ \ \mbox{and} \ \  M' \not  \in \sigma^{ss} \end{array} & \{B\} & \{B\}  \\   \hline 
v= \phi(b^{p})\ \mbox{and} \ s > \phi(b^{p}) &  \begin{array}{c} C_{0,\sigma \sigma}(\mc T) = \{\alpha^p, \beta^p \}\\ M' \in \sigma^{ss} \ \mbox{and} \  \phi(M')< s\end{array} & \{B\} & \{B\} \\   \hline 
v= \phi(b^{p})\ \mbox{and} \ s = \phi(b^{p}) & \begin{array}{c} C_{0,\sigma \sigma}(\mc T) = \{\alpha^j, \beta^j : j \leq p \} \\ M' \in \sigma^{ss} \ \mbox{and} \    \phi(M')= s\end{array}  & \{A, B\} &\{B\} \\   \hline  
\begin{array}{c} v<\phi(b^{p}) \ \mbox{and} \  s \leq t \end{array}&  C_{0,\sigma \sigma}(\mc T) = \emptyset, M'  \in \sigma^{ss} \iff s=t & \{B\} &\{B\} \\
\hline 
\begin{array}{c} v<\phi(b^{p}) \ \mbox{and} \   s > t \end{array}&\begin{array}{c} C_{0,\sigma \sigma}(\mc T) = \{\alpha^j, \beta^j: N\leq j \leq u\}, \\  \mbox{where $N\leq u$ are determined} \\ \mbox{in Proposition \ref{more careful study} (a.2)} \end{array} & \{B\} & \{B\}\\
\hline 
\end{array}   
\end{gather}
This   table and Remark \ref{action1} imply:
\begin{gather} \label{table for M' a a} \begin{array}{| c   | c  | c | c |}   \hline  
\begin{array}{c}\sigma \in  (M',a^p,a^{p+1})\\   v = \phi(a^{p+1}) -1,  s=\phi(M')+1, \\ t=\arg_{(\phi(a^{p})-1,\phi(a^{p}))}(Z(a^p)-Z(a^{p+1})) \end{array}      &    C_{0,\sigma \sigma}(\mc T) = ?  \ \mbox{and} \ M \in \sigma^{ss} ?  & C_{1,\sigma }(\mc T) & C_{1,\sigma \sigma}(\mc T)  \\
\hline
v> \phi(a^{p})\ \mbox{and} \ s< \phi(a^p)  &  C_{0,\sigma \sigma}(\mc T)= \emptyset \ \mbox{and} \ M \not \in \sigma^{ss} & \emptyset & \emptyset  \\
\hline
v> \phi(a^{p}) \ \mbox{and} \ s\geq \phi(a^p)  &  \begin{array}{c} \{\alpha^p \} \subset C_{0,\sigma \sigma}(\mc T) \subset \{\alpha^p, \beta^{p+1}\} \\ \hline \beta^{p+1} \in C_{0,\sigma \sigma}(\mc T) \iff M  \in \sigma^{ss}\\  v \leq   \arg_{(s-1, s]}(Z(a^p)- Z(M')) \end{array}& \emptyset & \emptyset  \\
\hline
v= \phi(a^{p})\ \mbox{and} \ s< \phi(a^{p})& \begin{array}{c}C_{0,\sigma \sigma}(\mc T)= \emptyset \ \ \mbox{and} \ \  M \not  \in \sigma^{ss} \end{array} & \{A\} & \{A\}\\   \hline 
v= \phi(a^{p})\ \mbox{and} \ s > \phi(a^{p}) & \begin{array}{c}C_{0,\sigma \sigma}(\mc T)= \{\beta^{p+1}, \alpha^p \} \\ M \in \sigma^{ss} \ \mbox{and} \  \phi(M)< s\end{array} & \{A\} & \{A\} \\   \hline 
v= \phi(a^{p})\ \mbox{and} \ s = \phi(a^{p}) & \begin{array}{c}C_{0,\sigma \sigma}(\mc T)= \{\beta^{j+1}, \alpha^j : j \leq p \} \\ M \in \sigma^{ss} \ \mbox{and} \ \phi(M)= s\end{array} & \{A, B\} & \{A\} \\   \hline  
\begin{array}{c} v<\phi(a^{p}) \ \mbox{and} \  s \leq t \end{array}& C_{0,\sigma \sigma}(\mc T)= \emptyset, M  \in \sigma^{ss} \iff s=t  & \{A\} & \{A\} \\
\hline 
\begin{array}{c} v<\phi(a^{p}) \ \mbox{and} \   s > t \end{array}&\begin{array}{c}C_{0,\sigma \sigma}(\mc T)= \{\beta^{j+1}, \alpha^j: N\leq j \leq u\}, \\  \mbox{where $N\leq u$ are determined in}  \\ \mbox{ Proposition \ref{more careful study} (a.2) by replacing}\\  \mbox{ $b$ with $a$ and  exchanging  $M'$ and $M$}  \end{array}  & \{A\} & \{A\} \\
\hline 
\end{array}   
\end{gather}
The following table collects results from Corollary \ref{genus 0 in bab} and  Propositions \ref{proposition for B in Tb}, \ref{proposition for  M,Mp in Tb}
\begin{gather} \label{table for b a b} \begin{array}{| c   | c  | c |  }   \hline  \mbox{ }
&   X= C_{0,\sigma \sigma}(\mc T) = ? \ \mbox{and} \ M,M' \in \sigma^{ss} ? & C_{1,\sigma}(\mc T) \\
\hline
\sigma \in (b^p,a^p,b^{p+1}) &  \begin{array}{c} X\subset \{\alpha^p,\beta^p\} \\ M' \in \sigma^{ss} \iff \alpha^p \in X \iff \phi(b^{p+1})-1 \leq \phi(a^{p})< \phi(b^{p+1}) \\M \in \sigma^{ss} \iff  \beta^p \in X \iff \phi(a^{p})-1 \leq \phi(b^{p})<\phi(a^{p})\end{array} &\emptyset \\
\hline 
\sigma \in (a^p,b^{p+1},a^{p+1}) &  \begin{array}{c} X\subset \{\beta^{p+1},\alpha^p\} \\M \in \sigma^{ss} \iff   \beta^{p+1} \in X \iff \phi(a^{p+1})-1 \leq \phi(b^{p+1})< \phi(a^{p+1}) \\M' \in \sigma^{ss} \iff  \alpha^p \in X \iff \phi(b^{p+1})-1 \leq \phi(a^{p})<\phi(b^{p+1})\end{array} & \emptyset  \\
\hline 
\end{array}
\end{gather}

In Subsection \ref{section for middle MMprime} we determine  the behavior of $C_{l,\sigma}(\mc T)$,  $C_{l,\sigma \sigma}(\mc T)$ in   $\st(\mc T)\setminus ( \mk{T}_a^{st}\cup  \mk{T}_b^{st})$. The results there imply Proposition \ref{sigma outside Tab} (the proof is given after Corollary \ref{phases of MMprime2}):

\begin{proposition}\label{sigma outside Tab}  Let $\sigma \not \in \mk{T}_a^{st}\cup  \mk{T}_b^{st}$. 
	
	If 	$\{M, M'\} \subset \sigma^{ss}$, then   $\phi(M)=\phi(M')$ and $C_{1,\sigma \sigma}(\mc T)=C_{1,\sigma }(\mc T)=\{A,B\}$, $C_{0,\sigma \sigma}(\mc T)=C_{0}(\mc T)$.
	
	If 	$\{M, M'\} \not \subset \sigma^{ss}$, then  $C_{1,\sigma \sigma}(\mc T)=C_{1,\sigma }(\mc T)=\emptyset$, $\abs{C_{0,\sigma \sigma}(\mc T)}\leq 1$. 
	
	The set $\{\sigma: \sigma \not \in \mk{T}_a^{st} \cup \mk{T}_b^{st} \ \mbox{and} \  \{M, M'\} \subset \sigma^{ss} \}$ is disjoint union (any two summands are non-empty and disjoint):
	\begin{gather} \bigcup_{p\in \ZZ} \{\sigma \in (a^p, M, b^{p+1}):  \phi(M)+1 =	\phi(b^{p+1})=  \phi(a^p)+1 \} \nonumber \\ \label{union with emptiness}  \cup  \bigcup_{p\in \ZZ} \{\sigma \in (b^p, M', a^{p}):  \phi(M')+1 =	\phi(a^{p})=  \phi(b^p)+1 \}  \\ \cup 
	\bigcap_{p\in \ZZ} (a^p, M, b^{p+1}) \cap (b^p, M', a^{p}). \nonumber 
	\end{gather}
	The set $\{\sigma: \sigma \not \in \mk{T}_a^{st} \cup \mk{T}_b^{st} \ \mbox{and} \  \{M, M'\} \not \subset \sigma^{ss} \}$ is disjoint union  (any two summands are disjoint and non-empty):
	\begin{gather} \bigcup_{p\in \ZZ} \left \{\sigma \in (a^p, M, b^{p+1}): \begin{array}{c}  \phi(M)+1 <	\phi(b^{p+1}) \  \mbox{and} \ \phi(a^{p})\leq  \phi(M)  \\ \mbox{or} \ \  (\phi(M)+1 =	\phi(b^{p+1}) \ \ \mbox{and}  \  \ \phi(a^{p}) <  \phi(M) ) \end{array} \right  \} \nonumber \\[-2mm] \label{second disjoint union}  \\[-2mm] \nonumber \cup   \bigcup_{p\in \ZZ} \left \{\sigma \in (b^p, M', a^{p}): \begin{array}{c}  \phi(M')+1 <	\phi(a^{p}) \  \mbox{and} \ \phi(b^{p})\leq  \phi(M')  \\ \mbox{or} \ \ (\phi(M')+1 =	\phi(a^{p}) \ \ \mbox{and}  \  \ \phi(b^{p}) <  \phi(M') )  \end{array} \right  \}. \nonumber
	\end{gather}

		For $\sigma \in (a^p,M,b^{p+1})$  we have: 
		{\rm (a) } $C_{0,\sigma \sigma}(\mc T)=\{\beta^{p}\}$, when $\phi(a^p)=\phi(M)$, $\phi(M)+1<\phi(b^{p+1})$;  	{\rm (b) } $C_{0,\sigma \sigma}(\mc T)=\{\beta^{p+1}\}$, when    $\phi(a^p)<\phi(M)$, $\phi(M)+1=\phi(b^{p+1})$; 	{\rm (c) }   $C_{0,\sigma \sigma}(\mc T)=\emptyset$, when $\phi(a^p)<\phi(M)$, $\phi(M)+1<\phi(b^{p+1})$.

			For $\sigma \in (b^p,M',a^{p})$  we have: 	{\rm (a) }  $C_{0,\sigma \sigma}(\mc T)=\{\alpha^{p-1}\}$, when $\phi(b^p)=\phi(M')$, $\phi(M')+1<\phi(a^{p})$;  	{\rm (b) } $C_{0,\sigma \sigma}(\mc T)=\{\alpha^{p}\}$, when $\phi(b^p)<\phi(M')$, $\phi(M')+1=\phi(a^{p})$;  {\rm (c) } $C_{0,\sigma \sigma}(\mc T)=\emptyset$, when $\phi(b^p)<\phi(M')$, $\phi(M')+1<\phi(a^{p})$.
	
\end{proposition}
Let us now discuss the proofs of the main statements which appeared in the introduction. Most of them follow straightforwardly from tables  \eqref{table for b b M'}, \eqref{table for a a M}, \eqref{table for M b b}, \eqref{table for M' a a},  \eqref{table for b a b} and Proposition \ref{sigma outside Tab}, however in some cases more work is required and it was carried out in the paper. 

\begin{proof}[Proof of Proposition \ref{positive answer}]
 Proposition \ref{positive answer} (a) is obtained by looking at all the cases in the tables and  Proposition \ref{sigma outside Tab}, recalling that  all summands in the unions \eqref{union with emptiness}, \eqref{second disjoint union} are non-empty and that  stability conditions in each row of the listed tables exist.  Proposition \ref{positive answer} (b) follows again by  looking at all the cases in the tables and  Proposition \ref{sigma outside Tab}. To prove that all of the subsets in the list are taken we need to use    Proposition \ref{arbitrary even integer} for  	$\{\alpha^j, \beta^j: i\leq j \leq k \}$ and then Remark \ref{action1} for $\{\alpha^j, \beta^{j+1}: i\leq j \leq k \}$. 
\end{proof}

\begin{proof}[Proof of the first sentence in Proposition \ref{some equal subsets}.]  We first prove  that the three sets are equal.  
 	
 	 The first equality  is in \eqref{first equality of}. If $\sigma \in \mk{T}_a^{st} \cup \mk{T}_b^{st} $  we see  in tables \eqref{table for a a M}, \eqref{table for M' a a}, \eqref{table for b b M'}, \eqref{table for M b b}, \eqref{table for b a b} that  $\abs{C_{1,\sigma}(\mc T)}=0$ iff   $\abs{C_{1,\sigma \sigma}(\mc T)}=0$ . Otherwise 
 	 we fall in Proposition \ref{sigma outside Tab}, where $C_{1,\sigma \sigma}(\mc T) =C_{1,\sigma }(\mc T)$.
 	
  We prove that the   set on the RHS  of the proved equality  is open. To that end,  we note that it is  the complement of $\{\sigma \in \st(\mc T): A \in C_{1,\sigma\sigma}(\mc T)  \} \cup \{\sigma \in \st(\mc T): B \in C_{1,\sigma \sigma}(\mc T)  \}$.  And we use Corollary \ref{closed subset} to deduce that the latter union is a closed subset. 
  
  Looking at tables \eqref{table for M b b}, \eqref{table for b b M'}, etc we see that there are non-empty open subsets, where $C_{1,\sigma \sigma}(\mc T) \not = \emptyset$ (see also  Remark \ref{remark for a homeomorphism}), hence  the set in question is not dense. 
  \end{proof}

\begin{proof}[Proof of Proposition \ref{no semistable1}.]  In tables \eqref{table for a a M}, \eqref{table for M' a a}, \eqref{table for b b M'}, \eqref{table for M b b} we see that there are $\sigma$ such that $C_{1,\sigma}(\mc T)=\{A,B\}$, $C_{1,\sigma \sigma}(\mc T)=\{A\}$ and there are $\sigma$ such that $C_{1,\sigma}(\mc T)=\{A,B\}$, $C_{1,\sigma \sigma}(\mc T)=\{B\}$. 
\end{proof}

\begin{proof}[Proof of Proposition \ref{MeqMprime}]  The first equivalence follows from    \eqref{st with T} and Corollaries \ref{in TaTb only one}, \ref{phases of MMprime1},  \ref{phases of MMprime2}. The second equivalence follows from  \eqref{st with T}, tables \eqref{table for a a M}, \eqref{table for M' a a}, \eqref{table for b b M'}, \eqref{table for M b b}, \eqref{table for b a b} and Proposition \ref{sigma outside Tab}. The last equivalence is in \eqref{a criterion}.
\end{proof}

\begin{proof}[Proof of Theorem \ref{main theorem}] This relies also on the tables  and  Proposition \ref{sigma outside Tab}, however, as for the proof of  Proposition \ref{MeqMprime}  it requires  additional work. which is carried out  in  Section \ref{dense open}. 
\end{proof}

\end{document}